\documentclass{article}
\usepackage{graphicx} 

\usepackage{layout} 
\usepackage[top=3cm, bottom=3cm, left=2cm, right=2cm]{geometry} 
\usepackage[english]{babel}
\usepackage[utf8]{inputenc}
\usepackage[T1]{fontenc}
\usepackage{amsmath}
\usepackage{amssymb}
\usepackage{mathrsfs}
\usepackage{amsthm}
\usepackage{enumerate} 
\usepackage{color, colortbl} 
\usepackage{soul} 
\usepackage{fancyhdr, fancybox, fancyvrb,fancyref} 
\usepackage{tcolorbox}
\usepackage[compact]{titlesec}
\usepackage{graphicx} 
\usepackage[tikz]{bclogo}
\usepackage{enumitem}
\usepackage{comment}
\usepackage{mathtools}
\usepackage{esint}
\usepackage{faktor}
\usepackage[colorlinks=true,urlcolor=blue,linkcolor=blue]{hyperref}
\usepackage[capitalize]{cleveref}

\newtheorem{theorem}{Theorem}[section]
\newtheorem{proposition}[theorem]{Proposition}

\newtheorem{lemma}[theorem]{Lemma}
\newtheorem{remark}[theorem]{Remark}
\newtheorem{defi}[theorem]{Definition}
\newtheorem{claim}[theorem]{Claim}

\newcommand{\scal}[2]{\left\langle #1,#2 \right\rangle}
\newcommand{\g}{\nabla}

\newcommand{\lap}{\Delta}
\newcommand{\dr}{\partial}
\newcommand{\vol}{\mathrm{vol}}

\newcommand{\tr}{\mathrm{tr}}
\newcommand{\II}{\mathrm{I\!I}}

\newcommand{\Riem}{\mathrm{Riem}}
\newcommand{\Ric}{\mathrm{Ric}}

\newcommand{\proj}{\mathrm{proj}}

\newcommand{\Hess}{\mathrm{Hess}}
\newcommand{\Id}{\mathrm{Id}}

\newcommand{\loc}{\mathrm{loc}}

\newcommand{\Diff}{\mathrm{Diff}}
\newcommand{\BiLip}{\mathrm{Bi\text{-}Lip}}
\newcommand{\Imm}{\mathrm{Imm}}

\newcommand{\R}{\mathbb{R}}

\newcommand{\N}{\mathbb{N}}

\newcommand{\s}{\mathbb{S}}

\newcommand{\B}{\mathbb{B}}

\newcommand{\Hr}{\mathcal{H}}

\newcommand{\Er}{\mathcal{E}}

\newcommand{\Ur}{\mathcal{U}}
\newcommand{\Vr}{\mathcal{V}}

\newcommand{\I}{\mathcal{I}}

\newcommand{\Pc}{\mathscr{P}}

\newcommand{\vp}{\varphi}
\newcommand{\eps}{\varepsilon}

\newcommand{\geu}{g_{\mathrm{eucl}}}

\newcommand{\vu}{\vec{u}}

\newcommand{\vq}{\vec{q}}

\newcommand{\vn}{\vec{n}}
\newcommand{\vtau}{\vec{\tau}}

\newcommand{\vPhi}{\vec{\Phi}}
\newcommand{\vPsi}{\vec{\Psi}}
\newcommand{\vphi}{\vec{\phi}}
\newcommand{\vpsi}{\vec{\psi}}
\newcommand{\vII}{\vec{\II}}
\newcommand{\vH}{\vec{H}}

\newcommand{\vw}{\vec{w}}

\newcommand{\vtheta}{\vec{\theta}}

\usepackage{authblk}
\title{Construction of harmonic coordinates for weak immersions}
\date{\today}
\author[1]{Dorian Martino}
\author[2]{Tristan Rivière}
\affil[1]{Department of Mathematics, ETH Zürich, 101 Rämistrasse, 8092 Zürich, Switzerland. Email: dorian.martino@math.ethz.ch}
\affil[2]{Department of Mathematics, ETH Zürich, 101 Rämistrasse, 8092 Zürich, Switzerland. Email: tristan.riviere@math.ethz.ch}

\begin{document}
	
	\maketitle
	
	\begin{abstract}
		We prove that any weak immersion in the critical Sobolev space $W^{\frac{n}{2}+1,2}(\R^n;\R^d)$ in even dimension $n\geq 4$, has global harmonic coordinates if its second fundamental form is small in the Sobolev space $W^{\frac{n}{2}-1,2}(\R^n;\R^d)$. This is a generalization to arbitrary even dimension $n\ge 4$ of a famous result of Müller--Sverak \cite{muller1995} for $n=2$. The existence of such coordinates is a key tool used by the authors in \cite{MarRiv20252} for the analysis of scale-invariant Lagrangians of immersions, such as the Graham--Reichert functional. From a purely intrinsic perspective, the proof of the main result leads to a general local existence theorem of harmonic coordinates for general metrics with Riemann tensor in $L^p$ for any $p>n/2$ in any dimension $n\geq 3$.
	\end{abstract}

	
	\section{Introduction}

	The study of analytic problems surrounding immersions (such as regularity or compactness of minimal or CMC surfaces) are fundamental questions that first require the construction of local coordinates. Those coordinates should be controlled in a manner that allows the use of elliptic regularity. For instance, if $\vPhi\colon \B^n\to\R^d$ is an immersion with induced metric $g_{\vPhi}$, then its mean curvature vector $\vH_{\vPhi}$ satisfies $n\, \vH_{\vPhi} = \lap_{g_{\vPhi}}\vPhi$. If the immersion is minimal, or if $\vH_{\vPhi}$ is assumed to be regular, one can not directly deduce from this system that $\vPhi$ is regular. Indeed, the Laplace--Beltrami operator $\lap_{g_{\vPhi}}$ depends on the immersion $\vPhi$ and consequently, is a degenerate operator. In order to be able to use elliptic regularity, we need to now a priori that the metric $g_{\vPhi}$ is well-behaved. That is to say, we need coordinates in which the coefficients $\big( g_{\vPhi}\big)_{ij}$ are uniformly elliptic and at least uniformly continuous. \\
	
	In dimension $n=2$, such coordinates are provided by the isothermal coordinates. Indeed, in these coordinates, the regularity boils down to the regularity of the conformal factor. Thanks to the Liouville equation, it suffices to know that the Gauss curvature is regular to deduce regularity for the conformal factor. For general metrics, Gauss proved in 1822 that any surface equipped with a real-analytic metric admits local isothermal coordinates. In 1916, Lichtenstein \cite{lichtenstein1916} extended Gauss's result by proving the existence of isothermal coordinates for metrics with Hölder continuity. Subsequent foundational works by Ahlfors, Bers, Chern, Hartman--Wintner, and Vekua \cite{ahlfors1955,chern1955,vekua1955,hartman1955} solidified the analytic framework, showing that such coordinates can be constructed even under the regularity of Dini continuity. However, in most of the variational problems concerning immersions (such as minimal surfaces, CMC surfaces or Willmore surfaces), it is not clear that the metric can be assumed a priori continuous. It has been proved by Toro \cite{toro1994} in 1994 that surfaces with second fundamental form in $L^2$ have a Lipschitz parametrization. In 1995, Müller--Sverak \cite{muller1995} proved that conformal parametrization of surfaces with second fundamental form in $L^2$ are Lipschitz parametrizations. In \cite{helein2002}, Hélein proved that isothermal coordinates can be obtained through the use of Coulomb frames. Together with Wente inequality, he proved that for surfaces with $L^2$ second fundamental form, there always exist isothermal coordinates in which the metric is continuous. We refer to \cite{Lan2025} for a recent survey of these results.
    \\
	
	In dimensions $n\geq 3$, it is well-known that isothermal coordinates do not exist for general metrics, the Weyl tensor being the obstruction. However, one can still wonder which immersed submanifolds carries a bi-Lipschitz parametrization. In order to study the Plateau problem, Reifenberg \cite{reifenberg1960} introduced a condition of approximability of submanifolds by affine planes, now known as the Reifenberg condition. Inspired by this condition, Semmes \cite{semmes1,semmes2,semmes3} proved in 1991 that hypersurfaces with a Gauss map being small in $BMO$ carry bi-Hölder parametrizations. In 1995, Toro \cite{toro1995} partially answered the question and proved that a submanifolds satisfying a refinement of the Reifenberg condition carries Lipschitz parametrizations. However, this condition requires to know a priori a certain decay of the density, which does not seem to be an easy condition to check for potential applications in variational problems.\\
	
	Instead of looking for Reifenberg-type conditions, one could look for assumptions on the second fundamental form, as it is done in dimension $n=2$. From an intrinsic view-point, we consider the induced metric $g_{\vPhi}$ of an immersion $\vPhi\colon \Sigma^n\to \R^d$, together with its Riemann tensor $\Riem^{g_{\vPhi}}$ given by the second fundamental form $\vII$ through the Gauss--Codazzi formula in local coordinates
	\begin{align}\label{eq:GC}
		\Riem^{g_{\vPhi}}_{ijkl} = \vII_{ik}\cdot \vII_{jl} - \vII_{il}\cdot \vII_{jk}.
	\end{align}
	Thus, a regularity condition on the second fundamental form can be translated directly to the Riemann tensor of $g_{\vPhi}$ \emph{without} the use of $g_{\vPhi}$ itself. This idea finds applications in the context of generalized Willmore energies. These functionals have been first introduced by Guven \cite{guven205} in 2005 who was looking for conformally invariant energies measuring the bending of 4-dimensional submanifolds of $\R^d$. In the context of the AdS/CFT correspondence, a procedure of volume renormalization first introduced by Henningson--Skenderis \cite{henningson1998} to compute the Weyl anomaly and then developed by Graham--Witten \cite{graham1999} to compute the renormalized volume of minimal submanifolds leads to the Willmore functional for hypersurfaces of $\R^d$. This computation has been generalized for 4-dimensional submanifolds of $\R^d$ independently by Graham--Reichert \cite{graham2020} and Zhang \cite{zhang2021} who recovered the same functional as Guven: given an immersion $\vPhi\colon \Sigma^4\to \R^5$,
	\begin{align}\label{eq:Willmore4d}
		\Er_{GR}(\vPhi) = \int_{\Sigma
        ^4} |\g \vII_{\vPhi}|^2_{g_{\vPhi}} - H^2_{\vPhi}\, |\vII_{\vPhi}|^2_{g_{\vPhi}} + 7\, H_{\vPhi}^4\ d\vol_{g_{\vPhi}}.
	\end{align}
	On the other hand, a notion of generalized Willmore energy in any (even) dimension has been introduced by Blitz, Gover and Waldron in \cite{blitz2024,gover2017,gover2020}, where they define a generalized Willmore energy as a conformally invariant functional for immersions $\vPhi\colon \Sigma^n\to \R^d$ whose Euler--Lagrange equation contains as the highest order term $\lap_{g_{\vPhi}}^{\frac{n}{2}} H_{\vPhi}$. This term arises in the Euler--Lagrange equation of every functional of the form 
	\begin{align*}
		\int_{\Sigma^n} \big| \g^{\frac{n-2}{2}} \vII_{\vPhi} \big|^2_{g_{\vPhi}} \ d\vol_{g_{\vPhi}} + l.o.t.
	\end{align*}
	Hence, the starting point for the analytical setting of such functionals is the assumption $\vII_{\vPhi} \in W^{\frac{n-2}{2},2}(\Sigma^n)$. Since the quantity $\|\vII_{\vPhi}\|_{W^{\frac{n}{2}-1,2}(\Sigma,g_{\vPhi})}$ is not scale invariant, we will rather work with the following functional for $n$ even
	\begin{align}\label{eq:prototype}
		\Er_n(\vPhi)\coloneq \sum_{i=0}^{\frac{n}{2}-1} \int_{\Sigma^n} \big| \g^i \vII_{\vPhi} \big|_{g_{\vPhi}}^{\frac{n}{i+1}}\ d\vol_{g_{\vPhi}} .
	\end{align}
	For instance, we have in low dimensions
	\begin{align*}
		& \Er_2(\vPhi) = \int_{\Sigma^2} \big| \vII_{\vPhi} \big|_{g_{\vPhi}}^2\ d\vol_{g_{\vPhi}}, \\[2mm]
		& \Er_4(\vPhi) = \int_{\Sigma^4} \big| \g \vII_{\vPhi} \big|^2_{g_{\vPhi}} + \big| \vII_{\vPhi} \big|^4_{g_{\vPhi}}\ d\vol_{g_{\vPhi}}, \\[2mm]
		& \Er_6(\vPhi) = \int_{\Sigma^6} \big| \g^2 \vII_{\vPhi} \big|^2_{g_{\vPhi}} + \big|\g \vII_{\vPhi} \big|^3_{g_{\vPhi}} + \big| \vII_{\vPhi} \big|^6_{g_{\vPhi}} \ d\vol_{g_{\vPhi}}.
	\end{align*}
	These functionals are all scale-invariant and provide an upper bound for any generalized Willmore energy. However, the setting of $C^{\infty}$ immersions is not adapted to the study of such problems since $\Er_n$ does not provide any constraint on the $L^{\infty}$ norm of any derivatives of $\vPhi$. Hence, the set of immersions has to be enlarged using a Sobolev version of immersions. To that extent, the notion of weak immersions in dimension 2 has been introduced by the second author in \cite{riviere2014}. We have the following definition.
	\begin{defi}
		Let $(\Sigma^n,h)$ be a closed oriented $n$-dimensional Riemannian manifold and $d>n$ be an integer. Given $k\in\N$ and $p\in[1,+\infty]$, we define the space of weak immersion $\I_{k,p}(\Sigma^n;\R^d)$ as follows:
		\begin{align*}
			\I_{k,p}(\Sigma^n;\R^d)\coloneqq \left\{
			\vPhi\in W^{k+2,p}(\Sigma;\R^d) : \exists c_{\vPhi} >1,\ c_{\vPhi}^{-1} h \leq g_{\vPhi} \leq c_{\vPhi}\, h
			\right\}.
		\end{align*}     
	\end{defi}

    Some comments are in order to understand this definition. The induced metric $g_{\vPhi} \coloneq \vPhi^*\geu$ is defined for any map $\vPhi\colon \Sigma^n\to \R^d$ for which the first derivatives lie in $L^2_{\loc}(\Sigma)$ using local coordinates: $\big(g_{\vPhi}\big)_{ij} = \dr_i \vPhi\cdot\dr_j \vPhi$. The existence of a constant $c_{\vPhi} >1$ such that $c_{\vPhi}^{-1} h \leq g_{\vPhi} \leq c_{\vPhi}\, h$ holds a.e. is the crucial information allowing to manipulate $\vPhi$ in a similar manner as for standard immersions. First, the upper bound $g_{\vPhi} \leq c_{\vPhi}\, h$ implies that $\vPhi$ is Lipschitz. Indeed, we have in local coordinates $(g_{\vPhi})_{ii} = |\dr_i \vPhi|^2$. However, we do not impose any restriction on $c_{\vPhi}$, meaning that we allow weak immersions to have a constant $c_{\vPhi}$ arbitrarily large. This is the key point that allows for singularities to appear in the study of compactness questions. Moreover, if $\vPhi\in W^{k+2,p}\cap W^{1,\infty}(\Sigma^n;\R^d)$, then the coefficients $\big(g_{\vPhi}\big)_{ij}$ are of class $W^{k+1,p}\cap L^{\infty}$ and in particular, not necessarily continuous, depending on the parameters $k$ and $p$. On the other hand, the lower bound $c_{\vPhi}^{-1} h < g_{\vPhi}$ implies that the partial derivatives of $\vPhi$ are linearly independent. Hence we can define the generalized Gauss map of $\vPhi$ in local coordinates as the multivector
    \begin{align*}
        \vn_{\vPhi} \coloneq *_{\R^d} \frac{\dr_1\vPhi \wedge \cdots \wedge \dr_n \vPhi}{ \left| \dr_1\vPhi \wedge \cdots \wedge \dr_n \vPhi \right| }.
    \end{align*}
    The two conditions $c_{\vPhi}^{-1} h < g_{\vPhi}$ and $\vPhi\in W^{k+2,p}(\Sigma^n;\R^d)$ imply that $\vn_{\vPhi}$ lies in $W^{k+1,p}$. Consequently, the second fundamental form of $\vPhi$ is a well-defined tensor with components in $W^{k,p}$ (this is the reason for the indices $k,p$ in the definition of $\I_{k,p}(\Sigma^n;\R^d)$). Hence, the space $\I_{k,p}(\Sigma^n;\R^d)$ can be roughly understood as a Sobolev version of the classical immersions, this time with second fundamental form merely in $W^{k,p}(\Sigma^n)$.\\
    
	In the two-dimensional case, if isolated singularities (branch points) of the induced metric are allowed in the definition of the space $\I_{0,2}(\Sigma^2;\R^d)$, then we obtain the closure of $\Imm(\Sigma^2;\R^d)$ under a bound on the $L^2$ norm of the second fundamental form by gluing such branched weak immersions in a bubble tree, as proved by Mondino and the second author in \cite{mondino2014}. 
    For such weak immersions, we recover the standard tools of the classical immersions, such as isothermal coordinates and Gauss--Bonnet formula (see for instance the lecture notes \cite{riviere2016} \textcolor{red}{and the recent survey \cite{Lan2025}}), but in a setting of very low regularity. The crucial tools required to develop this theory have been the study of Coulomb frame introduced by Hélein \cite{helein2002}, together with a remarkable div-curl structure in the Liouville equation, which turns out to be the analytic ingredient neeeded the obtain the continuity of the conformal factor. However, conformally flat coordinates do not exist on general manifolds in higher dimensions and thus, finding coordinates for which the metric is continuous is not straightforward.\\

    In high dimensions, the harmonic coordinates introduced by Einstein \cite{einstein1916} and Lanczos \cite{lanczos1922} are an alternative to conformal coordinates providing the maximal regularity for the coefficients of the metric depending on the regularity of its curvature. This is exactly the case for isothermal coordinates in dimension $n=2$. The goal of the present work is to construct harmonic coordinates for elements of $\I_{\frac{n}{2}-1,2}(\R^n;\R^d)$, that the authors apply in the coming work \cite{MarRiv20252} to study the analysis of functionals such as \eqref{eq:Willmore4d} and \eqref{eq:prototype}. 
    Our main result is the following.
    
    \begin{theorem}\label{th:Imm}
		Let $n\geq 2$ be an even integer and $d>n$ be an integer. There exists $\eps_*>0$ and $C_*>0$ depending only on $n$ and $d$ such that the following holds.  Let $\vPhi\in \I_{\frac{n}{2}-1,2}(\R^n;\R^d)$ be such that 
		\begin{align*}
			\Er_n(\vPhi) < \eps_*.
		\end{align*}
		Then, there exist a bi-Lipschitz homeomorphism $z\colon \R^n\to \R^n$ providing harmonic coordinates for $g_{\vPhi}$ on $\R^n$ such that for all $1\leq i,j\leq n$, it holds $\big(g_{\vPhi\circ z}\big)_{ij}\in C^0(\R^n)$ with the estimate
		\begin{align*}
			\left\| \big(g_{\vPhi\circ z}\big)_{ij} - \delta_{ij} \right\|_{L^{\infty}(\R^n)} \leq C_*\, \Er_n(\vPhi).
		\end{align*}
	\end{theorem} 

    For weak immersions $\vPhi\in \I_{\frac{n}{2}-1,2}(\B^n;\R^d)$, the construction of harmonic coordinates boils down to knowing whether we can extend $\vPhi$ into a weak immersion of $\R^n$ into $\R^d$ while keeping a control on the second fundamental form. For instance, we can consider immersions $\vPhi\colon \B^n\to \R^d$ such that its restriction to the sphere $\vPhi\colon \s^{n-1}\to \s^{d-1}$ parametrizes the graph of a function close to zero. We perform this extension in \Cref{th:local_imm_v2} and obtain the existence of harmonic coordinates in this setting.\\
    
    The continuity of the metric might be surprising at first glance since we are assuming merely $\vPhi\in W^{\frac{n}{2}+1,2}(\B^n)\hookrightarrow W^{2,n}(\B^n)$. This space does not embed in $C^1$, and thus the metric $g_{\vPhi}$ has no reason to be continuous a priori. Indeed, the standard Sobolev injections state that $W^{2,p}(\B^n)\hookrightarrow C^{1,1-\frac{n}{p}}(\B^n)$ for $p>n$, which is the reason why Hölder spaces appear in non-critical situations, but $W^{2,n}$ does not embed in $C^1$. The crucial analytic tool in our situation is the refinement of Sobolev immersions using Lorentz spaces $L^{(p,q)}$ (with $q\in[1,+\infty]$). These Lorentz spaces are refinements of the standard Lebesgue spaces $L^p$ in the following sense: 
    \begin{enumerate}
        \item We have $L^{(p,p)} = L^p$,
        
        \item If $p>1$ and $1\leq q<r\leq +\infty$, then $L^{(p,q)} \subsetneq L^{(p,r)}$,
    
        \item If $p>q>1$ and $r,s\in[1,+\infty]$, then $L^{(p,r)}\subsetneq L^{(q,s)}$.
    \end{enumerate}
    Moreover, if $p<n$, then the injection $W^{1,p}(\B^n)\hookrightarrow L^{\frac{np}{n-p}}(\B^n)$ can be improved in $W^{1,p}(\B^n)\hookrightarrow L^{\left( \frac{np}{n-p}, p \right)}(\B^n)$. If now we consider functions with derivates in a Lorentz space, then the Lorentz exponent is preserved: we have $W^{1,(p,q)}(\B^n)\hookrightarrow L^{\left(\frac{np}{n-p},q\right)}(\B^n)$.
    Consequently, the second fundamental form $\vII_{\vPhi}$ of a weak immersion $\vPhi\in\I_{\frac{n}{2}-1,2}(\B^n;\R^n)$ lies in $W^{\frac{n}{2}-1,2}(\B^n)$, which implies that $\vII_{\vPhi}$ also lies in $L^{(n,2)}(\B^n)$. Another property of Lorentz spaces is that the Lorentz exponent improves with products. More precisely, if we consider two functions $f,g\in L^{(p,q)}$ with $p\geq 2$, then the product $fg$ lies in $L^{\left(\frac{p}{2},\frac{q}{2}\right)}$. For a weak immersion $\vPhi\in \I_{\frac{n}{2}-1,2}(\B^n;\R^d)$, the Gauss--Codazzi equations \eqref{eq:GC} imply that $\Riem^{g_{\vPhi}} \in L^{\left(\frac{n}{2},1\right)}(\B^n)$. If $g_{\vPhi}$ possessed controlled harmonic coordinates, then we could conclude that the coefficients $(g_{\vPhi})_{ij}$ lie in $W^{2,\left(\frac{n}{2},1\right)}$. The crucial property of this Lorentz--Sobolev space, is a critical improvement of the Sobolev spaces. Indeed we have the injections $W^{2,\left(\frac{n}{2},1\right)}(\B^n)\hookrightarrow W^{1,(n,1)}(\B^n)\hookrightarrow C^0(\B^n)$ (not just $L^{\infty}$). Consequently, the existence of harmonic coordinates imply continuity estimates on the coefficients of the metric by its curvature.\\

    However, understanding the interaction between a given Riemannian metric $g$ and its Riemann curvature tensor $\Riem^g$ is a difficult problem due to the highly non-linear nature of the Riemann tensor but also because of its geometric nature: both $g$ and $\Riem^g$ are tensors and thus, independent of the choice of coordinates. If one roughly describe the Levi-Civita connection $\g^g$ as the "gradient" of $g$, then $\Riem^g$ should consists in the Hessian of $g$ and its trace, namely the Ricci curvature $\Ric^g$, should be a geometric version of the Laplacian of $g$. This idea can be made partially true by choosing good coordinates, now known as harmonic coordinates. More precisely, the components of the Ricci tensor in local coordinates are given by the following formula, as proved in \cite[Lemma 4.1]{deturck1981}, where $\Gamma_{ij}^k$ are the Christoffel symbols of $g$ and $Q(g,\dr g)$ denotes quadratic terms in the coefficients $g_{\alpha\beta}$, $g^{\alpha\beta}$ and $\dr_{\gamma} g_{\alpha\beta}$, 
	\begin{align}\label{eq:Ricci}
		\begin{cases} 
			\Ric^g_{ij} = -\frac{1}{2}\, g^{\alpha\beta}\, \dr^2_{\alpha\beta} g_{ij} + \frac{1}{2}\left( g_{ri}\, \dr_j \Gamma^r + g_{rj}\, \dr_i \Gamma^r \right) + Q_{ij}(g,\dr g), \\[2mm]
			\Gamma^r = g^{ij}\, \Gamma^r_{ij}.
		\end{cases} 
	\end{align}
	If one finds coordinates in which $\Gamma^r=0$ and the coefficients $g^{\alpha\beta}$ are sufficiently regular, then the above formulas can be understood as an elliptic system on the coefficients $g_{ij}$. Such coordinates $(x^1,\ldots,x^n)$ are called harmonic because they satisfy $\lap_g x^i=0$ and their existence have been proved by DeTurck--Kazdan \cite{deturck1981} in 1981. However, no quantitative estimate on the size of the domain and no geometric condition that would guarantee their existence on some ball of fixed size have been proved in \cite{deturck1981}. Due to the possible applications to analytic problems such as the regularity and compactness questions of Einstein or Bach-flat metrics \cite{anderson1989,anderson2004,czimek2019,yangI1992,yangII1992}, a lot of attention have been drown to the construction of such coordinates. Indeed, once we know that a sequence of manifold converges in Gromov--Hausdorff topology, we need an atlas with uniform estimates to prove that the convergence holds in a stronger topology, as described in \cite{petersen2016}, if for instance we have an $L^p$ bound on the Riemannian curvature for some $p>\frac{n}{2}$. In the Lorentzian setting, such coordinates are known as the wave coordinates and have been used to study the Cauchy problem and the stability Minkowski space in general relativity, see for instance \cite{choquet2009}. \\
	
	There are essentially three common ways to construct harmonic coordinates on a Riemannian manifold $(M^n,g)$ of dimension $n\geq 3$. 
	\begin{enumerate}
		\item The first is to consider a given ball $B_g(p,r)\subset M$ that carries coordinates $(x^1,\ldots,x^n)$ in which we already know that the coefficients $g_{ij}$ of the metric are sufficiently regular for the following argument to hold. We then consider the elliptic system 
		\begin{align}
			\begin{cases}
				\lap_g z^i =0 & \text{ in }B_g(p,r),\\
				z^i = x^i & \text{ on }\dr B_g(p,r).
			\end{cases}
		\end{align}
		If $r$ is small enough, then the coefficients $g_{ij}$ are close to a constant matrix and one can deduce that the map $z\coloneq (z^1,\ldots,z^n)$ defines a coordinates system, which is harmonic by definition. To perform this argument and obtain a control on $z$, one needs to have an \textit{a priori} control on the metric in at least some Hölder or Sobolev space. 
		
		\item A second way is to proceed by contradiction, as done for instance in \cite{anderson1992,hiroshima1995}. Under some geometric conditions, for instance the Ricci tensor bounded in some $L^p$ for $p>\frac{n}{2}$ by some $\Lambda>0$ with the injectivity radius bounded from below by some $\iota>0$ and the volume bounded from above by some $V>0$, we assume that the harmonic radius (the smallest radius where harmonic coordinates exist) is not bounded from below by a constant depending only on $\Lambda$, $\iota$ and $V$. We then obtain a sequence of manifolds that (up to rescaling) will converge to $\R^n$ but with finite harmonic radius, which is impossible.
		
		\item A third possibility, is to regularize the metric through the Ricci flow as proved in \cite{yangI1992,yangII1992}. Indeed, the $\eps$-regularity coming with such a regularization allows to reduce the construction of harmonic coordinates to the case where one has $L^{\infty}$ bounds on the curvature. This requires to have an a priori estimate on the Sobolev constant. In this setting, the construction of harmonic coordinates have been obtained by Jost--Karcher \cite{jost1982}, see also \cite{jost1984}, thanks to a careful study of the distance function and Jacobi fields.
	\end{enumerate}
	The first method is not adapted for compactness questions of metrics satisfying some $L^p$ bound on some curvature, since we do not have any a priori control on the metric. The other methods require to have at least the Ricci tensor in some $L^p$ for $p>\frac{n}{2}$ but also the control on other geometric quantities such as the injectivity radius or the Sobolev constants. Hence, one can wonder whether there is an intermediate space containing all the $L^p$ for $p>\frac{n}{2}$ that would allow for the existence and control of harmonic coordinates. From a heuristic view-point, one needs at least $L^{\infty}$ estimates on the metric to use elliptic regularity for the system 
	\begin{align*}
		\Ric^g_{ij} = -\frac{1}{2}\, g^{\alpha\beta}\, \dr^2_{\alpha\beta} g_{ij} + Q_{ij}(g,\dr g).
	\end{align*}
	Thus, the question boils down to knowing which is the minimal functional space $X$ for which solutions to $\lap u \in X$ verify $u \in L^{\infty}$. If $X=L^p$ for $p\in(1,\infty)$, then we obtain $u\in W^{2,p}$ by Calderon--Zygmund theory. If $p\in (\frac{n}{2},n)$, then we have the Sobolev injections $W^{2,p}\hookrightarrow W^{1,\frac{np}{n-p}}$, with $\frac{np}{n-p}>n$ for $p>\frac{n}{2}$. Thus, there exists $\alpha=\alpha(n,p)\in(0,1)$ such that $W^{1,\frac{np}{n-p}}\hookrightarrow C^{0,\alpha}$ and we obtain $L^{\infty}$ estimates. However, for $p\leq \frac{n}{2}$, we have $\frac{np}{n-p}\leq n$ and the injections $W^{1,\frac{np}{n-p}}\hookrightarrow L^{\infty}$ is known to fail. For $p=\frac{n}{2}$, we have $\frac{np}{n-p}=n$ and the space $W^{1,n}$ is a subset of the set of functions with bounded mean oscillations $BMO$, different from $L^{\infty}$ (the logarithm for instance, lies in $BMO\setminus L^{\infty}$). In order to have precisely the injection in $L^{\infty}$, one needs to refine the Sobolev space and consider for instance Lorentz spaces. Indeed, we have the injection $W^{1,(n,1)}\hookrightarrow C^0$ but $W^{1,(n,q)}\not\hookrightarrow L^{\infty}$ for any $q>1$. This suggests that the critical regularity scale for the existence of harmonic coordinates should be that $\dr_k g_{ij} \in L^{(n,1)}$. Adding one more derivative in order to define the Riemann tensor, one would ask that the second derivatives $\dr^2_{kl} g_{ij}$ lie in $L^{\left(\frac{n}{2},1\right)}$ and thus $\Riem_{ijkl}\in L^{\left(\frac{n}{2},1\right)}$. Moreover, for any $p>\frac{n}{2}$, we have $L^p \subset L^{\left(\frac{n}{2},1\right)}$. Hence, the assumption $\Riem_{ijkl}\in L^{\left(\frac{n}{2},1\right)}$ is much weaker than the assumption $\Riem_{ijkl}\in L^p$ for some $p>\frac{n}{2}$. As explained previously, this setting arise naturally in the context of variational calculus in extrinsic conformal geometry and thus, requires a careful analysis. \\
    
	In this work, we study the construction of coordinates having maximal regularity for metrics $g$ satisfying $\Riem^g\in L^{\left(\frac{n}{2},1\right)}$. The main idea is to generalize the use of Coulomb frame introduce by Hélein in the framework of Willmore surfaces to construct isothermal coordinates in dimension 2 \cite{helein2002}. Assume for simplicity, that $g_{ij}$ is close to $\delta_{ij}$. The first step is to apply Uhlenbeck's construction of Coulomb coframe \cite{uhlenbeck1982} (for the Euclidean metric) which guarantees the existence of a coframe $(\omega^1,\ldots,\omega^n)$ for $g$ and Coulomb for the Euclidean metric, under the assumption $\Riem^g\in L^{\left(\frac{n}{2},1\right)}$. To construct coordinates on a ball $\B^n$, we consider for each $i\in\{1,\ldots,n\}$ the solution $y^i\colon \B^n \to \R$ to the system
	\begin{align}\label{eq:system_y}
		\begin{cases}
			-\lap_g y^i = d^{*_g} \omega^i & \text{ in }\B^n ,\\[2mm]
			\dr_{\nu}y^i = \omega^i(\nu) & \text{ on }\dr\B^n ,\\[2mm]
			\int_{\B} y^i(x)\, dx = 0.
		\end{cases}
	\end{align}
	Since $g_{ij}$ is close to $\delta_{ij}$, we can apply the standard elliptic estimates and we obtain that $dy^i$ is close to $\omega^i$ at distance $\|\Riem^g\|_{L^{\left(\frac{n}{2},1\right)}}$. Since $g=\sum_i \omega^i \otimes \omega^i$, we deduce that $g_{ij}$ is close to $\scal{\dr_i y}{\dr_j y}_{\geu}$ and that $(y^{-1})^*g$ is close to $\delta_{ij}$ at distance $\|\Riem^g\|_{L^{\left(\frac{n}{2},1\right)}}$ as well. This idea is at the heart of the continuity method used in \cite{uhlenbeck1982} to construct Coulomb frames. In order to close this continuity argument, we need to control the domain $y(\B^n)$. This is \textit{a priori} unclear in the full generality that we are addressing. For instance one could consider $g$ of the form $f^*\geu$ for some arbitrary diffeomorphism $f\colon \B^n\to \Omega$ with $\Omega\subset \R^n$ being an open set. In this case, the coordinates we construct are exactly $y=f$ and we are working on an arbitrary open set $\Omega\subset \R^n$ that is diffeomorphic to $\B^n$. Hence, we will rather work on the full $\R^n$ instead of a fixed domain and discuss this restriction below. By solving the following system, we obtain harmonic coordinates
    \begin{align}\label{eq:perturb}
        \begin{cases}
            \lap_{(y^{-1})^*g} z^i = 0 & \text{ in }\B^n,\\[2mm]
            z^i = y^i & \text{ on }\dr \B^n.
        \end{cases}
    \end{align}
    This construction answers a question raised by Hélein in Remark 5.4.6 of \cite{helein2002} about the construction of coordinates induced by Coulomb frames in dimensions $n\geq 3$.\\
	
	In order to work on $\R^n$, we need to use Sobolev norms that are scale-invariant. Given an open set $\Omega\subset \R^n$, we define the following norm for maps $f\colon \Omega\to \R$ whose Hessian lies in $L^{\left(\frac{n}{2},1\right)}$
	\begin{align*}
		\|f\|_{\bar{W}^{2,\left(\frac{n}{2},1\right)}(\Omega)} \coloneqq \|f\|_{L^{\infty}(\Omega)} + \|\g f\|_{L^{(n,1)}(\Omega)} + \|\g^2 f\|_{L^{\left(\frac{n}{2},1\right)}(\Omega)}.
	\end{align*}
	Since we will work with discontinuous metrics, we define the following notion of weak metric, which is essentially the same as the one introduced by the second author in the context of Willmore surfaces in \cite{riviere2016}.
	
	\begin{defi}\label{def:Weak_metric}
		Let $n\geq 3$ be an integer, $\Omega\subset \R^n$ be an open set and $S_n^+$ be the set of symmetric positive definite 2-tensors on $\R^n$. We say that a map $g\colon \Omega\to S_n^+$ is a weak metric of class $W^{k,p}(\Omega)$ (or $\bar{W}^{2,\left(\frac{n}{2},1\right)}(\Omega)$) if the following two conditions are satisfied for the coefficients $g_{\alpha\beta}$ in the expansion $g=g_{\alpha\beta}\, dx^{\alpha}\otimes dx^{\beta}$:
		\begin{enumerate}
			\item there exists $\lambda>1$ such that the following pointwise inequality of matrices is verified a.e. on $\Omega$:
			\begin{align*}
				\lambda^{-1}\, \delta_{\alpha\beta} \leq g_{\alpha\beta} \leq \lambda\, \delta_{\alpha\beta}.
			\end{align*}
			\item the coefficients $g_{\alpha\beta}$ lies in the Sobolev space $W^{k,p}(\Omega)$ (or  $\bar{W}^{2,\left(\frac{n}{2},1\right)}(\Omega)$).
		\end{enumerate}
	\end{defi}

    It turns out that this strategy works for general weak metrics on $\R^n$ with Riemann tensor in $L^{\left(\frac{n}{2},1\right)}$. Hence, we first prove the following result, which provides the existence of controlled coordinates for weak metrics of class $\bar{W}^{2,\left(\frac{n}{2},1\right)}(\R^n)$, see \Cref{th:diff_Rn_v2} below for a stronger version. 
	
	\begin{theorem}\label{th:diff_Rn}
		Let $n\geq 3$ be an integer. There exist $\eps_0>0$ and $C_0>0$ depending only on $n$ satisfying the following property. Let $g$ be a weak metric on $\R^n$ of class $ \bar{W}^{2,\left(\frac{n}{2},1\right)}(\R^n)$ such that 
		\begin{align}\label{eq:smallness}
			\|\Riem^g\|_{L^{\left(\frac{n}{2},1\right)}(\R^n,g)} \leq \eps_0.
		\end{align}
		Then there exists a $C^1$ diffeomorphism $y\colon \R^n\to \R^n$ such that for $i,j\in\{1,\ldots,n\}$, it holds
		\begin{align}\label{eq:control_components}
			\|(y^*g)_{ij} - \delta_{ij} \|_{\bar{W}^{2,\left(\frac{n}{2},1\right)}(\R^n)} \leq C_0\, \|\Riem^g\|_{L^{\left(\frac{n}{2},1\right)}(\R^n,g)}.
		\end{align}
	\end{theorem} 
	
	The main achievement of the above result is that the initial regularity assumption $\bar{W}^{2,\left(\frac{n}{2},1\right)}(\R^n)$ of the metric does not need to be quantified and no estimate on the Sobolev constant or injectivity radius is needed. In particular, the constant $\lambda$ required to define a weak metric does not play any role. Hence, we can lower this initial regularity as long as one can still make sense of the smallness assumption \eqref{eq:smallness}. For instance, we can extend it to weak metrics with Hessian in $L^{\left(\frac{n}{2},2\right)}$, see \Cref{th:Weak_metrics_v2} below. Moreover, the control \eqref{eq:control_components} implies that by solving \eqref{eq:perturb}, we can obtain controlled harmonic coordinates. We obtain the following result.
	
	\begin{theorem}\label{th:Weak_metrics_harm}
    Let $n\geq 3$ be an integer.
		There exist $\eps_1>0$ and $C_1>0$ depending only on $n$ satisfying the following property. Let $g$ be a weak metric on $\R^n$ such that $\dr_k g_{ij}\in L^{(n,2)}(\R^n)$, $\dr^2_{kl} g_{ij} \in L^{\left(\frac{n}{2},2\right)}(\R^n)$ and 
		\begin{align*}
			\|\Riem^g\|_{L^{\left(\frac{n}{2},1\right)}(\R^n,g)} \leq \eps_1.
		\end{align*}
		Then there exists a bi-Lipschitz homeomorphism $z\colon \R^n\to \R^n$ providing harmonic coordinates for $g$ on $\R^n$ and such that the following estimates hold for all $1\leq i,j\leq n$
		\begin{align*}
			\left\|(z^*g)_{ij} - \delta_{ij} \right\|_{\bar{W}^{2,\left(\frac{n}{2},1\right)}(\R^n,\geu)} \leq C_1\, \|\Riem^g\|_{L^{\left(\frac{n}{2},1\right)}(\R^n,g)}.
		\end{align*}
	\end{theorem} 
	
	In order to obtain a local version, say for weak metrics on $\B^n$, the only problem boils down to extending a given metric on $\B^n$ to a metric on $\R^n$ that preserves the smallness of the Riemann tensor. We do not know whether this is possible in full generality. For instance, we obtain the following result.
	\begin{theorem}\label{th:local}
		Let $n\geq 3$. There exists $\eps_2>0$ and $C_2>0$ depending only on $n$ such that the following holds. Let $g$ be a weak metric on the unit ball $\B^n$ of class $W^{2,\left(\frac{n}{2},2\right)}(\B^n)$ such that for any $i,j\in\{1,\ldots,n\}$, it holds
		\begin{align}\label{eq:boundary}
			\|g_{ij} - \delta_{ij}\|_{W^{2,n}(\s^{n-1})} + \|\dr_r g_{ij}\|_{W^{1,n}(\s^{n-1})}+ \|\Riem^g\|_{L^{\left(\frac{n}{2},1\right)}(\B^n,g)} \leq \eps_2.
		\end{align}
		Then there exists an open set $\Omega\subset \R^n$ and a bi-Lipschitz homeomorphism $z\colon \Omega\to \B^n$ providing harmonic coordinates for $g$ such that for every $1\leq i,j\leq n$, it holds
		\begin{align*}
			\left\| (z^*g)_{ij} - \delta_{ij} \right\|_{W^{2,\left(\frac{n}{2},1\right)}(\Omega,\geu)} \leq C_2\, \left(  \sum_{i,j}\left( \|g_{ij} - \delta_{ij}\|_{W^{2,n}(\s^{n-1})} + \|\dr_r g_{ij}\|_{W^{1,n}(\s^{n-1})}\right)+ \|\Riem^g\|_{L^{\left(\frac{n}{2},1\right)}(\B^n,g)} \right).
		\end{align*}
	\end{theorem}

    \begin{remark}
    \begin{enumerate}
        \item In this result, we can use Hölder inequality to replace the smallness of $\|\Riem^g\|_{L^{\left(\frac{n}{2},1\right)}(\B^n,g)}$ by the smallness of $\vol_g(\B)^{\frac{2}{n} - \frac{1}{p}}\, \|\Riem^g\|_{L^p(\B^n,g)}$ for some $p>\frac{n}{2}$. The constants $\eps_2$ and $C_2$ will then depend on $p$ and blow up as $p\to \frac{n}{2}$.

        \item The initial assumption \eqref{eq:boundary} does not prevent a priori the coefficients $g_{ij}$ to degenerate in $\B^n$.
    \end{enumerate}
    \end{remark}
	
	The system \eqref{eq:system_y} does not provide harmonic coordinates. However, it provides coordinates such that the quantities $\Gamma^r$ defined in \eqref{eq:Ricci} are controlled in $W^{1,\left(\frac{n}{2},2\right)}$ by the Riemann tensor. Thus, the coefficients of the metric in these coordinates still verify an elliptic system for which the regularity is fully determined by the Riemann tensor (the regularity of a Coulomb coframe is fully determined by the regularity of its curvature). With this method, it seems that the full Riemann tensor is involved, and not just the Ricci tensor as in the known existence results for harmonic coordinates. \\

	\subsubsection*{Organization of the paper.} In \Cref{sec:Preliminaries}, we define the Lorentz spaces and present the Sobolev inequalities in this setting. We also define the different notions of curvature, as a differential form or a tensor, and introduce the notations concerning the covariant derivatives. In \Cref{sec:Proofmain}, we prove \Cref{th:diff_Rn}. In \Cref{sec:ProofSecond}, we prove \Cref{th:Weak_metrics_harm}. In \Cref{sec:extension}, we prove \Cref{th:local}. In \Cref{sec:Imm}, we study the case of immersions and we prove \Cref{th:Imm} and its local version \Cref{th:local_imm_v2}.\\

    \subsubsection*{Acknowledgments.} 
    This project is financed by Swiss National Science Foundation, project SNF 200020\textunderscore219429.

	\section{Analytic and geometric preliminaries}\label{sec:Preliminaries}
	
	In this section, we fix some notations, introduce the Lorentz spaces on Riemannian manifolds and present the necessary formulas concerning the curvature and covariant derivatives in coordinates and in moving frames.\\
	
	\emph{Notations:} 
	\begin{enumerate}
		\item In the rest of the paper, the constants $C$ written without dependence in the estimates are constants depending only on $n$.
        
		\item We denote $\geu$ the Euclidean metric on $\R^n$, that is to say $\geu = \delta_{ij}\, dx^i\otimes dx^j$.

        \item Given an open set $\Omega\subset \R^n$, we denote $\Diff(\Omega)$ the set of $C^1$ diffeomorphisms of $\Omega$. We denote $\BiLip(\Omega)$ the set of bi-Lipschitz homeomorphisms of $\Omega$.

        \item For $s>0$, the set $\B_s=\B_s^n$ denotes the open Euclidean ball of $\R^n$ centred at the origin with radius $s$.
	\end{enumerate}

	\subsection{Lorentz and Sobolev spaces}
	
	For an introduction to Lorentz spaces on more general spaces, we refer to \cite{grafakos2014,bennett1988}. Given a Riemannian manifold $(M^n,g)$, we define the Lorentz space $L^{(p,q)}(M)$ for $p,q\in[1,+\infty]$ as the set of measurable functions $f\colon M\to \R$ such that 
	\begin{align*}
		\|f\|_{L^{(p,q)}(M,g)}^q = \|f\|_{L^{(p,q)}(M)}^q = p \left\| \lambda\mapsto \lambda\, \vol_g\left(\{x\in M: |f(x)|>\lambda \}\right)^{\frac{1}{p}} \right\|_{L^q\left((0,+\infty),\frac{d\lambda}{\lambda}\right)}^q.
	\end{align*}
	The Lorentz spaces verifies the following properties:
	\begin{enumerate}
		\item If $p=q$, then $L^{(p,p)}(M)=L^p(M)$ with 
		\begin{align*}
			\forall f\in L^p(M),\qquad \|f\|_{L^p(M)} = \|f\|_{L^{(p,p)}(M)}.
		\end{align*}
		\item If $q<r$, then $L^{(p,q)}(M) \subsetneq L^{(p,r)}(M)$ with 
		\begin{align*}
			\forall f\in L^{(p,q)}(M),\qquad \|f\|_{L^{(p,r)}(M)} \leq C(p,q,r)\, \|f\|_{L^{(p,q)}(M)}.
		\end{align*}
		\item If $p<q$ and $s,t\in[1,\infty]$ and $\vol_g(M)<+\infty$, then $L^{(q,s)}(M)\subsetneq L^{(p,t)}(M)$ with 
		\begin{align*}
			\forall f\in L^{(q,s)}(M),\qquad \|f\|_{L^{(p,t)}(M)} \leq C(p,q,s,t)\, \vol_g(M)^{\frac{1}{p}-\frac{1}{q}} \, \|f\|_{L^{(q,s)}(M)}.
		\end{align*}
	\end{enumerate}
	We denote $W^{k,(p,q)}(M)$ the space of functions $f\colon M\to\R$ whose first $k^{th}$ derivatives lie in $L^{(p,q)}(M)$. We then have the following Sobolev inequalities on $(\B^n,\geu)$, valid for all $f\in C^{\infty}(\B)$ and $p\in[1,n)$ and $q\in[1,+\infty]$,
	\begin{align*}
		\begin{cases}
			\displaystyle \|f\|_{L^{\left(\frac{np}{n-p},q\right)}(\B)} \leq C(n,p,q)\, \|f\|_{W^{1,(p,q)}(\B)} , \\[3mm]
			\displaystyle \|f\|_{L^{\infty}(\B)} \leq C(n)\, \|f\|_{W^{1,(n,1)}(\B)}.
		\end{cases}
	\end{align*}
	If we denote $\bar{f}\coloneq \fint_{\B} f(x)\, dx$, we also have the estimates
	\begin{align}\label{eq:Sobolev_avg}
		\begin{cases}
			\displaystyle \|f-\bar{f}\|_{L^{\left(\frac{np}{n-p},q\right)}(\B)} \leq C(n,p,q)\, \|\g f\|_{L^{(p,q)}(\B)} , \\[3mm]
			\displaystyle \|f-\bar{f}\|_{L^{\infty}(\B)} \leq C(n)\, \|\g f\|_{L^{(n,1)}(\B)}.
		\end{cases}
	\end{align}
	
	\subsection{Covariant derivatives}
	
	In this section, we fix a metric $g$ on an open set $\Omega\subset \R^n$ for $n\geq 3$. We refer to \cite{petersen2016} for the different definitions and formulas presented here.\\
	
	\paragraph*{Metrics and Levi-Civita connections.}
	Using the Euclidean coordinates on $\Omega$, we view $g$ as a map sending each point $x\in \B$ to a symmetric positive definite matrix $(g_{\alpha\beta}(x))_{1\leq \alpha,\beta\leq n}$ where $g_{\alpha\beta} \coloneq g\left(\frac{\dr}{\dr x_{\alpha}}, \frac{\dr}{\dr x_{\beta}}\right)$. We denote $(g^{\alpha\beta})_{1\leq \alpha,\beta\leq n}$ the inverse matrix of $(g_{\alpha\beta})_{1\leq \alpha,\beta\leq n}$. Given a vector field $X\colon \Omega\to \R^n$, its norm with respect to $g$ will be denoted
	\begin{align*}
		\forall x\in \Omega,\qquad |X|_g(x) \coloneq \left( g_{ij}(x)\, X^i(x)\, X^j(x) \right)^{\frac{1}{2}}.
	\end{align*}
	To a weak metric $g$, we associate its Levi-Civita connection $\g^g$ defined as the unique torsion-free connection verifying
	\begin{align*}
		\forall X,Y\in C^{\infty}(\Omega;\R^n),\qquad d\left(\scal{X}{Y}_g\right) = \scal{\g^g X}{Y}_g + \scal{X}{\g^g Y}_g.
	\end{align*}
	To express $\g^g X$ in coordinates, we introduce the Christoffel symbols
	\begin{align}\label{eq:Christoffel}
		\Gamma_{ij}^k \coloneq \frac{1}{2}\, g^{kl}\left( \dr_i g_{jl} + \dr_j g_{il} - \dr_l g_{ij} \right).
	\end{align}
	In coordinates, we have the expression 
	\begin{align*}
		\g^g_i X = \g^g_{\frac{\dr}{\dr x^i}} X = \left( \frac{\dr X^j}{\dr x^i} + \Gamma^j_{ik}\, X^k \right) \frac{\dr}{\dr x^j}
	\end{align*}

	\paragraph*{Derivatives of differential forms.}
	
	We consider a $k$-form on $\Omega$ given in coordinates by
	\begin{align*}
		\alpha = \sum_{i_1<\cdots<i_k} \alpha_{i_1\cdots i_k}\, dx^{i_1}\wedge \cdots \wedge dx^{i_k} = \frac{1}{k!}\, \alpha_{i_1\cdots i_k}\, dx^{i_1}\wedge \cdots \wedge dx^{i_k} .
	\end{align*}
	The covariant derivatives of $\alpha$ with respect to $g$ are given by
	\begin{align}\label{eq:covariant_form}
		\g^g_j\alpha = \frac{1}{k!} \left( \dr_j \alpha_{i_1\ldots i_k} - \sum_{\ell=1}^k \Gamma^m_{j\, i_{\ell}}\, \alpha_{i_1\cdots i_{\ell-1}\, m\, i_{\ell+1}\cdots i_k} \right)\, dx^{i_1}\wedge \cdots \wedge dx^{i_k}.
	\end{align}
	We denote
	\begin{align*}
		\g_j^g \alpha_{i_1\cdots i_k} \coloneq \dr_j \alpha_{i_1\ldots i_k} - \sum_{\ell=1}^k \Gamma^m_{j\, i_{\ell}}\, \alpha_{i_1\cdots i_{\ell-1}\, m\, i_{\ell+1}\cdots i_k}.
	\end{align*}
    The codifferential of $\alpha$ is given by
    \begin{align}
        d^{*_g}\alpha & = -\frac{1}{k!}\, g^{ml}\, (\g^g_l \alpha_{m\, i_1\cdots i_{k-1}} )\, dx^{i_1}\wedge \cdots \wedge dx^{i_{k-1}} \nonumber \\[2mm]
         & = d^{*_{\geu}}\alpha -\frac{1}{k!}\left[ \left(  g^{ml} - \delta^{ml}\right)\, (\dr_l \alpha_{m\, i_1\cdots i_{k-1}} ) + g^{ml}\, \sum_{\ell=1}^k \Gamma^m_{j\, i_{\ell}}\, \alpha_{i_1\cdots i_{\ell-1}\, m\, i_{\ell+1}\cdots i_k}  \right] dx^{i_1}\wedge \cdots \wedge dx^{i_{k-1}}. \label{eq:codiff_cov}
    \end{align}
	
	\paragraph*{Curvature in a frame.}
	Let $(e_1,\ldots,e_n)$ be an orthonormal frame for $g$, that is to say for each $x\in \Omega$, $(e_1(x),\ldots,e_n(x))$ is a basis of $\R^n$ such that $g(e_i,e_j)=\delta_{ij}$. We consider $(\omega^1,\ldots,\omega^n)$ its dual coframe, that is to say $\omega^i(e_j)=\delta^i_j$. We define the connection forms $\omega_i^j$ as the 1-forms satisfying 
	\begin{align}\label{eq:domegai}
		d\omega^i = \sum_{j=1}^n \omega^j\wedge \omega_i^j.
	\end{align}
	We obtain the formula
	\begin{align}\label{eq:der_frame}
		\g^g e_i = \omega^j_i\otimes e_j.
	\end{align}
	We end up with the formula
	\begin{align}\label{eq:dstar_omegai}
		d^{*_g} \omega^i = \sum_{j=1}^n \left( \g^g_{e_j}\omega^i\right)(e_j) = \sum_{j=1}^n \left( e_j\left(\omega^i(e_j)\right) - \omega^i\left(\g^g_{e_j} e_j\right) \right) = -\sum_{j=1}^n \omega^i_j(e_j).
	\end{align}
	Now, we define the curvature 2-forms $F_i^j$ by
	\begin{align}\label{eq:Curv_frame}
		F_i^j = d\omega^j_i + \sum_{k=1}^n \omega^j_k \wedge \omega^k_i.
	\end{align}
	The Riemann tensor of $g$ and its curvature 2-form are linked by the relations
	\begin{align}\label{eq:curvatures}
		\begin{cases} 
			\displaystyle \Riem^g(e_i,e_j)e_k = \sum_{l=1}^n F_k^l(e_i,e_j)e_l\\[3mm]
			\Riem^g(e_i,e_j,e_k,e_l) = \scal{\Riem^g(e_i,e_j)e_k}{e_l}_g.
		\end{cases} 
	\end{align}
	Using the coordinates, the Riemann tensor can be expressed in terms of the Christoffel symbols defined in \eqref{eq:Christoffel} as
	\begin{align}\label{eq:Riem_coordinates}
		(\Riem^g)_{ijk}^{\ \ \ l} =  \dr_i ({^{g} \Gamma}^l_{jk}) - \dr_j({^{g}\Gamma}^l_{ik}) + ({^{g}\Gamma}^m_{jk})\, ({^{g}\Gamma}^l_{im})-({^{g}\Gamma}^m_{ik})\, ({^{g}\Gamma}^l_{jm}).
	\end{align}
	
	\paragraph*{Laplace--Beltrami operator.}
	Let $(e_1,\ldots,e_n)$ be an orthonormal frame for $g$. Then the Laplace--Beltrami operator for $g$ is defined as follows: for any $f\in C^{\infty}(\Omega)$,
	\begin{align*}
		\lap_g f = \tr_g\left(\Hess_g\, f\right)= \sum_{i=1}^n \left(\g^g_{e_i} \g^g_{e_i} f - \g^g_{\g^g_{e_i} e_i} f \right).
	\end{align*}
	In coordinates, we have 
	\begin{align}\label{eq:Laplacian_coordinates}
		\lap_g f = \frac{1}{\sqrt{\det (g_{ij})}}\, \dr_{\alpha}\left( \sqrt{\det (g_{ij})}\, g^{\alpha\beta}\, \dr_{\beta} f \right).
	\end{align}

	\section{Metrics on $\R^n$ of class $W^{2,\left(\frac{n}{2},1\right)}$}\label{sec:Proofmain}
	
	In order to prove \Cref{th:diff_Rn}, we proceed by a continuity argument reminiscent to the one of Uhlenbeck \cite{uhlenbeck1982} to construct a Coulomb gauge. We first prove in \Cref{sec:Apriori}, that if a metric on $\R^n$ is close to the Euclidean metric, then we can construct a Coulomb frame using the continuity argument of Uhlenbeck \cite{uhlenbeck1982}. Using this frame, we construct the coordinates by solving a well-chosen elliptic system. Finally, we prove \Cref{th:diff_Rn} in \Cref{sec:Continuity} by another continuity argument.

    \begin{remark}\label{rk:Uhlenbeck}
        Uhlenbeck's gauge extraction also works in the setting of connections with curvatures in $L^{\left(\frac{n}{2},1\right)}$. Indeed, the two necessary ingredients for \cite{uhlenbeck1982} holds in this setting. The first one is to know that given a solution $A\in W^{1,(n,1)}$ to the system
        \begin{align*}
            \begin{cases}
                dA  = F - A\wedge A \in L^{\left(\frac{n}{2},1\right)} & \text{ in }\B,\\[2mm]
                d^{*_{\geu}} A = 0  & \text{ in }\B,\\[2mm]
                A(\nu) = 0 & \text{ on } \dr \B,
            \end{cases}
        \end{align*}
        then we have 
        \begin{align*}
            \|A\|_{W^{1,\left(\frac{n}{2},1\right)}(\B)} \leq C(n)\, \|F - A\wedge A\|_{L^{\left(\frac{n}{2},1\right)}(\B)}. 
        \end{align*}
        This is obtained by interpolation and provides an alternative to \cite[Lemma 2.5]{uhlenbeck1982} in this setting. Using that $W^{1,(n,1)}(\B^n)\hookrightarrow C^0(\B^n)$, we can recover Lemma 2.7 in \cite{uhlenbeck1982}. Thus, Theorem 2.1 in \cite{uhlenbeck1982} holds for $\Riem \in L^{\left( \frac{n}{2},1\right)}(\B^n)$.
    \end{remark}
	
	\subsection{A priori estimates for metrics on $\R^n$}\label{sec:Apriori}

	The goal of this section is to prove the most technical part of the continuity argument. Namely, we prove the following result. 
	
	\begin{proposition}\label{pr:apriori_global}
		Let $n\geq 3$ be an integer. There exist $\eps_0>0$ and $C_0>0$ depending only on $n$ such that the following holds for any $\eps\in(0,\eps_0)$. 
		Let $g$ be a weak metric on $\R^n$ such that there exist coordinates on $\R^n$ satisfying $g_{ij}\in \bar{W}^{2,\left(\frac{n}{2},1\right)}(\R^n,\geu)$ such that for any $1\leq i,j,k,l\leq n$, it holds 
		\begin{align}
			& \|g_{ij}-\delta_{ij}\|_{L^{\infty}(\R^n,\geu)} + \|\dr_k g_{ij} \|_{L^{(n,1)}(\R^n,\geu) } + \|\dr^2_{kl} g_{ij} \|_{L^{\left(\frac{n}{2},1\right)}(\R^n,\geu) } \leq \eps, \label{eq:Initial_g1}\\[2mm]
			& \|\Riem^g\|_{L^{\left(\frac{n}{2},1\right)}(\R^n,g)}\leq \eps. \label{eq:Initial_R1}
		\end{align} 
		Then there exist a $C^1$-diffeomorphism $y\in \Diff(\R^n)$ and a coframe $(\omega^1,\ldots,\omega^n)$ for $g$ on $\R^n$ such that 
		\begin{align*}
			& \qquad 1) \qquad dy\in\bar{W}^{2,\left(\frac{n}{2},1\right)}(\R^n), \\[3mm]
			& \qquad 2) \qquad \left\| \left( (y^{-1})^*g \right)_{\alpha\beta} - \delta_{\alpha\beta} \right\|_{\bar{W}^{2,\left(\frac{n}{2},1\right)}(\R^n,\geu)}  \leq C_0\, \|\Riem^g\|_{L^{\left(\frac{n}{2},1\right)}(\R^n,g)}, \\[4mm]
			& \qquad 3) \qquad \| \omega^i_j\|_{L^{(n,1)}(\R^n,g)} + \|\g^g \omega^i_j \|_{L^{\left(\frac{n}{2},1\right)}(\R^n,g)} \leq C_0\, \|\Riem^g\|_{L^{\left(\frac{n}{2},1\right)}(\R^n,g)}, \\[4mm]
			& \qquad 4) \qquad \|dy^i - \omega^i\|_{\bar{W}^{2,\left(\frac{n}{2},1\right)}(\R^n,g)} \leq C_0\, \|\Riem^g\|_{L^{\left(\frac{n}{2},1\right)}(\R^n,g)}.
		\end{align*}
	\end{proposition}
	
	In order to prove Proposition \ref{pr:apriori_global}, we first prove the following localized version on a ball $\B_s$ for some $s>0$ fixed, which might have an independent interest.
	\begin{proposition}\label{pr:apriori_local}
		There exist $\eps_0>0$ and $C_0>0$ depending only on $n$ such that the following holds for any $\eps\in(0,\eps_0)$. 
		Let $s>0$ and $g$ be a weak metric on $\B^n_s$ such that $g_{ij}\in W^{2,\left(\frac{n}{2},1\right)}(\B_s,\geu)$ and for any $1\leq i,j,k,l\leq n$, it holds 
		\begin{align}
			& \|g_{ij}-\delta_{ij}\|_{L^{\infty}(\B_s,\geu)} + \|\dr_k g_{ij} \|_{L^{(n,1)}(\B_s,\geu) } + \|\dr^2_{kl} g_{ij}\|_{L^{\left(\frac{n}{2},1\right)}(\B_s,\geu) } \leq \eps, \label{eq:Initial_g}\\[2mm]
			& \|\Riem^g\|_{L^{\left(\frac{n}{2},1\right)}(\B_s,g)}\leq \eps. \label{eq:Initial_R}
		\end{align} 
		Then there exist a coframe $(\omega^{s,1},\ldots,\omega^{s,n})$ for $g$ on $\B_s$ of class $\bar{W}^{2,\left(\frac{n}{2},1\right)}(\B_s)$ such that its connection forms verifies $d^{*_{\geu}}\omega^{s,i}_{\ j}=0$ in $\B_s$ with $\omega^{s,i}_{\ j}(\nu)=0$ on $\dr\B_s$, together with a $C^1$-diffeomorphism $y_s\colon \B_s \to y_s(\B_s)$ such that 
		\begin{align*}
			& \|dy^{s,i}-\omega^{s,i}\|_{L^{\infty}(\B_s,\geu)} + \left\|\g^{\geu}\left( dy^{s,i}-\omega^{s,i} \right) \right\|_{L^{\left(n,1\right)}(\B_s,\geu)} + \left\|(\g^{\geu})^2 \left( dy^{s,i}-\omega^{s,i} \right) \right\|_{L^{\left(\frac{n}{2},1\right)}(\B_s,\geu)}  \\[2mm]
			& \leq C_0\, \|\Riem^g\|_{L^{\left(\frac{n}{2},1\right)}(\B_s,g)}. 
		\end{align*}
		We also have
		\begin{align*}
			\|\g^{\geu} \omega^{s,i} \|_{L^{(n,1)}(\B_s,\geu)} \leq C_0\, \eps.
		\end{align*}
		The map	$y^s\colon \B_s\to y^s(\B_s)$ is moreover uniformly bi-Lipschitz (the bi-Lipschitz constant is independent of $s$):
		\begin{align*}
			\forall a,b\in \B_s,\qquad \frac{|a-b|}{100} \leq | y^s(a) - y^s(b)| \leq 100\, |a-b|.
		\end{align*}
	\end{proposition}

	\begin{proof} 
		If $\eps_0>0$ is small enough, we obtain from \eqref{eq:Initial_g} the pointwise inequality between matrices
		\begin{align}\label{eq:Pointwise_g}
			\frac{1}{2}\, \delta_{ij} \leq g_{ij} \leq 2\, \delta_{ij} \qquad \text{ on }\R^n.
		\end{align}
		Therefore, we obtain the estimate
		\begin{align}\label{eq:comparison_Riem}
			\|\Riem^g\|_{L^{\left(\frac{n}{2},1\right)}(\B_s,\geu)} \leq C(n)\, \|\Riem^g\|_{L^{\left(\frac{n}{2},1\right)}(\B_s,g)}\leq C(n)\, \eps.
		\end{align}
		Let $(\alpha^1,\ldots,\alpha^n)$ be the coframe for $g$ obtain from the basis $(dx^1,\ldots,dx^n)$ by a Gram--Schmidt process. Since the coefficients $g_{ij}$ lie in $\bar{W}^{2,\left(\frac{n}{2},1\right)}$, the 1-forms $\alpha^i$ also lie in $\bar{W}^{2,\left(\frac{n}{2},1\right)}$. Then the curvature of $g$ is given by \eqref{eq:Curv_frame} using the connections forms $\alpha^i_j$ defined through the structure equations $d\alpha^i = \alpha^i_j\wedge \alpha^j$ and we have by \eqref{eq:curvatures} that
		\begin{align*}
			\left\| d\alpha^i_j + \alpha^i_k\wedge \alpha^k_j \right\|_{L^{\left(\frac{n}{2},1\right)}(\B_s,\geu)} \leq C(n)\, \eps_0.
		\end{align*}
		If $\eps_0>0$ is small enough (depending only on $n$), we can apply \cite[Theorem 1.3]{uhlenbeck1982} (see Remark \ref{rk:Uhlenbeck}) and we obtain a frame $\left( e_{s,1},\ldots,e_{s,n}\right)$ with an associated coframe $\left(\omega^{s,1},\ldots,\omega^{s,n}\right)$ satisfying the system
        \begin{align}
            & d\omega^{s,i} = \omega^{s,j}\wedge \omega^{s,i}_{\ j},  \label{eq:Structure}\\[2mm]
			& d^{*_{\geu}}\omega^{s,i}_{\ j} = 0 \qquad \text{in }\B_s,  \nonumber \\[2mm]
			& \omega^{s,i}_{\ j}(\nu) = 0 \qquad \text{on }\dr \B_s. \nonumber
        \end{align}
        Moreover, we have the following estimate (using the scale invariance of the involved Sobolev norms\footnote{We can apply Uhlenbeck result to the coframe $(\omega^{s,1}(s\cdot),\ldots,\omega^{s,n}(s\cdot))$ which is orthonormal for the metric $g(s\cdot)$ on $\B_1$. The connection form is then given by $s\, \omega^{s,i}_{\ j}(s\cdot)$. Hence, we have $\| s\, \omega^{s,i}_{\ j}(s\cdot) \|_{L^{(n,1)}(\B_1,\geu)} = \|\omega^{s,i}_{\ j}\|_{L^{(n,1)}(\B_s,\geu)}$ and $\left\| \g^{\geu}\left( s\, \omega^{s,i}_{\ j}(s\cdot) \right) \right\|_{L^{\left( \frac{n}{2},1\right)}(\B_1,\geu)} = \|\g^{\geu}\omega^{s,i}_{\ j}\|_{L^{\left(\frac{n}{2},1\right)}(\B_s,\geu)}$.})
		\begin{align*}
			& \sum_{1\leq i,j\leq n} \left( \|\omega^{s,i}_{\ j}\|_{L^{(n,1)}(\B_s,\geu)} + \|\g^{\geu} \omega^{s,i}_{\ j}\|_{L^{\left(\frac{n}{2},1\right)}(\B_s,\geu)} \right) \leq C(n)\, \|d\omega^{s,i}_{\ j} + \omega^{s,i}_{\ k} \wedge \omega^{s,k}_{\ j}\|_{L^{\left(\frac{n}{2},1\right)}(\B_s,\geu)}. 
		\end{align*}
		By \eqref{eq:comparison_Riem}, the above estimate provides the following.
		\begin{align}\label{eq:Omegaij}
			\sum_{1\leq i,j\leq n} \left( \|\omega^{s,i}_{\ j}\|_{L^{(n,1)}(\B_s,\geu)} + \|\g^{\geu} \omega^{s,i}_{\ j}\|_{L^{\left(\frac{n}{2},1\right)}(\B_s,\geu)} \right) \leq C(n)\, \|\Riem^g\|_{L^{\left(\frac{n}{2},1\right)}(\B_s,g)}. 
		\end{align}
		We also deduce from \eqref{eq:Pointwise_g} and the equalities $g(\omega^{s,i},\omega^{s,i})=1$, that the following pointwise inequalities hold a.e.
		\begin{align}\label{eq:Pointwise_frame}
			\frac{1}{2} \leq |\omega^{s,i}|_{\geu} \leq 2.
		\end{align}
        We now prove that $\g^{\geu}\omega^{s,i}$ is controlled by $\eps$ and that $d\omega^{s,i}$ is controlled by $\|\Riem^g\|_{L^{\left(\frac{n}{2},1\right)}(\B_s,g)}$.
        
		\begin{claim}
			Up to reducing $\eps_0$ depending only on $n$, it holds 
			\begin{align}\label{eq:Est_der_omegai}
				\|d\omega^{s,i}\|_{L^{(n,1)}(\B_s,\geu)} + \|\g^{\geu}(d\omega^{s,i})\|_{L^{\left( \frac{n}{2},1\right)}(\B_s,\geu)} \leq C(n)\, \|\Riem^g\|_{L^{\left(\frac{n}{2},1\right)}(\B_s,g)}.
			\end{align}
			We also have 
			\begin{align}\label{eq:Der_frame}
				\|\g^{\geu} \omega^{s,i} \|_{L^{(n,1)}(\B_s,\geu)} + \left\| (\g^{\geu})^2 \omega^{s,i} \right\|_{L^{\left(\frac{n}{2},1\right)}(\B_s,\geu)} \leq C(n)\, \eps.
			\end{align}
		\end{claim}
		
		\begin{proof} 
			We estimate the exterior derivative $d\omega^{s,i}$ using \eqref{eq:Pointwise_frame} and \eqref{eq:Structure}
			\begin{align}\label{eq:Domegai0}
				\|d\omega^{s,i}\|_{L^{(n,1)}(\B_s,\geu)} \leq \sum_j\|\omega^{s,i}_{\ j}\|_{L^{(n,1)}(\B_s,\geu)}\, \|\omega^{s,j}\|_{L^{\infty}(\B_s,\geu)} \leq  C(n)\, \|\Riem^g\|_{L^{\left(\frac{n}{2},1\right)}(\B_s,g)}.
			\end{align}
			By differentiating \eqref{eq:Structure}, its derivatives verifies
			\begin{align} 
				& \|\g^{\geu} (d\omega^{s,i})\|_{L^{\left( \frac{n}{2},1 \right)}(\B_s,\geu)}  \nonumber \\[2mm]
				& \leq \sum_j \left( \|\g^{\geu} \omega^{s,i}_{\ j}\|_{L^{\left(\frac{n}{2},1\right)}(\B_s,\geu)}\, \|\omega^{s,j}\|_{L^{\infty}(\B_s,\geu)} + \|\omega^{s,i}_{\ j}\|_{L^{(n,2)}(\B_s,\geu)}\, \|\g^{\geu} \omega^{s,j} \|_{L^{(n,2)}(\B_s,\geu)} \right) \nonumber \\[2mm]
				& \leq C\, \|\Riem^g\|_{L^{\left(\frac{n}{2},1\right)}(\B_s,g)}\left(1 + \sum_j \|\g^{\geu} \omega^{s,j}\|_{L^{(n,2)}(\B_s,\geu)} \right). \label{eq:Domegai}
			\end{align}
			Concerning the last term, \eqref{eq:Initial_g} together with \eqref{eq:covariant_form} and \eqref{eq:Pointwise_frame} imply that 
			\begin{align*}  
				\|\g^{\geu} \omega^{s,j}\|_{L^{(n,1)}(\B_s,\geu)} \leq C(n)\, \eps + C(n)\, \|\g^g \omega^{s,j}\|_{L^{(n,1)}(\B_s,\geu)}.
			\end{align*}
			Moreover, it holds $\g^g\omega^i = -\omega^i_j\otimes \omega^j$ thanks to the following computation
			\begin{align}
				(\g^g_{e_k} \omega^i) (e_j) & = e_k\left(\omega^i(e_j)\right) - \omega^i(\g^g_{e_k} e_j) \nonumber \\[2mm]
				& = - \scal{\g^g_{e_k} e_j}{e_i} \nonumber \\[2mm]
				& = - \omega^i_l(e_k)\, \omega^l(e_i). \label{eq:gomegai}
			\end{align}
			Hence, we obtain 
			\begin{align*}
				\|\g^{\geu} \omega^{s,j}\|_{L^{(n,1)}(\B_s,\geu)} \leq C(n)\, \eps + C(n)\, \sum_{1\leq k,l\leq n} \| \omega^{s,k}_{\ l}\|_{L^{(n,1)}(\B_s,\geu)}\, \|\omega^{s,l}\|_{L^{\infty}(\B_s,\geu)}.
			\end{align*}
			Thanks to \eqref{eq:Pointwise_frame} and \eqref{eq:Omegaij}, we obtain 
			\begin{align*}
				\|\g^{\geu} \omega^{s,j}\|_{L^{(n,1)}(\B_s,\geu)} \leq C(n)\, \eps.
			\end{align*}
			Plugging this into \eqref{eq:Domegai}, we obtain 
			\begin{align}
				\|\g^{\geu} (d\omega^{s,i})\|_{L^{\left( \frac{n}{2},1 \right)}(\B_s,\geu)} \leq C\, \|\Riem^g\|_{L^{\left(\frac{n}{2},1\right)}(\B,g)}. \label{eq:Domegai1}
			\end{align}
			Differentiating once more \eqref{eq:gomegai}, we have
			\begin{align*}
				\|(\g^{\geu})^2 \omega^{s,j}\|_{L^{\left( \frac{n}{2},1\right)}(\B_s,\geu)} & \leq C(n)\, \eps + \sum_{1\leq k,l\leq n} \| \omega^{s,k}_{\ l}\|_{\bar{W}^{1,\left(\frac{n}{2},1\right)}(\B_s,\geu)}\, \|\omega^{s,l}\|_{\bar{W}^{1,(n,1)}(\B_s,\geu)} \\[2mm]
				& \leq C(n)\, \eps.
			\end{align*}
		\end{proof}
		
		Given $1\leq i\leq n$, we define $y^{s,i}\colon \B_s\to \R$ to be the solution to\footnote{The compatibility condition is verified by integration by parts:
			\begin{align}
				\int_{\B} (d^{*_g} \omega^{s,i})\, d\vol_g = \int_{\B} (\lap_g y^{s,i})\, d\vol_g = \int_{\dr \B} \omega^{s,i}(\nu)\, d\vol_g.
			\end{align}
		}
		\begin{align*}
			\begin{cases}
				\displaystyle \lap_g y^{s,i} = d^{*_g} \omega^{s,i} & \text{ in }\B_s,\\[2mm]
				\dr_{\nu} y^{s,i} = \omega^{s,i}(\nu) & \text{ on }\dr \B_s,\\[2mm]
				\displaystyle \int_{\B_s} y^{s,i} = 0.
			\end{cases}
		\end{align*}
		To estimate $dy^{s,i}$, we formulate the above system as a system on the differential form $dy^{s,i} - \omega^{s,i}$. It satisfies the following system
        \begin{align}\label{eq:Syst_y_omega}
			\begin{cases}
				d\left(dy^{s,i} - \omega^{s,i}\right) = -d\omega^{s,i} & \text{in }\B_s,\\[2mm]
				\displaystyle d^{*_g}\left( dy^{s,i} - \omega^{s,i} \right) = 0 & \text{in }\B_s,\\[2mm]
				\left(dy^{s,i}-\omega^{s,i}\right)(\nu) = 0 & \text{on }\dr\B_s.
			\end{cases}
		\end{align}
		We obtain the following estimates by Hodge decomposition.
		\begin{claim}
			For any $s>0$ and $i\in\{1,\ldots,n\}$, it holds 
			\begin{align}
				& \|dy^{s,i}-\omega^{s,i}\|_{L^{\infty}(\B_s,\geu)} + \left\|\g^{\geu}\left( dy^{s,i}-\omega^{s,i} \right) \right\|_{L^{\left(n,1\right)}(\B_s,\geu)} + \left\|(\g^{\geu})^2 \left( dy^{s,i}-\omega^{s,i} \right) \right\|_{L^{\left(\frac{n}{2},1\right)}(\B_s,\geu)} \nonumber \\[2mm]
				& \leq C(n)\, \|\Riem^g\|_{L^{\left(\frac{n}{2},1\right)}(\B_s,g)}. \label{eq:Unif_dy}
			\end{align}
		\end{claim}
		
		\begin{proof} 
			For simplicity, we consider the case $s=1$ and drop the index $s$. The result is obtained by change of scale with $g\leftrightarrow g(s\cdot)$, $\omega^i \leftrightarrow \omega^{s,i}(s\cdot)$ and $y^i \leftrightarrow s^{-1} y^i(s\cdot)$. In this setting, the norms $W^{2,\left(\frac{n}{2},1\right)}(\B^n)$ and $\bar{W}^{2,\left(\frac{n}{2},1\right)}(\B^n)$ are equivalent. However, after the rescaling, powers of $s$ arise in the norm $W^{2,\left(\frac{n}{2},1\right)}(\B_s)$, but not in the norm $\bar{W}^{2,\left(\frac{n}{2},1\right)}(\B_s)$. This observation provides the constant independent of $s$ in \eqref{eq:Unif_dy}. Indeed, we have 
            \begin{align*}
                & \left\| d\left( s^{-1}y^{s,i}(s\cdot) \right) - \omega^{s,i}(s\cdot) \right\|_{W^{2,\left(\frac{n}{2},1\right)}(\B^n)} \\[2mm]
                & = s^{-2}\|dy^{s,i}-\omega^{s,i}\|_{L^{\left(\frac{n}{2},1\right)}(\B_s)} + s^{-1}\|\g^{\geu}\left( dy^{s,i}-\omega^{s,i}\right) \|_{L^{\left(\frac{n}{2},1\right)}(\B_s)} + \|(\g^{\geu})^2\left( dy^{s,i}-\omega^{s,i}\right) \|_{L^{\left(\frac{n}{2},1\right)}(\B_s)}.
            \end{align*}
			However, we have
            \begin{align*}
                 \left\| d\left( s^{-1}y^{s,i}(s\cdot) \right) - \omega^{s,i}(s\cdot) \right\|_{\bar{W}^{2,\left(\frac{n}{2},1\right)}(\B^n)} =  \left\| dy^{s,i}- \omega^{s,i} \right\|_{\bar{W}^{2,\left(\frac{n}{2},1\right)}(\B^n_s)}.
            \end{align*}
            
			We consider the Hodge decomposition of $dy^i - \omega^i$ with respect to the Euclidean metric. There exists $\alpha\in W^{2,(n,1)}(\B)$, $\beta\in W^{2,(n,1)}(\B;\Lambda^2\R^n)$ such that $d\beta=0$ and $h\in W^{1,(n,1)}(\B;\Lambda^1\R^n)$ such that (see for instance \cite[Theorem 6.9 (i)]{csato2012}\footnote{We obtain this decomposition as follows. We first define $\alpha\colon \B\to \R$ as the solution to
				\begin{align*}
					\begin{cases}
						\lap_{\geu}\alpha = d^{*_{\geu}} (\omega^i-dy^i) & \text{ in }\B,\\
						\dr_{\nu}\alpha = 0 & \text{ on }\s^{n-1},\\
						\int_{\B} \alpha=0.
					\end{cases}
				\end{align*}
			    We obtain the estimate by standard elliptic regularity. Then we define $\beta\colon \B\to \Lambda^2\R^n$ to be the solution to 
				\begin{align*}
					\begin{cases}
						\lap_{\geu}\beta = d(\omega^i-dy^i) & \text{ in }\B,\\
						d\beta=0 & \text{ in }\B,\\
						\beta\llcorner \nu = 0 & \text{ on }\s^{n-1},\\
						(d^{*_{\geu}}\beta)\llcorner \nu = 0 & \text{ on }\s^{n-1}.
					\end{cases}
				\end{align*}
				We obtain $h \coloneq dy^i - \omega^i - d\alpha- d^{*_{\geu}}\beta $ which satisfies $dh=0$, $d^{*_{\geu}}h=0$ and $h\llcorner\nu = (dy^i-\omega^i)\llcorner \nu$.
			})
			\begin{align*}
				dy^i - \omega^i = d\alpha + d^{*_{\geu}}\beta + h \qquad \text{ in }\B.
			\end{align*}
			Moreover, $h$ is harmonic with homogeneous Neumann boundary condition: $h(\nu)=0$ on $\dr\B$, $dh=0$ and $d^{*_{\geu}}h=0$ in $\B$. The following estimates are verified:
			\begin{align*}
				\begin{cases}
					\|\beta\|_{W^{3,\left(\frac{n}{2},1\right)}(\B,\geu)} \leq C(n)\, \|d (\omega^i-dy^i)\|_{W^{1,\left(\frac{n}{2},1\right)}(\B,\geu)}, \\[2mm]
					\|\alpha\|_{W^{3,\left(\frac{n}{2},1\right)}(\B,\geu)}  \leq C(n)\, \|d^{*_{\geu}} (\omega^i-dy^i)\|_{W^{1,\left(\frac{n}{2},1\right)}(\B,\geu)}.
				\end{cases}
			\end{align*}

			\textit{Step 1: $h=0$.} \\
			Since $dh=0$, we can write $h=d\tilde{h}$, where $h$ is a harmonic function such that $\dr_{\nu}h=0$ on $\dr\B$. Hence $\tilde{h}$ is constant and $h=0$. \\
			
			\textit{Step 2: Estimate of $\beta$.}\\
			Concerning $\beta$, we have, using \eqref{eq:Est_der_omegai}, that
			\begin{align*}
				\|\beta\|_{W^{3,\left(\frac{n}{2},1\right)}(\B,\geu)} \leq C(n)\, \|d\omega^i\|_{W^{1,\left(\frac{n}{2},1\right)}(\B,\geu)} \leq C(n)\, \|\Riem^g\|_{L^{\left(\frac{n}{2},1\right)}(\B,g)}.
			\end{align*}
			
			\textit{Step 3: Estimate of $\alpha$.}\\
			Concerning $\alpha$, we have
			\begin{align*}
				\|\alpha\|_{W^{3,\left(\frac{n}{2},1\right)}(\B,\geu)} & \leq C(n)\, \|d^{*_{\geu}} (\omega^i-dy^i)\|_{W^{1,\left(\frac{n}{2},1\right)}(\B,\geu)}.
			\end{align*}
			Thanks to \eqref{eq:codiff_cov}, we have the pointwise estimates
            \begin{align}
                & \left| d^{*_{\geu}} (\omega^i-dy^i) -d^{*_g}(\omega^i-dy^i) \right|_{\geu} \nonumber \\[2mm]
                &\leq C(n)\, \sum_{\alpha,\beta,\gamma} \left(  \left|g^{\alpha\beta} - \delta^{\alpha\beta} \right|\, \left| \dr_{\gamma}(\omega^i-dy^i)\right|_{\geu} + |\Gamma_{\alpha\beta}^{\gamma}|\, |\omega^i-dy^i|_{\geu} \right) .\label{eq:domega_dy}
            \end{align}
            Differentiating once more, we have
            \begin{align}\label{eq:domega_dy2}
            \begin{aligned}
                & \left| \dr_j\left( d^{*_{\geu}} (\omega^i-dy^i) -d^{*_g}(\omega^i-dy^i) \right) \right|_{\geu} \\[2mm]
                &\leq C(n)\, \sum_{\alpha,\beta,\gamma} \Big(  \left|\dr_j g^{\alpha\beta} \right|\, \left| \dr_{\gamma}(\omega^i-dy^i)\right|_{\geu} 
                + \left|g^{\alpha\beta} - \delta^{\alpha\beta} \right|\, \left| \dr^2_{j\gamma}(\omega^i-dy^i)\right|_{\geu} \\[2mm]
                &\qquad  + |\dr_j \Gamma_{\alpha\beta}^{\gamma}|\, |\omega^i-dy^i|_{\geu}
                +  |\Gamma_{\alpha\beta}^{\gamma}|\, |\dr_j(\omega^i-dy^i)|_{\geu}
                \Big) .
                \end{aligned}
            \end{align}
            By  \eqref{eq:Initial_g} and \eqref{eq:Christoffel}, we deduce from \eqref{eq:domega_dy}
            \begin{align*}
                & \left| d^{*_{\geu}} (\omega^i-dy^i) -d^{*_g}(\omega^i-dy^i) \right|_{\geu} \\[2mm]
                &\leq C(n)\, \sum_{\alpha,\beta,\gamma} \left(  \left|g^{\alpha\beta} - \delta^{\alpha\beta} \right|\, \left| \dr_{\gamma}(\omega^i-dy^i)\right|_{\geu} + |\dr_{\gamma} g_{\alpha\beta} |\, |\omega^i-dy^i|_{\geu} \right) .
            \end{align*}
            With one more derivative, we obtain from \eqref{eq:domega_dy2}
            \begin{align*}
                & \left| \dr_j\left( d^{*_{\geu}} (\omega^i-dy^i) -d^{*_g}(\omega^i-dy^i) \right) \right|_{\geu} \\[2mm]
                &\leq C(n)\, \sum_{\alpha,\beta,\gamma} \Big(  \left|\dr_j g_{\alpha\beta} \right|\, \left| \dr_{\gamma}(\omega^i-dy^i)\right|_{\geu} 
                + \left|g^{\alpha\beta} - \delta^{\alpha\beta} \right|\, \left| \dr^2_{j\gamma}(\omega^i-dy^i)\right|_{\geu}  + |\dr^2_{j\gamma} g_{\alpha\beta}|\, |\omega^i-dy^i| \Big) .
            \end{align*}
            Using \eqref{eq:Initial_g}, we obtain
			\begin{align*}
				\|\alpha\|_{W^{3,\left(\frac{n}{2},1\right)}(\B,\geu)} & \leq C(n)\, \|d^{*_g} (\omega^i-dy^i)\|_{W^{1,\left(\frac{n}{2},1\right)}(\B,\geu)} + C(n)\, \eps_0\, \|\omega^i - dy^i\|_{W^{1,\left(\frac{n}{2},1\right)}(\B,\geu)}.
			\end{align*}
			By \eqref{eq:Syst_y_omega}, the first term vanishes and we obtain 
			\begin{align*}
				\|\alpha\|_{W^{3,\left(\frac{n}{2},1\right)}(\B,\geu)} & \leq C(n)\, \eps_0\, \|\omega^i - dy^i\|_{W^{2,\left(\frac{n}{2},1\right)}(\B,\geu)}.
			\end{align*}
			
			\textit{Step 4: Conclusion.}\\
			Combining Steps 1 to 3, we obtain
			\begin{align*}
				\|dy^i-\omega^i\|_{W^{2,\left(\frac{n}{2},1\right)}(\B,\geu)}\leq C(n)\, \|\Riem^g \|_{L^{\left(\frac{n}{2},1\right)}(\B,g)} + C(n)\, \eps_0\, \|\omega^i - dy^i\|_{W^{2,\left(\frac{n}{2},1\right)}(\B,\geu)}.
			\end{align*}
			Up to reducing $\eps_0$, we have $C(n)\, \eps_0<\frac{1}{2}$, and thus
			\begin{align*}
				\|dy^i-\omega^i\|_{W^{2,\left(\frac{n}{2},1\right)}(\B,\geu)}\leq C(n)\, \|\Riem^g \|_{L^{\left(\frac{n}{2},1\right)}(\B,g)}.
			\end{align*}
			Using the Sobolev injection $W^{2,\left(
				\frac{n}{2},1\right)}(\B,\geu)\hookrightarrow W^{1,(n,1)}(\B,\geu) \hookrightarrow L^{\infty}(\B)$, we obtain 
			\begin{align*}
				\| dy^i - \omega^i\|_{L^{\infty}(\B,\geu)} + \left\| \g^{\geu}\left( dy^i - \omega^i \right) \right\|_{L^{(n,1)}(\B,\geu)} & \leq C(n)\, \|dy^i-\omega^i\|_{W^{2,\left(\frac{n}{2},1\right)}(\B,\geu)} \\[2mm]
				& \leq C(n)\, \|\Riem^g\|_{L^{\left(\frac{n}{2},1\right)}(\B,g)}.
			\end{align*}
		\end{proof}

		For $\eps_0>0$ small enough, we obtain by \eqref{eq:Initial_R} and \eqref{eq:Unif_dy} that $(dy^{s,1},\ldots,dy^{s,n})$ is a basis of 1-forms. Hence the map $y^s\coloneq (y^{s,1},\ldots,y^{s,n})\colon \B_s\to \R^n$ is an immersion.\footnote{Here we use the injection $W^{1,(n,1)}(\B^n)\hookrightarrow C^0(\B^n)$, rather than just the $L^{\infty}$ estimate of \eqref{eq:Unif_dy}. Indeed, the map $dy^s$ is close in the $C^0$ topology to a given continuous coframe. This crucial information allows the use of the inverse function theorem, which implies that $y^s$ is indeed a $C^1$-immersion. This would be false if we had only the Lipschitz regularity. In dimension $n=2$, the map $z\mapsto z^2/|z|$ is a counterexample, see for instance \cite{Lan2025,Plot}.} 
        
		\begin{claim}\label{cl:y_bilip}
			Up to reducing $\eps_0$ (depending only on $n$), the maps $y^s\colon \B_s\to y^s(\B_s)$ define coordinate systems which are uniformly bi-Lipschitz in $s$:
			\begin{align}\label{eq:y_biLip}
				\forall s>0,\ \forall a,b\in \B_s,\qquad \frac{|a-b|}{100} \leq | y^s(a) - y^s(b)| \leq 100\, |a-b|.
			\end{align}
		\end{claim}
		\begin{proof}
			For $1\leq i\leq n$, we denote $(e_{s,1},\ldots,e_{s,n})$ the dual frame of $(\omega^{s,1},\ldots,\omega^{s,n})$ and define
			\begin{align*}
				\overline{e_{s,i}} \coloneq \fint_{\B_s} e_{s,i}(x)\, dx.
			\end{align*}
			Thanks to \eqref{eq:Der_frame} and \eqref{eq:Initial_g}, we have for all $1\leq i\leq n$ that
			\begin{align}\label{eq:e_constant}
				\left\| e_{s,i} - \overline{e_{s,i}}\right\|_{L^{\infty}(\B_s,\geu)} \leq C(n)\, \eps_0.
			\end{align}
			Indeed, the estimate \eqref{eq:Pointwise_g} implies that $\frac{1}{2} < |e_{s,i}|_{\geu} < 2$. Combining \eqref{eq:der_frame} and \eqref{eq:Omegaij}, we obtain 
			\begin{align*}
				\left\| \g^{\geu} e_{s,i} \right\|_{L^{(n,1)}(\B_s,\geu)} \leq C(n)\, \eps_0.
			\end{align*}
			We obtain \eqref{eq:e_constant} by the Sobolev inequality \eqref{eq:Sobolev_avg}. Hence, if $\eps_0=\eps_0(n)$ is small enough, we deduce from \eqref{eq:Pointwise_g} and \eqref{eq:e_constant} that the family $\left(\overline{e_1},\ldots,\overline{e_n}\right)$ defines a constant basis of $\R^n$ such that 
			\begin{align}\label{eq:Unif_basis}
				\forall x\in \R^n,\qquad \frac{1}{10}\, |x|_{\geu} \leq \left|\begin{pmatrix}
					\scal{\overline{e_1}}{x}_{\geu} \\
					\vdots\\
					\scal{\overline{e_n}}{x}_{\geu}
				\end{pmatrix}\right| \leq 10\, |x|_{\geu}.
			\end{align}
			Moreover, \eqref{eq:Unif_dy} together with \eqref{eq:Initial_g} imply that 
			\begin{align}\label{eq:gy}
				\left\|\g^{\geu} y^{s,i} - \overline{e_{s,i}} \right\|_{L^{\infty}(\B_s,\geu)} \leq C(n)\, \eps_0.
			\end{align}
			If $\eps_0>0$ is small enough, we obtain that $y^{s,i}$ is uniformly bi-Lipschitz. Indeed, if $a,b\in \B_s$, then 
			\begin{align*}
				y^{s,i}(a)-y^{s,i}(b) = \int_0^1 dy^{s,i}( (1-t)b + ta )(a-b)\, dt.
			\end{align*}
			Hence, we obtain 
			\begin{align*}
				\left| y^s(a)-y^s(b) - \begin{pmatrix}
					\scal{\overline{e_{s,1}}}{a-b}_{\geu} \\
					\vdots\\
					\scal{\overline{e_{s,n}}}{a-b}_{\geu}
				\end{pmatrix} \right| \leq C(n)\, \eps_0\, |a-b|.
			\end{align*}
			As a consequence of \eqref{eq:Unif_basis}, we obtain that if $\eps_0>0$ is small enough, then
			\begin{align*}
				\frac{1}{100}\, |a-b| \leq |y^s(a)-y^s(b)| \leq 100\, |a-b|.
			\end{align*}
			Hence the map $y^s\colon \B_s\to y^s(\B_s)$ is injective, bi-Lipschitz and thus, defines a coordinate system.
		\end{proof}

		Concerning the metric $g$, it holds 
		\begin{align*}
			g=\sum_{i=1}^n \omega^{s,i}\otimes \omega^{s,i} \qquad \text{ on }\B_s.
		\end{align*}
		Then \eqref{eq:Unif_dy}, \eqref{eq:Omegaij} and \eqref{eq:Pointwise_frame} implies that $\|dy^s\|_{\bar{W}^{2,\left(\frac{n}{2},1\right)}(\B_s,\geu)} \leq C(n)$. Hence, for any $1\leq \alpha,\beta\leq n$, 
		\begin{align*}
			& \left\| g_{\alpha\beta} - \scal{\dr_{\alpha}y^s}{ \dr_{\beta} y^s}_{\geu} \right\|_{\tilde{W}^{2,\left(\frac{n}{2},1\right)}(\B_s,\geu)} \\[2mm]
			&\leq \sum_{i=1}^n \|\omega^{s,i}(\dr_{\alpha}) - \dr_{\alpha} y^{s,i}\|_{\tilde{W}^{2,\left(\frac{n}{2},1\right)}(\B_s,\geu)} \left( \|\omega^{s,i}(\dr_{\beta})\|_{\tilde{W}^{2,\left(\frac{n}{2},1\right)}(\B_s,\geu)} + \|\dr_{\beta} y^{s,i}\|_{\tilde{W}^{2,\left(\frac{n}{2},1\right)}(\B_s,\geu)} \right) \\[2mm]
			& \qquad + \sum_{i=1}^n \|\omega^{s,i}(\dr_{\beta}) - \dr_{\beta} y^{s,i}\|_{\tilde{W}^{2,\left(\frac{n}{2},1\right)}(\B_s,\geu)} \left( \|\omega^{s,i}(\dr_{\alpha})\|_{\tilde{W}^{2,\left(\frac{n}{2},1\right)}(\B_s,\geu)} + \|\dr_{\alpha} y^{s,i}\|_{\tilde{W}^{2,\left(\frac{n}{2},1\right)}(\B_s,\geu)} \right) \\[2mm]
			& \leq C(n)\, \|\Riem^g\|_{L^{\left(\frac{n}{2},1\right)}(\B_s,g)}.
		\end{align*}
		Moreover, \eqref{eq:e_constant} and \eqref{eq:Unif_basis} imply that 
		\begin{align*}
			\|d(y^s)^{-1}\|_{L^{\infty}(\B_s,\geu)} \leq C(n).
		\end{align*}
		Thus, we obtain for any $1\leq \alpha,\beta \leq n$,
		\begin{align*}
			\left\| \left( ((y^s)^{-1})^*g \right)_{\alpha\beta} - \delta_{\alpha\beta} \right\|_{\tilde{W}^{2,\left(\frac{n}{2},1\right)}(y^s(\B_s),\geu)}  \leq C\, \|\Riem^g\|_{L^{\left(\frac{n}{2},1\right)}(\B_s,g)}.
		\end{align*}
	\end{proof} 
	
	We now pass to the limit $s\to +\infty$ in Proposition \ref{pr:apriori_local} to obtain Proposition \ref{pr:apriori_global}.
	
	\begin{proof}[Proof of Proposition \ref{pr:apriori_global}]
		We use the notations of the proof of Proposition \ref{pr:apriori_local}.\\
		The pointwise estimate \eqref{eq:Pointwise_g} implies that each $\omega^{s,i}$ is uniformly bounded in $L^{\infty}(\B_s)$. Combining \eqref{eq:Der_frame} and \eqref{eq:Omegaij}, we can pass to the limit $s\to +\infty$ and, up to a subsequence, we obtain a coframe $(\omega^1,\ldots,\omega^n)$ such that $d\omega^i = \omega^j \wedge \omega^i_j$ on $\R^n$ with the properties
		\begin{align}
			& d^{*_{\geu}}\omega^{i}_{j} = 0 \qquad \text{in }\R^n, \nonumber \\[2mm]
			& \sum_{1\leq i,j\leq n} \left( \|\omega^i_j\|_{L^{(n,1)}(\R^n,\geu)} + \|\g^{\geu} \omega^i_j\|_{L^{\left(\frac{n}{2},1\right)}(\R^n,\geu)} \right) \leq C\, \|\Riem^g\|_{L^{\left(\frac{n}{2},1\right)}(\R^n,g)}. \label{eq:intermediate_frame}
		\end{align}
		Concerning the maps $y^s\colon \B_s\to y^s(\B_s)$, they are $C^1$-diffeomorphism thanks to \eqref{eq:Unif_dy} with distortion bounded from above by $C(n)>1$ by \eqref{eq:y_biLip}. Letting $s\to +\infty$, we obtain a $C^1$ diffeomorphism $y\colon \R^n\to \R^n$ such that 
		\begin{align*}
			& \left\| \left( (y^{-1})^*g \right)_{\alpha\beta} - \delta_{\alpha\beta} \right\|_{L^{\infty}(\R^n,\geu)} +  \left\|\g^{\geu} \left( (y^{-1})^*g \right)_{\alpha\beta} \right\|_{L^{(n,1)}(\R^n,\geu)} + \left\|(\g^{\geu})^2 \left( (y^{-1})^*g \right)_{\alpha\beta} \right\|_{L^{\left(\frac{n}{2},1\right)}(\R^n,\geu)}  \\[2mm]
			& \leq C\, \|\Riem^g\|_{L^{\left(\frac{n}{2},1\right)}(\R^n,g)}.
		\end{align*}
		Moreover, the estimate \eqref{eq:Initial_g1} implies that
		\begin{align*}
			& \sum_{1\leq i,j\leq n} \left( \|\omega^i_j\|_{L^{(n,1)}(\R^n,g)} + \|\g^g \omega^i_j\|_{L^{\left(\frac{n}{2},1\right)}(\R^n,g)} \right) \\[2mm]
			& \leq C\, \eps\, \sum_{1\leq i,j\leq n} \|\omega^i_j\|_{L^{(n,1)}(\R^n,\geu)} + \sum_{1\leq i,j\leq n} \left( \|\omega^i_j\|_{L^{(n,1)}(\R^n,\geu)} + \|\g^{\geu} \omega^i_j\|_{L^{\left(\frac{n}{2},1\right)}(\R^n,\geu)} \right).
		\end{align*}
		The estimate \eqref{eq:intermediate_frame} implies that for $\eps_0$ small enough (depending only on $n$), it holds
		\begin{align*}
			\sum_{1\leq i,j\leq n} \left( \|\omega^i_j\|_{L^{(n,1)}(\R^n,g)} + \|\g^g \omega^i_j\|_{L^{\left(\frac{n}{2},1\right)}(\R^n,g)} \right) \leq C(n)\, \|\Riem^g\|_{L^{\left(\frac{n}{2},1\right)}(\R^n,g)}.
		\end{align*}
	\end{proof}

	\subsection{The continuity argument}\label{sec:Continuity}
	
	In this section, we prove a slightly stronger version of \Cref{th:diff_Rn}. 
	\begin{theorem}\label{th:diff_Rn_v2}
		There exist $\eps_0>0$ and $C_0>0$ depending only on $n$ satisfying the following property. Let $g$ be a metric on $\R^n$ of class $\bar{W}^{2,\left(\frac{n}{2},1\right)}(\R^n)$ such that  $g_{ij}(x_0)=\delta_{ij}$ at some point $x_0\in \R^n$ and 
		\begin{align*}
			\|\Riem^g\|_{L^{\left(\frac{n}{2},1\right)}(\R^n,g)} \leq \eps_0.
		\end{align*}
		Then there exist a $C^1$-diffeomorphism $y\colon \R^n\to \R^n$ and a coframe $(\omega^1,\ldots,\omega^n)$ for $g$ in $\R^n$ with connection forms $\omega^i_j$ such that the following estimates hold for all $1\leq i,j\leq n$
		\begin{align*}
			& 1) \qquad \left\|(y^*g)_{ij} - \delta_{ij} \right\|_{\bar{W}^{2,\left(\frac{n}{2},1\right)}(\R^n,\geu)} \leq C_0\, \|\Riem^g\|_{L^{\left(\frac{n}{2},1\right)}(\R^n,g)},\\[3mm]
			& 2) \qquad \left\| d(y^{-1})^i - \omega^i \right\|_{\bar{W}^{2,\left(\frac{n}{2},1\right)}(\R^n,g)} \leq C_0\,\|\Riem^g\|_{L^{\left(\frac{n}{2},1\right)}(\R^n,g)},\\[3mm]
			& 3) \qquad \| \omega^i_j\|_{L^{(n,1)}(\R^n,g)} + \|\g^g \omega^i_j \|_{L^{\left(\frac{n}{2},1\right)}(\R^n,g)} \leq C\, \|\Riem^g\|_{L^{\left(\frac{n}{2},1\right)}(\R^n,g)}.
		\end{align*}
	\end{theorem} 
	
	Given $\eps>0$ and $C>0$, we consider the sets $\Ur^{\eps}$ and $\Vr^{\eps}_C$ defined as follows
	\begin{align}\label{eq:def_Ueps}
		\Ur^{\eps} \coloneq \left\{ g\ \text{ weak } W^{2,\left( \frac{n}{2},1 \right)} \text{ metric on }\R^n : \begin{array}{l}
			\bullet \quad \exists x\in\R^n,\quad g(x)=\geu,\\[2mm]
			\bullet \quad g_{ij} \in \bar{W}^{2,\left(\frac{n}{2},1\right)}(\R^n), \\[2mm]
			\bullet \quad  \|\Riem^g\|_{L^{\left(\frac{n}{2},1\right)}(\R^n,g)} \leq \eps .
		\end{array} \right\},
	\end{align}
    
	\begin{align}\label{eq:def_Veps}
		\Vr^{\eps}_{C} \coloneq \left\{ g\in \Ur^{\eps} : \begin{array}{l}
			\exists y\in \Diff(\R^n)\text{ s.t.}\\[3mm]
			\qquad \bullet \qquad dy\in\bar{W}^{2,\left(\frac{n}{2},1\right)}(\R^n), \\[3mm]
			\qquad \bullet \qquad \left\| \left( (y^{-1})^*g \right)_{\alpha\beta} - \delta_{\alpha\beta} \right\|_{\bar{W}^{2,\left(\frac{n}{2},1\right)}(\R^n,\geu)}  \leq C\, \|\Riem^g\|_{L^{\left(\frac{n}{2},1\right)}(\R^n,g)}, \\[5mm]
			\exists (\omega^1,\ldots,\omega^n)\ \text{ coframe for $g$ on $\R^n$ s.t.}\\[3mm]
			\qquad \bullet \qquad \| \omega^i_j\|_{L^{(n,1)}(\R^n,g)} + \|\g^g \omega^i_j \|_{L^{\left(\frac{n}{2},1\right)}(\R^n,g)} \leq C\, \|\Riem^g\|_{L^{\left(\frac{n}{2},1\right)}(\R^n,g)}, \\[4mm]
			\qquad \bullet \qquad \|dy^i - \omega^i\|_{\bar{W}^{2,\left(\frac{n}{2},1\right)}(\R^n,g)} \leq C\, \|\Riem^g\|_{L^{\left(\frac{n}{2},1\right)}(\R^n,g)}.
		\end{array} \right\}
	\end{align}
    
    We now show that $\Vr^{\eps}_C$ is a non-empty closed and open set in $\Ur^{\eps}$ for well-chosen constants $C$ and $\eps$, in the $\bar{W}^{2,\left(\frac{n}{2},1\right)}_{\loc}(\R^n)$-topology. 
    As a first step, we prove that $\Ur^{\eps}$ is path-connected.
	
	\begin{claim}\label{cl:U_connected}
		$\Ur^{\eps}$ is path-connected in the $W^{2,\left(\frac{n}{2},1\right)}_{\loc}(\R^n)$-topology for every $\eps$.
	\end{claim}
	\begin{proof}
		Let $g\in \Ur^{\eps}$. Let $x\in\R^n$ be such that $g_{ij}(x)=\delta_{ij}$. For simplicity, we assume that $x=0$.	For $t\in[0,1]$, we define $g_t \coloneqq g(t\, \cdot)$. We have $g_0=\geu$ and  $\Riem^{g_t} = t^2\, \Riem^g(t\cdot)$. Thus, we have
		\begin{align*}
			\|\Riem^{g_t}\|_{L^{\left(\frac{n}{2},1\right)}(\R^n,g_t)} & = \frac{n}{2}\int_0^{+\infty} \vol_{g_t}\left( \left\{ x\in\R^n : |\Riem^{g_t}(x)|_{g_t(x)} \geq s \right\}\right)^{\frac{2}{n}}\, ds \\[2mm]
			& = \frac{n}{2}\int_0^{+\infty} \vol_{g_t}\left( \left\{ x\in\R^n : t^2|\Riem^g(tx)|_{g(tx)} \geq s \right\}\right)^{\frac{2}{n}}\, ds \\[2mm]
			& = \frac{n}{2}\int_0^{+\infty} \vol_g\left( \left\{ y\in\R^n : |\Riem^g(y)|_{g(y)} \geq \frac{s}{t^2} \right\}\right)^{\frac{2}{n}}\, \frac{ds}{t^2} \\[2mm]
			& = \frac{n}{2}\int_0^{+\infty} \vol_g\left( \left\{ y\in\R^n : |\Riem^g(y)|_{g(y)} \geq \tau \right\}\right)^{\frac{2}{n}}\, d\tau \\[2mm]
			& = \|\Riem^g\|_{L^{\left(\frac{n}{2},1\right)}(\R^n,g)} \leq \eps.
		\end{align*}
		Moreover since the coefficients $g_{\alpha\beta}$ are continuous and equal to $\delta_{\alpha\beta}$ at $0$, we have for any radius $R>0$ and $1\leq \alpha,\beta\leq n$,
        \begin{align*}
         \|(g_t)_{\alpha\beta} - \delta_{\alpha\beta} \|_{\bar{W}^{2,\left(\frac{n}{2},1\right)}(\B^n_R,\geu)} = \|g_{\alpha\beta} - \delta_{\alpha\beta} \|_{\bar{W}^{2,\left(\frac{n}{2},1\right)}(\B^n_{tR},\geu)} \xrightarrow[t\to 0]{} 0.
        \end{align*}
        Hence, $(t\mapsto g_t)$ is a continuous path in $\Ur^{\eps}$ for the $\bar{W}^{2,\left(\frac{n}{2},1\right)}_{\loc}(\R^n)$-topology connecting $g$ to $\geu$.
	\end{proof}

	Since $\geu\in \Vr^{\eps}_C$ for any $\eps$ and $C$, it is not empty. We now show that $\Vr^{\eps}_C$ is closed for $\eps$ small enough.
	
	\begin{claim}
		For any $C>0$, there exists $\eps_c>0$ depending only on $n$ and $C$ such that for every $\eps\in(0,\eps_c)$, the set $\Vr^{\eps}_C$ is closed in $\Ur^{\eps}$ for the strong $\bar{W}^{2,\left(\frac{n}{2},1\right)}_{\loc}(\R^n)$-topology.
	\end{claim}
	\begin{proof}
		Let $(g^k)_{k\in\N}\subset \Vr^{\eps}_{C}$ be a sequence converging to some $g^{\infty}\in \Ur^{\eps}$ in the $\bar{W}^{2,\left(\frac{n}{2},1\right)}_{\loc}(\R^n)$-topology. 
        Up to a subsequence, the associated coframes $(\omega^{k,1},\ldots,\omega^{k,n})$ are bounded in  in the $L^{\infty}\cap \dot{W}^{1,(n,1)}(\R^n)$ topology, and thus converge locally uniformly and weakly in $\dot{W}^{1,(n,1)}(\R^n)$ to some coframe $(\omega^{\infty,1},\ldots,\omega^{\infty,n})$ for $g^{\infty}$ such that 
		\begin{align}\label{eq:bound_gomegai}
			\| \g^{\geu} \omega^{\infty,i}\|_{L^{(n,1)}(\R^n,\geu)} \leq \sup_{R>0} \ \liminf_{k\to +\infty}\ \| \g^{\geu} \omega^{k,i}\|_{L^{(n,1)}(\B_R,\geu)}  \leq C(g).
		\end{align}
		Indeed, we have for any $R>0$ by \eqref{eq:covariant_form} and \eqref{eq:gomegai}
		\begin{align*}
			\| \g^{\geu} \omega^{k,i}\|_{L^{(n,1)}(\B_R,\geu)} & \leq C(n)\, \|g^k\|_{\bar{W}^{2,\left(
					\frac{n}{2},1\right)} (\B_R,\geu)}^2 + C(n)\|\g^{\geu}\omega^{k,i}_{\ j}\|_{L^{(n,1)}(\B_R,\geu)}\, \|\omega^{k,j}\|_{L^{\infty}(\B_R,\geu)} .
		\end{align*}
        Since $g^k|_{\B_R}$ converges strongly to $g|_{\B_R}$ in the $W^{2,\left(\frac{n}{2},1\right)}(\B_r)$-topology, we obtain 
        \begin{align*}
            \forall R>0,\qquad \liminf_{k\to +\infty}\ \| \g^{\geu} \omega^{k,i}\|_{L^{(n,1)}(\B_R,\geu)}  \leq C(n,g).
        \end{align*}
        This implies \eqref{eq:bound_gomegai}.\\
        
		By the same argument, we have (up to a subsequence) that the maps $dy_k$ converge to some 1-form $\theta\in \bar{W}^{2,\left(\frac{n}{2},1\right)}(\R^n;\R^n\otimes\Lambda^1\R^n)$ in the following sense:
        \begin{align*}
            dy_k \xrightarrow[k\to +\infty]{} \theta\qquad \text{ in }C^0_{\loc}(\R^n) \text{ and weakly in }\bar{W}^{2,\left(\frac{n}{2},1\right)}_{\loc}(\R^n).
        \end{align*}
        Hence, we can pass to the limit in the equation $d(dy_k)=0$ to obtain $d\theta=0$. Since $\R^n$ is simply connected, there exists a map $y_{\infty}\in C^1(\R^n;\R^n)$ such that $dy_{\infty}=\theta$.\\
        
		On the other hand, we have the following estimates
		\begin{align*}
			\left\| \left( (y^{-1}_k)^*g^k \right)_{\alpha\beta} - \delta_{\alpha\beta} \right\|_{\bar{W}^{2,\left(\frac{n}{2},1\right)}(\R^n,\geu)}  \leq C\, \|\Riem^{g^k}\|_{L^{\left(\frac{n}{2},1\right)}(\R^n,g^k)}.
		\end{align*}
		Therefore, it holds for any $R>0$,
		\begin{align*}
			\limsup_{k\to +\infty} \|d(y^{-1}_k) \|_{\bar{W}^{2,\left(\frac{n}{2},1\right)}(\B_R)} & \leq C\, \eps + c(n) + \limsup_{k\to +\infty} \|g^k\|_{\bar{W}^{2,\left(\frac{n}{2},1\right)}(\B_R)} \\[2mm]
             & \leq C(n,g)<+\infty.
		\end{align*}
		Thus, up to a subsequence, $d(y^{-1}_k)$ converges to some 1-form $\xi\in \bar{W}^{2,\left(\frac{n}{2},1\right)}(\R^n;\R^n\otimes \Lambda^1\R^n)$ in the following sense:
        \begin{align*}
            d(y_k^{-1}) \xrightarrow[k\to +\infty]{} \xi \qquad \text{ in }C^0_{\loc}(\R^n) \text{ and weakly in }\bar{W}^{2,\left(\frac{n}{2},1\right)}_{\loc}(\R^n).
        \end{align*}
        Thus, we can pass to the limit in the equation $d\left( d(y_k^{-1})\right)=0$ to obtain $d\xi=0$. Since $\R^n$ is simply connected, we obtain a map $z\in C^1(\R^n)$ such that $dz=\xi$.\\
        
        We can also pass to the limit in the relation $\left( \g^{\geu} y_k\circ y_k^{-1}\right) \g^{\geu}(y_k^{-1})=I_n$ to obtain $I_n = (\g^{\geu} y_{\infty}\circ z) \g^{\geu} z = \g^{\geu} (y_{\infty}\circ z)$. Thus, up to a translation,  we have that $y_{\infty}$ is invertible and $z = y_{\infty}^{-1}$. Thus, we can pass to the limit in the estimate of $(y_k^{-1})^*g^k$ and obtain for any $R>0$, 
		\begin{align*}
			\left\| \left( (y_{\infty}^{-1})^*g^{\infty} \right)_{\alpha\beta} - \delta_{\alpha\beta} \right\|_{\bar{W}^{2,\left(\frac{n}{2},1\right)}(\B_R,\geu)} & \leq \liminf_{k\to +\infty} \left\| \left( (y_{k}^{-1})^*g^{k} \right)_{\alpha\beta} - \delta_{\alpha\beta} \right\|_{\bar{W}^{2,\left(\frac{n}{2},1\right)}(\B_R,\geu)} \\[3mm]
			&  \leq C\, \liminf_{k\to +\infty} \|\Riem^{g^k}\|_{L^{\left(\frac{n}{2},1\right)}(\R^n,g^k)} \\[3mm]
			&  \leq C\, \eps.
		\end{align*}
		Thus, for $\eps<\eps_c$ with $\eps_c=\eps(n,C)$, we can apply Proposition \ref{pr:apriori_global} and we obtain $g^{\infty}\in \Vr^{\eps}_C$.
	\end{proof}
	
	It remains to show that $\Vr^{\eps}_C$ is open. This is essentially a direct application of Proposition \ref{pr:apriori_global}.
	
	\begin{claim}
		There exists $\eps_o>0$ and $C_o>0$ depending only on $n$ such that for any $\eps\in(0,\eps_o)$ and $C>C_o$, the set $\Vr^{\eps}_C$ is open in $\Ur^{\eps}$ for the $W^{2,\left(\frac{n}{2},1\right)}_{\loc}$-topology.
	\end{claim}
	
	\begin{proof}
		Let $g\in \Vr^{\bar{\eps}}_{\bar{C}}$, $h\in \Ur^{\bar{\eps}}$ and $\delta>0$ be such that 
		\begin{align}\label{eq:Initial_h}
			\forall i,j\in\{1,\ldots,n\},\qquad \|g_{ij} - h_{ij}\|_{W^{2,\left(\frac{n}{2},1\right)}(\R^n,\geu)} < \delta.
		\end{align}
		Since $g\in \Vr^{\bar{\eps}}_{\bar{C}}$, there exists a diffeomorphism $y\in \Diff(\R^n)$ such that 
		\begin{align*}
			\begin{aligned}
				\left\| \left( (y^{-1})^*g \right)_{\alpha\beta} - \delta_{\alpha\beta} \right\|_{\bar{W}^{2,\left(\frac{n}{2},1\right)}(\R^n)}  \leq \bar{C}\, \|\Riem^g\|_{L^{\left(\frac{n}{2},1\right)}(\R^n,g)}.
			\end{aligned}
		\end{align*}
		Up to reducing $\delta$, we obtain 
		\begin{align}\label{eq:Init_h}
			\forall \alpha,\beta\in\{1,\ldots,n\},\qquad \left\| \left( (y^{-1})^*h\right)_{\alpha\beta} - \delta_{\alpha\beta} \right\|_{\bar{W}^{2,\left(\frac{n}{2},1\right)}(\R^n)} \leq 2\, \bar{C}\, \bar{\eps}.
		\end{align}
		In particular, for $\bar{\eps}\leq \eps_1(n,\bar{C})$, we have the following inequality of matrices
		\begin{align*}
			\forall \alpha,\beta\in\{1,\ldots,n\},\qquad \left\| \left( (y^{-1})^*h\right)_{\alpha\beta} - \delta_{\alpha\beta} \right\|_{\bar{W}^{2,\left(\frac{n}{2},1\right)}(\R^n)} \leq \eps_0.
		\end{align*}
		By choosing $\bar{C}>C_0$ and $\bar{\eps}< (1+ 2\, \bar{C})^{-1}\, \eps_0$, where $\eps_0$ and $C_0$ are defined in Proposition \ref{pr:apriori_global}, we obtain $(y^{-1})^*h\in \Ur^{\eps_0}$ since
        \begin{align*}
           \left\| \Riem^{(y^{-1})^\ast h} \right\|_{L^{\left(\frac{n}{2},1\right)}\left(\R^n,(y^{-1})^\ast h \right)} = \left\| \Riem^{ h} \right\|_{L^{\left(\frac{n}{2},1\right)}\left(\R^n,h \right)} \leq \bar{\eps} < \eps_0. 
        \end{align*}
        Thus, we obtain a diffeomorphism $z\in \Diff(\R^n)$ and a coframe $(\omega^1,\ldots,\omega^n)$ for $(y^{-1})^*h$ on $\R^n$ such that 
		\begin{align}
			& \qquad \bullet \qquad dz\in\bar{W}^{2,\left(\frac{n}{2},1\right)}(\R^n),  \nonumber \\[3mm]
			& \qquad \bullet \qquad \left\| \left( ((z\circ y)^{-1})^*h \right)_{\alpha\beta} - \delta_{\alpha\beta} \right\|_{\bar{W}^{2,\left(\frac{n}{2},1\right)}(\R^n,\geu)}  \leq C_0\, \|\Riem^h\|_{L^{\left(\frac{n}{2},1\right)}(\R^n,h)}, \nonumber \\[4mm]
			& \qquad \bullet \qquad \| \omega^i_j\|_{L^{(n,1)}(\R^n,(y^{-1})^*h)} + \|\g^{(y^{-1})^*h} \omega^i_j \|_{L^{\left(\frac{n}{2},1\right)}(\R^n,(y^{-1})^*h)} \leq C_0\, \|\Riem^h\|_{L^{\left(\frac{n}{2},1\right)}(\R^n,h)}, \label{eq:CF1} \\[4mm]
			& \qquad \bullet \qquad \|dz^i - \omega^i\|_{\bar{W}^{2,\left(\frac{n}{2},1\right)}(\R^n,(y^{-1})^*h)} \leq C_0 \, \|\Riem^h\|_{L^{\left(\frac{n}{2},1\right)}(\R^n,h)}. \label{eq:CF2}
		\end{align}
		Since the pullback commutes with the exterior derivative, we obtain that the coframe $(y^*\omega^1,\ldots,y^*\omega^n)$ for $h$ satisfies the following structure equations:
		\begin{align}
			d\left( y^*\omega^i \right) = \left(y^*\omega^i_j\right)\wedge \left(y^*\omega^j\right).
		\end{align}
		By change of variable the estimates \eqref{eq:CF1} and \eqref{eq:CF2} imply that 
		\begin{align*}
			& \qquad \bullet \qquad \| y^*\omega^i_j\|_{L^{(n,1)}(\R^n,(y^{-1})^*h)} + \|\g^{h} \left(y^*\omega^i_j\right) \|_{L^{\left(\frac{n}{2},1\right)}(\R^n,h)} \leq C_0\, \|\Riem^h\|_{L^{\left(\frac{n}{2},1\right)}(\R^n,h)}, \\[4mm]
			& \qquad \bullet \qquad \left\| y^*\left(dz^i - \omega^i \right) \right\|_{\bar{W}^{2,\left(\frac{n}{2},1\right)}(\R^n,h)} \leq C_0 \, \|\Riem^h\|_{L^{\left(\frac{n}{2},1\right)}(\R^n,h)}.
		\end{align*}
		We conclude that $h\in \Vr^{\bar{\eps}}_{C_0}$ thanks to the equality $y^*(dz^i) = d(z^i\circ y)$.
	\end{proof}

	\section{Metrics on $\R^n$ of class $W^{2,\left(\frac{n}{2},2\right)}$}\label{sec:ProofSecond}
	
	In this section, we prove a slightly strong version of \Cref{th:diff_Rn}, namely we reduce the Lorentz exponent of the metric. 
		\begin{theorem}\label{th:Weak_metrics_v2}
		There exist $\eps_1>0$ and $C_1>0$ depending only on $n$ satisfying the following property. Let $g$ be a weak metric on $\R^n$ such that $\dr_k g_{ij}\in L^{(n,2)}(\R^n)$, $\dr^2_{kl} g_{ij} \in L^{\left(\frac{n}{2},2\right)}(\R^n)$ and 
		\begin{align*}
			\|\Riem^g\|_{L^{\left(\frac{n}{2},1\right)}(\R^n,g)} \leq \eps_1.
		\end{align*}
		Then there exists $y\in \BiLip(\R^n)$ and a coframe $(\omega^1,\ldots,\omega^n)$ for $g$ in $\R^n$ with connection forms $\omega^i_j$ such that the following estimates hold for all $1\leq i,j\leq n$
		\begin{align*}
			& 1) \qquad \left\|(y^*g)_{ij} - \delta_{ij} \right\|_{\bar{W}^{2,\left(\frac{n}{2},1\right)}(\R^n,\geu)} \leq C_1\, \|\Riem^g\|_{L^{\left(\frac{n}{2},1\right)}(\R^n,g)},\\[3mm]
			& 2) \qquad \left\| d(y^{-1})^i - \omega^i \right\|_{\bar{W}^{2,\left(\frac{n}{2},1\right)}(\R^n,g)} \leq C_1\,\|\Riem^g\|_{L^{\left(\frac{n}{2},1\right)}(\R^n,g)},\\[3mm]
			& 3) \qquad \| \omega^i_j\|_{L^{(n,1)}(\R^n,g)} + \|\g^g \omega^i_j \|_{L^{\left(\frac{n}{2},1\right)}(\R^n,g)} \leq C_1\, \|\Riem^g\|_{L^{\left(\frac{n}{2},1\right)}(\R^n,g)}.
		\end{align*}
	\end{theorem} 
	
	To prove the above result, we approximate a given metric $g\in \bar{W}^{2,\left(\frac{n}{2},2\right)}(\R^n)$ by $C^{\infty}$ metrics $(g^{\delta})_{\delta>0}$ obtained by convolution with a mollifier. We prove that if the Riemann tensor of $g$ lies in $L^{\left(\frac{n}{2},1\right)}(\R^n)$, then the one of $g^{\delta}$ also, with estimates. Thus, we can apply \Cref{th:diff_Rn_v2} to get coordinates for $g^{\delta}$ with control in $\bar{W}^{2,\left(\frac{n}{2},1\right)}(\R^n)$. Since these coordinates come with an orthonormal coframe, we prove that these coordinates converge.
	
	\begin{proof}[Proof of \Cref{th:Weak_metrics_v2}.]
		Let $g$ be a weak metric satisfying the assumptions of \Cref{th:Weak_metrics_v2}. By Definition \ref{def:Weak_metric}, there exists $\lambda>1$ such that the following pointwise inequality for matrices hold almost everywhere
		\begin{align}\label{eq:Ineq_g}
			\lambda^{-1}\, \delta_{ij} \leq g_{ij} \leq \lambda\, \delta_{ij} \qquad \text{ a.e. in }\R^n.
		\end{align}
		Consider a family $(g^{\delta})_{\delta>0}$ of $C^{\infty}$-metrics that converges to $g$ as $\delta\to 0$ strongly in $W^{2,\left( \frac{n}{2},2\right)}_{\loc}(\R^n,\geu)$ and almost everywhere. We can choose each $g^{\delta}$ to be given by the convolution of $g$ with a mollifier, that is to say, if $\vp\in C^{\infty}_c(\B^n;\R^+)$ with $\int_{\B^n}\vp(x)\ dx=1$, then we consider
		\begin{align*}
			\forall x\in \R^n,\qquad g^{\delta}_{\alpha\beta}(x) \coloneqq \int_{\B^n(0,\delta)} \frac{1}{\delta^n} \vp\left( \frac{y}{\delta}\right)\, g_{\alpha\beta}(x-y)\, dy.
		\end{align*}	
		As a consequence of \eqref{eq:Ineq_g}, we have the following uniform ellipticity property for $g^{\delta}$
		\begin{align}\label{eq:Ineq_gd}
			\lambda^{-1}\, \delta_{\alpha\beta} \leq g^{\delta}_{\alpha\beta} \leq \lambda\, \delta_{\alpha\beta} \qquad \text{ in }\R^n.
		\end{align}
        Indeed, for any fixed vector $X\in \R^n$ and $p\in\R^n$, we have 
        \begin{align*}
            g^{\delta}_{\alpha\beta}(p)\, X^{\alpha}\, X^{\beta} = \int_{\B^n(0,\delta)} \frac{1}{\delta^n} \vp\left(\frac{y}{\delta}\right)\, g_{\alpha\beta}(p-y)\, X^{\alpha}\, X^{\beta}\, dy.
        \end{align*}
        Since $\vp\geq 0$, we obtain from \eqref{eq:Ineq_g} that 
        \begin{align*}
            g^{\delta}_{\alpha\beta}(p)\, X^{\alpha}\, X^{\beta} \leq \int_{\B^n(0,\delta)} \frac{1}{\delta^n} \vp\left(\frac{y}{\delta}\right)\, \lambda\, |X|^2_{\geu}\, dy = \lambda\, |X|^2_{\geu}.
        \end{align*}
        We also have the reverse inequality for the same reason
        \begin{align*}
            g^{\delta}_{\alpha\beta}(p)\, X^{\alpha}\, X^{\beta} \geq \lambda^{-1}\, |X|^2_{\geu}.
        \end{align*}
        
		As a consequence of \eqref{eq:Ineq_gd}, we obtain the convergence of the Christoffel symbols of $g^{\delta}$ by dominated convergence as $\delta\to 0$:
		\begin{align*}
			{^{g^{\delta}}\Gamma}^k_{ij} = \frac{1}{2}\, (g^{\delta})^{kl}\, \Big( \dr_i g^{\delta}_{jl} + \dr_j g^{\delta}_{il} - \dr_l g^{\delta}_{ij}\Big) \xrightarrow[\delta\to 0]{} {^g \Gamma}_{ij}^k \qquad \text{strongly in }W^{1,\left(\frac{n}{2},2\right)}(\R^n).
		\end{align*}
		The Riemann tensor is given by \eqref{eq:Riem_coordinates}
		\begin{align*}
			(\Riem^g)_{ijk}^{\ \ \ l} =  \dr_i ({^{g} \Gamma}^l_{jk}) - \dr_j({^{g}\Gamma}^l_{ik}) + ({^{g}\Gamma}^m_{jk})\, ({^{g}\Gamma}^l_{im})-({^{g}\Gamma}^m_{ik})\, ({^{g}\Gamma}^l_{jm}).
		\end{align*}
		Lowering the last index, we have 
		\begin{align*}
			\Riem^g_{ijkp} & = g_{pl} \left[ \dr_i ({^{g} \Gamma}^l_{jk}) - \dr_j({^{g}\Gamma}^l_{ik}) + ({^{g}\Gamma}^m_{jk})\, ({^{g}\Gamma}^l_{im})-({^{g}\Gamma}^m_{ik})\, ({^{g}\Gamma}^l_{jm}) \right] \\[3mm]
			& = \frac{1}{2}\, g_{pl} \left[ \dr_i\left( g^{ql}(\dr_j g_{kq} + \dr_k g_{jq} - \dr_q g_{jk} ) \right) - \dr_j \left( g^{lq}(\dr_i g_{kq} + \dr_k g_{iq} - \dr_q g_{ik}) \right) \right] \\[3mm]
			& \qquad + \frac{1}{4}\, (\dr_i g_{pm} + \dr_m g_{ip} - \dr_p g_{im} )\, g^{mq}\, (\dr_j g_{kq} + \dr_k g_{jq} - \dr_q g_{jk}) \\[3mm]
			& \qquad -\frac{1}{4}\, (\dr_j g_{pm} + \dr_m g_{jp} - \dr_p g_{jm} )\, g^{mq} (\dr_i g_{kq} + \dr_k g_{iq} - \dr_q g_{ik}).
		\end{align*}
		We expand the derivatives of the first line and obtain 
		\begin{align*}
			\Riem^g_{ijkp}& = \frac{1}{2}\, g_{pl} \Bigg[ -g^{q\alpha}\, (\dr_i g_{\alpha\beta})\, g^{\beta l}\, (\dr_j g_{kq} + \dr_k g_{jq} - \dr_q g_{jk} ) +   g^{ql}(\dr^2_{ij} g_{kq} + \dr^2_{ik} g_{jq} - \dr^2_{iq} g_{jk} ) \\[3mm]
			& \qquad + g^{l\alpha}\, (\dr_j g_{\alpha\beta})\, g^{\beta q}(\dr_i g_{kq} + \dr_k g_{iq} - \dr_q g_{ik})  - g^{lq}(\dr^2_{ij} g_{kq} + \dr^2_{jk} g_{iq} - \dr^2_{jq} g_{ik}) \Bigg] \\[3mm]
			& \qquad + \frac{1}{4}\, (\dr_i g_{pm} + \dr_m g_{ip} - \dr_p g_{im} )\, g^{mq}\, (\dr_j g_{kq} + \dr_k g_{jq} - \dr_q g_{jk}) \\[3mm]
			& \qquad -\frac{1}{4}\, (\dr_j g_{pm} + \dr_m g_{jp} - \dr_p g_{jm} )\, g^{mq} (\dr_i g_{kq} + \dr_k g_{iq} - \dr_q g_{ik}).
		\end{align*}
		We rearrange the terms and end up with the following expression
		\begin{align}
			\Riem^g_{ijkp}& = \frac{1}{2}\,  \Big[ \dr^2_{ij} g_{kp} + \dr^2_{ik} g_{jp} - \dr^2_{ip} g_{jk}  - \dr^2_{ij} g_{kp} + \dr^2_{jk} g_{ip} - \dr^2_{jp} g_{ik} \Big] \label{eq:second_der_g} \\[3mm]
			&\qquad + \frac{1}{2}\, \Big[ -g^{q\alpha}\, (\dr_i g_{\alpha p})\, (\dr_j g_{kq} + \dr_k g_{jq} - \dr_q g_{jk} )  +(\dr_j g_{p \beta})\, g^{\beta q}(\dr_i g_{kq} + \dr_k g_{iq} - \dr_q g_{ik})  \Big] \nonumber \\[3mm]
			& \qquad + \frac{1}{4}\, (\dr_i g_{pm} + \dr_m g_{ip} - \dr_p g_{im} )\, g^{mq}\, (\dr_j g_{kq} + \dr_k g_{jq} - \dr_q g_{jk}) \nonumber \\[3mm]
			& \qquad -\frac{1}{4}\, (\dr_j g_{pm} + \dr_m g_{jp} - \dr_p g_{jm} )\, g^{mq} (\dr_i g_{kq} + \dr_k g_{iq} - \dr_q g_{ik}).\nonumber
		\end{align}
        In order to shorten the notations, we introduce the decomposition $\Riem^g_{ijkl} = A_{ijkl}(g) + B_{ijkl}(g)$, where
        \begin{align*}
            & A_{ijkl}(g) \coloneq \frac{1}{2}\,  \Big[ \dr^2_{ij} g_{kp} + \dr^2_{ik} g_{jp} - \dr^2_{ip} g_{jk}  - \dr^2_{ij} g_{kp} + \dr^2_{jk} g_{ip} - \dr^2_{jp} g_{ik} \Big], \\[2mm]
            & B_{ijkl}(g) \coloneq \frac{1}{2}\, \Big[ -g^{q\alpha}\, (\dr_i g_{\alpha p})\, (\dr_j g_{kq} + \dr_k g_{jq} - \dr_q g_{jk} )  +(\dr_j g_{p \beta})\, g^{\beta q}(\dr_i g_{kq} + \dr_k g_{iq} - \dr_q g_{ik})  \Big]  \\[3mm]
			& \qquad + \frac{1}{4}\, (\dr_i g_{pm} + \dr_m g_{ip} - \dr_p g_{im} )\, g^{mq}\, (\dr_j g_{kq} + \dr_k g_{jq} - \dr_q g_{jk})  \\[3mm]
			& \qquad -\frac{1}{4}\, (\dr_j g_{pm} + \dr_m g_{jp} - \dr_p g_{jm} )\, g^{mq} (\dr_i g_{kq} + \dr_k g_{iq} - \dr_q g_{ik}).
        \end{align*}
		Since $\Riem^g\in L^{\left(\frac{n}{2},1\right)}(\R^n,g) = L^{\left(\frac{n}{2},1\right)}(\R^n,\geu)$ (the equality of spaces follows from \eqref{eq:Ineq_g}) and the metrics $g_{\delta}$ are given by convolution of $g$ with a mollifier, we obtain that the line \eqref{eq:second_der_g} also lies in $L^{\left(\frac{n}{2},1\right)}(\R^n,\geu)$.
        Indeed, we estimate the Riemann tensor of $g^{\delta}$ using the decomposition obtained above. On one hand, we have 
        \begin{align*}
            A_{ijkl}(g^{\delta})  = \frac{1}{\delta^n} \vp\left(\frac{\cdot}{\delta}\right) * A_{ijkl}(g) = \frac{1}{\delta^n} \vp\left(\frac{\cdot}{\delta}\right) *\left[ \Riem^g_{ijkl} - B_{ijkl}(g) \right].
        \end{align*}
        Since $B_{ijkl}(g)$ is quadratic in the derivatives of $g$, we obtain 
        \begin{align*}
            \left\| A_{ijkl}(g^{\delta}) \right\|_{L^{\left(\frac{n}{2},1\right)}(\R^n,\geu)} & \leq C(n,\lambda)\left( \|\Riem^g\|_{L^{\left(\frac{n}{2},1\right)}(\R^n,g)} +  \left\| B_{ijkl}(g) \right\|_{L^{\left(\frac{n}{2},1\right)}(\R^n,\geu)} \right) \\[2mm]
            & \leq C(n,\lambda)\left( \|\Riem^g\|_{L^{\left(\frac{n}{2},1\right)}(\R^n,g)} +  \sum_{\alpha,\beta,\gamma} \left\| \dr_{\alpha} g_{\beta\gamma} \right\|_{L^{\left(n,2\right)}(\R^n,\geu)}^2 \right).
        \end{align*}
        As well for $B_{ijkl}(g^{\delta})$, using \eqref{eq:Ineq_gd}, we have 
        \begin{align*}
            \left\| B_{ijkl}(g^{\delta}) \right\|_{L^{\left(\frac{n}{2},1\right)}(\R^n,\geu)} & \leq C(n,\lambda)\, \sum_{\alpha,\beta,\gamma} \left\| \dr_{\alpha} g_{\beta\gamma}^{\delta} \right\|_{L^{\left(n,2\right)}(\R^n,\geu)}^2 \\[2mm]
            & \leq C(n,\lambda) \sum_{\alpha,\beta,\gamma} \left\| \dr_{\alpha} g_{\beta\gamma} \right\|_{L^{\left(n,2\right)}(\R^n,\geu)}^2 .
        \end{align*}
        Hence, we can pass to the limit $\delta\to0$ both in the linear and the quadratic terms. We obtain the strong convergence of the Riemann tensor as $\delta\to 0$:
		\begin{align*}
			\Riem^{g^{\delta}}_{ijkl} \xrightarrow[\delta\to 0]{} \Riem^g_{ijkl} \qquad \text{strongly in }L^{\left( \frac{n}{2},1\right)}(\R^n).
		\end{align*}
		Hence, we deduce that for $0<\delta<\delta_0$ large enough (by dominated convergence and using that $g^{\delta}$ converges a.e. to $g$), it holds
		\begin{align*}
			\|\Riem^{g^{\delta}}\|_{L^{\left(\frac{n}{2},1\right)}(\R^n,g_{\delta})} \leq  \|\Riem^g\|_{L^{\left( \frac{n}{2},1\right)}(\R^n,g)} +\eps_1 \leq 2\, \eps_1.
		\end{align*}
		Since $g^{\delta}$ converges a.e. to $g$, there exists a Lebesgue point $x_0\in \B$ such that $g(x_0)$ is well-defined and $g^{\delta}(x_0)\to g(x_0)$ as $\delta\to 0$. Thus, up to a linear change of coordinates, we can assume that $g^{\delta}_{ij}(x_0)=\delta_{ij} = g_{ij}(x_0)$. Letting $\eps_1\coloneq \eps_0/2$, we can now apply \Cref{th:diff_Rn_v2} to the metrics $g^{\delta}$.\\
		
		For $\delta\in(0,\delta_0)$, there exist $y_{\delta}\in \Diff(\R^n)$ and a coframe $(\omega^{\delta,1},\ldots,\omega^{\delta,n})$ for $g^{\delta}$ in $\R^n$ with connection forms $\omega^{\delta,i}_{\ \ j}$ such that the following estimates hold for all $1\leq i,j\leq n$
		\begin{align}
			& \bullet \qquad \left\|(y_{\delta}^*g^{\delta})_{ij} - \delta_{ij} \right\|_{\bar{W}^{2,\left(\frac{n}{2},1\right)}(\R^n,\geu)} \leq C_0\, \left\| \Riem^{g^{\delta}} \right\|_{L^{\left(\frac{n}{2},1\right)}\left( \R^n,g^{\delta} \right)}, \label{eq:gd_good} \\[3mm]
			& \bullet \qquad \left\| d(y^{-1}_{\delta})^i - \omega^{\delta,i} \right\|_{\bar{W}^{2,\left(\frac{n}{2},1\right)}\left( \R^n,g^{\delta} \right)} \leq C_0\, \left\|\Riem^{g^{\delta}} \right\|_{L^{\left(\frac{n}{2},1\right)} \left(\R^n,g^{\delta} \right)}, \label{eq:Linfty_yinv_bad}\\[3mm]
			& \bullet \qquad \left\| \omega^{\delta,i}_{\ \ j} \right\|_{L^{(n,1)} \left( \R^n,g^{\delta} \right)} + \left\|\g^{g^{\delta}} \omega^{\delta,i}_{\ \ j} \right\|_{L^{\left(\frac{n}{2},1\right)}\left(\R^n,g^{\delta} \right)} \leq C\, \left\|\Riem^{g^{\delta}} \right\|_{L^{\left(\frac{n}{2},1\right)} \left(\R^n,g^{\delta} \right)}.\label{eq:Conn_bad1}
		\end{align} 
		Since the Christoffel symbols of $g$ are merely in $L^{(n,2)}(\R^n)$ and their first derivates in $L^{\left(\frac{n}{2},2\right)}(\R^n)$, we obtain that the space $\bar{W}^{2,\left(\frac{n}{2},1\right)}\left( \R^n,g \right)$ is a priori not well-defined. However, we have the following uniform versions of the last two estimates. Concerning \eqref{eq:Linfty_yinv_bad}, we have the uniform estimate in $\delta$
        \begin{align}\label{eq:Linfty_yinv}
        \begin{aligned} 
            & \left\| d(y^{-1}_{\delta})^i - \omega^{\delta,i} \right\|_{L^{\infty}\left( \R^n,g^{\delta} \right)} + \left\|\g^{g^{\delta}}\left( d(y^{-1}_{\delta})^i - \omega^{\delta,i} \right)\right\|_{L^{(n,2)}\left( \R^n,g^{\delta} \right)} + \left\| (\g^{g^{\delta}})^2\left( d(y^{-1}_{\delta})^i - \omega^{\delta,i}\right) \right\|_{L^{\left(\frac{n}{2},2\right)}\left( \R^n,g^{\delta} \right)} \\[2mm]
            & \leq C_0\, \left\|\Riem^{g^{\delta}} \right\|_{L^{\left(\frac{n}{2},1\right)} \left(\R^n,g^{\delta} \right)}.
            \end{aligned}
        \end{align}
        Whereas for \eqref{eq:Conn_bad1}, we have
		\begin{align}
			\left\| \omega^{\delta,i}_{\ \ j} \right\|_{L^{(n,1)} \left( \R^n,g^{\delta} \right)} + \left\|\g^{g^{\delta}} \omega^{\delta,i}_{\ \ j} \right\|_{L^{\left(\frac{n}{2},2\right)}\left(\R^n,g^{\delta} \right)} \leq C\, \left\|\Riem^{g^{\delta}} \right\|_{L^{\left(\frac{n}{2},1\right)} \left(\R^n,g^{\delta} \right)}. \label{eq:Conn_bad}
		\end{align} 
		Thanks to \eqref{eq:Ineq_gd}, we obtain that both $dy_{\delta}$ and $d(y_{\delta}^{-1})$ are bounded in $L^{\infty}(\R^d)$. Indeed, combining \eqref{eq:Ineq_gd} and \eqref{eq:Linfty_yinv}, we obtain 
        \begin{align}\label{eq:Linfty_yinv2}
            \left\|d(y^{-1}_{\delta})^i \right\|_{L^{\infty}(\R^n)} \leq C_0\, \eps_0 + \|\omega^{\delta,i}\|_{L^{\infty}(\R^n)} \leq C_0\, \eps_0 + \lambda.
        \end{align}
        On the other hand, the map $y_{\delta}$ is uniformly bi-Lispchitz. By \eqref{eq:gd_good}, we have for any $X\in\R^n$
        \begin{align*}
            \left\| |dy_{\delta}(X)|^2_g - |X|_{\geu} \right\|_{L^{\infty}(\R^n,\geu)} \leq C_0\, \eps_1\, |X|_{\geu}.
        \end{align*}
        Hence, it holds
        \begin{align*}
            |dy_{\delta}(X)|^2_g  \leq (1+C_0\, \eps_1)\, |X|_{\geu}^2 \qquad \text{ and }\qquad (1-C_0\,\eps_1)\, |X|^2_{\geu} \leq |dy_{\delta}(X)|^2_g.
        \end{align*}
        By choosing $\eps_1=1/(2C_0)$ and using \eqref{eq:Ineq_gd}, we obtain
        \begin{align*}
            \lambda^{-1} |dy_{\delta}(X)|^2_{\geu}  \leq \frac{3}{2}\, |X|_{\geu}^2 \qquad \text{ and }\qquad \frac{1}{2}\, |X|^2_{\geu} \leq \lambda\,|dy_{\delta}(X)|^2_{\geu}.
        \end{align*}
        Combining \eqref{eq:Linfty_yinv2} and the above estimate, we deduce that the maps $y_{\delta}$ are uniformly bi-Lipschitz. As well, their first derivatives are bounded in $L^{(n,2)}(\R^d)$ and second derivatives in $L^{\left(\frac{n}{2},2\right)}(\R^n)$. Up to a translation, we can pass to the limit in the relation $y_{\delta}\circ y_{\delta}^{-1} = y_{\delta}^{-1}\circ y_{\delta}=\Id_{\R^n}$ and obtain a limit $y\colon \R^n\to \R^n$ which is a bi-Lipschitz homeomorphism of $\R^n$ with $dy,d(y^{-1})\in \bar{W}^{2,\left(\frac{n}{2},2\right)}(\R^n)$.
		\end{proof}
		
		Up to a small perturbation of the diffeomorphism constructed in \Cref{th:Weak_metrics_v2}, we obtain harmonic coordinates.
		
			\begin{theorem}\label{th:Weak_metrics_harm_v2}
			There exist $\eps_2>0$ and $C_2>0$ depending only on $n$ satisfying the following property. Let $g$ be a weak metric on $\R^n$ such that $\dr_k g_{ij}\in L^{(n,2)}(\R^n)$, $\dr^2_{kl} g_{ij} \in L^{\left(\frac{n}{2},2\right)}(\R^n)$ and 
			\begin{align*}
				\|\Riem^g\|_{L^{\left(\frac{n}{2},1\right)}(\R^n,g)} \leq \eps_2.
			\end{align*}
			Then there exists a map $z\in \BiLip(\R^n)$ such that $\dr^2_{ij}z,\dr^2_{ij}(z^{-1})\in L^{\left(n,2\right)}(\R^n)$ and $\dr^3_{ijk}z,\dr^3_{ijk}(z^{-1})\in L^{\left(\frac{n}{2},2\right)}(\R^n)$ providing harmonic coordinates for $g$ on $\R^n$ and such that the following estimates hold for all $1\leq i,j\leq n$
			\begin{align*}
				\left\|(z^*g)_{ij} - \delta_{ij} \right\|_{\bar{W}^{2,\left(\frac{n}{2},1\right)}(\R^n,\geu)} \leq C_2\, \|\Riem^g\|_{L^{\left(\frac{n}{2},1\right)}(\R^n,g)}.
			\end{align*}
		\end{theorem} 
		
		\begin{proof}
			Thanks to \Cref{th:Weak_metrics_v2}, if $\eps_2<\eps_1$, then there exists $y\in \BiLip(\R^n)$ such that for all $1\leq i,j\leq n$
			\begin{align}\label{eq:controlh}
				\left\|(y^*g)_{ij} - \delta_{ij} \right\|_{\bar{W}^{2,\left(\frac{n}{2},1\right)}(\R^n,\geu)} \leq C_2\, \|\Riem^g\|_{L^{\left(\frac{n}{2},1\right)}(\R^n,g)}.
			\end{align}
			For simplicity, we will denote $h\coloneq y^*g$. For $s>0$ and $i\in\{1,\ldots,n\}$, we consider the solution $\tilde{z}^i_s\colon \B_s\to \R$ to the elliptic equation
			\begin{align*}
				\begin{cases}
					\lap_h \tilde{z}_s^i = 0 & \text{ in }\B_s,\\[2mm]
					\tilde{z}_s^i = x^i & \text{ on }\dr \B_s.
				\end{cases}
			\end{align*}
			Hence the map $\tilde{z}_s^i-x^i$ satisfies $\tilde{z}_s^i-x^i=0$ on $\dr\B_s$ and the equation
			\begin{align*}
				\lap_{\geu} (\tilde{z}_s^i-x^i) & = \dr_{\alpha}\left( \left[ \delta^{\alpha\beta} -\sqrt{\det h}\, h^{\alpha\beta} \right]\dr_{\beta} (\tilde{z}_s^i-x^i) \right) + \dr_{\alpha}\left( \sqrt{\det h}\, h^{\alpha\beta}\, \dr_{\beta} (\tilde{z}_s^i-x^i) \right) \\[2mm]
				& = \dr_{\alpha}\left( \left[ \delta^{\alpha\beta} -\sqrt{\det h}\, h^{\alpha\beta} \right]\dr_{\beta} (\tilde{z}_s^i-x^i) \right) + \dr_{\alpha}\left( \sqrt{\det h}\, h^{\alpha i} \right).
			\end{align*}
			By elliptic regularity up to the boundary, together with \eqref{eq:controlh}, we have 
			\begin{align*}
				\|d(\tilde{z}_s^i - x^i)\|_{\bar{W}^{2,\left(\frac{n}{2},1\right)}(\B_s,\geu)} \leq C(n)\, \eps_2\, \|d(\tilde{z}_s^i - x^i)\|_{\bar{W}^{2,\left(\frac{n}{2},1\right)}(\B_s,\geu)} + C(n)\, \sum_{i,j,k} \|\dr_k h_{ij}\|_{\bar{W}^{1,\left(\frac{n}{2},1\right)}(\B_s,\geu)}.
			\end{align*}
			For $\eps_2>0$ small enough so that $C(n)\, \eps_2< \frac{1}{2}$ and using \eqref{eq:controlh}, we have
			\begin{align*}
				\|d(\tilde{z}_s^i - x^i)\|_{\bar{W}^{2,\left(\frac{n}{2},1\right)}(\B_s,\geu)} \leq C(n)\, \|\Riem^g\|_{L^{\left(\frac{n}{2},1\right)}(\R^n,g)} \leq C(n)\, \eps_2.
			\end{align*}
			Letting $s\to +\infty$, we obtain a map $\tilde{z}^i\colon \R^n\to \R$ such that 
			\begin{align*}
				\|d(\tilde{z}^i - x^i)\|_{\bar{W}^{2,\left(\frac{n}{2},1\right)}(\R^n,\geu)} \leq C(n)\, \|\Riem^g\|_{L^{\left(\frac{n}{2},1\right)}(\R^n,g)} \leq C(n)\, \eps_2.
			\end{align*}
			Following the proof of Claim \ref{cl:y_bilip} (where we replace the frame $(e_{s,1}, \ldots,e_{s,n})$ by the canonical basis of $\R^n$, already constant), we obtain that for $\eps_2>0$ small enough (depending only on $n$), the map $\tilde{z}\coloneq (\tilde{z}^1,\ldots,\tilde{z}^n)$ is a $C^1$ diffeomorphism of $\R^n$ satisfying $\lap_h \tilde{z}^i=0$ on $\R^n$. Moreover, we have the following estimates for the metric $\tilde{z}^*h$
			\begin{align}
				& \left\| (\tilde{z}^*h)_{ij} - \delta_{ij} \right\|_{\bar{W}^{2,\left(\frac{n}{2},1\right)}(\R^n,\geu)} \label{eq:est_harm_coord} \\[3mm]
				& = \left\| h_{\alpha\beta}\, (\dr_i \tilde{z}^{\alpha})\, (\dr_j \tilde{z}^{\beta}) - \delta_{ij} \right\|_{\bar{W}^{2,\left(\frac{n}{2},1\right)}(\R^n,\geu)} \nonumber\\[3mm] 
				&= \left\| (h_{\alpha\beta}-\delta_{\alpha\beta})\, (\dr_i \tilde{z}^{\alpha})\, (\dr_j \tilde{z}^{\beta}) - \delta_{\alpha\beta} \left[ (\dr_i \tilde{z}^{\alpha})\, (\dr_j \tilde{z}^{\beta}) - (\dr_i x^{\alpha})\, (\dr_j x^{\beta}) \right] \right\|_{\bar{W}^{2,\left(\frac{n}{2},1\right)}(\R^n,\geu)} \nonumber \\[3mm]
				& \leq C(n)\, \sum_{\alpha,\beta} \|h_{\alpha\beta} - \delta_{\alpha\beta} \|_{\bar{W}^{2,\left(\frac{n}{2},1\right)}(\R^n,\geu)} + C(n)\, \|d\tilde{z}-I_n\|_{\bar{W}^{2,\left(\frac{n}{2},1\right)}(\R^n,\geu)} \nonumber\\[3mm]
				& \leq C(n)\, \|\Riem^g\|_{L^{\left(\frac{n}{2},1\right)}(\R^n,g)}. \nonumber
			\end{align}
            The map $z\in\BiLip(\R^n)$ is given by $z\coloneq \tilde{z}\circ y$.
		\end{proof}
		
		\section{An extension problem}\label{sec:extension}
		
		In this section, we prove \Cref{th:local}. The boundary conditions in \eqref{eq:boundary} allows to extend the coefficients $g_{ij}-\delta_{ij}$ by some $\gamma_{ij}$ in a controlled manner. By gluing $\delta_{ij}+\gamma_{ij}$ to $\delta_{ij}$, we preserve a weak metric on $\R^n$ with some curvature in $L^{\left(\frac{n}{2},1\right)}(\R^n)$. We can then apply \Cref{th:Weak_metrics_harm}.
		
		\begin{proof}[Proof of \Cref{th:local}]
			We consider the trace operator:
			\begin{align*}
				T\colon \left| \begin{array}{c c c}
					W^{2+\frac{2}{n},\frac{n}{2}}(\B_2\setminus \B_1) & \to & W^{2,\frac{n}{2}}(\s^{n-1}_2\cup \s^{n-1}_1)\times W^{1,\frac{n}{2}}(\s^{n-1}_2\cup \s^{n-1}_1) \\[2mm]
					u & \mapsto & (u,\dr_r u)
				\end{array}
				 \right. 
			\end{align*}
			It has a bounded right-inverse (an extension operator, see for instance \cite{lamberti2020})
			\begin{align*}
				E\colon 
					W^{2,\frac{n}{2}}(\s^{n-1}_2\cup \s^{n-1}_1)\times W^{1,\frac{n}{2}}(\s^{n-1}_2\cup \s^{n-1}_1)  \to W^{2+\frac{2}{n},\frac{n}{2}}(\B_2\setminus \B_1) .
			\end{align*} 
			We define the maps $\kappa_{ij},\tau_{ij} \colon \s^{n-1}_2\cup \s^{n-1}_1\to \R$ as follows
			\begin{align*}
				\kappa_{ij} \coloneq \begin{cases}
					0 & \text{ on }\s^{n-1}_2,\\[1mm]
					g_{ij}-\delta_{ij} & \text{ on }\s^{n-1}_1.
				\end{cases}, \qquad \tau_{ij} \coloneq \begin{cases}
				0 & \text{ on }\s^{n-1}_2,\\[1mm]
				\dr_r g_{ij} & \text{ on }\s^{n-1}_1.
				\end{cases},
			\end{align*}
			We consider the extension
			\begin{align*}
				\gamma_{ij} \coloneq E(\kappa_{ij},\tau_{ij}).
			\end{align*}
			Thanks to \eqref{eq:boundary}, we have 
			\begin{align}\label{eq:Sob_ext}
            \begin{aligned} 
				\|\gamma_{ij} \|_{W^{2+\frac{2}{n},\frac{n}{2}}(\B_2\setminus \B_1)} & \leq C(n)\left( \|g_{ij}-\delta_{ij}\|_{W^{2,\frac{n}{2}}(\s^{n-1})} +  \|\dr_r g_{ij} \|_{W^{1,\frac{n}{2}}(\s^{n-1})} \right)\\[2mm]
				& \leq C(n)\, \eps. 
            \end{aligned}
			\end{align}
			Since $W^{2+\frac{2}{n},\frac{n}{2} }(\B_2\setminus \B_1)\hookrightarrow L^{\infty}(\B_2\setminus \B_1)$, we obtain 
			\begin{align}\label{eq:Linfty_ext}
				\|\gamma_{ij} \|_{L^{\infty}(\B_2\setminus \B_1)} \leq C(n)\, \eps.
			\end{align}
			Let $\chi\in C^{\infty}(\B_2\setminus \B_1;[0,1])$ be a cut-off function such that $\chi=1$ in $\s^{n-1}_1$ and $\chi=0$ on $\s^{n-1}_2$. We consider the weak metric
			\begin{align*}
				\tilde{g}_{ij} \coloneq \begin{cases}
					\delta_{ij} & \text{ in }\R^n\setminus \B_2,\\[1mm]
					\delta_{ij} + \chi\, \gamma_{ij} & \text{ in } \B_2\setminus \B_1,\\[1mm]
					g_{ij} & \text{ in } \B_1.
				\end{cases}
			\end{align*}
			If $\eps>0$ is small enough, we obtain from \eqref{eq:Sob_ext} and \eqref{eq:Linfty_ext} that $\tilde{g}_{ij}$ defines a weak metric on $\R^n$. Using \eqref{eq:Riem_coordinates}, we have
			\begin{align}
				& \left\| \Riem^{\tilde{g}} \right\|_{L^{\left(\frac{n}{2},1\right)}(\R^n,\tilde{g})}  \nonumber \\[3mm]
				& \leq C(n)\left(\left\| \Riem^{\tilde{g}} \right\|_{L^{\left(\frac{n}{2},1\right)}(\R^n\setminus \B_2,\tilde{g})} + \left\| \Riem^{\tilde{g}} \right\|_{L^{\left(\frac{n}{2},1\right)}(\B_2\setminus \B_1,\tilde{g})} + \left\| \Riem^{\tilde{g}} \right\|_{L^{\left(\frac{n}{2},1\right)}(\B_1,\tilde{g})}\right) \nonumber  \\[3mm]
				& \leq C(n)\left( \sum_{i,j} \|\g^2 \tilde{g}_{ij}\|_{L^{\left(\frac{n}{2},1\right)}(\B_2\setminus \B_1,\geu) } + \sum_{i,j} \|\g \tilde{g}_{ij}\|_{L^{(n,2)}(\B_2\setminus \B_1,\geu) }^2 + \left\| \Riem^g \right\|_{L^{\left(\frac{n}{2},1\right)}(\B_1,g)}  \right) \nonumber  \\[3mm]
				& \leq C(n)\left( \sum_{i,j} \|g_{ij}-\delta_{ij}\|_{W^{2,\frac{n}{2}}(\s^{n-1})} + \sum_{i,j}  \|\dr_r g_{ij} \|_{W^{1,\frac{n}{2}}(\s^{n-1})} +  \left\| \Riem^g \right\|_{L^{\left(\frac{n}{2},1\right)}(\B_1,g)} \right) \label{eq:est_Riem_ext} \\[3mm]
				& \leq C(n)\, \eps. \nonumber 
			\end{align}
			If $\eps>0$ is small enough, we obtain $C(n)\, \eps<\eps_0$ and we can apply \Cref{th:Weak_metrics_harm}. There exists $f\in \BiLip(\R^n)$ providing harmonic coordinates for $\tilde{g}$ and such that 
			\begin{align*}
				\left\| (f^*\tilde{g})_{ij} - \delta_{ij} \right\|_{\bar{W}^{2,\left(\frac{n}{2},1\right)}(\R^n,\geu)} \leq C_0 \, \left\| \Riem^{\tilde{g}} \right\|_{L^{\left(\frac{n}{2},1\right)}(\R^n,\tilde{g})}.
			\end{align*}
			Restricting the left-hand side to $f^{-1}(\B_1)$, and using \eqref{eq:est_Riem_ext}, we obtain 
			\begin{align*}
				& \sum_{i,j} \left\| (f^*g)_{ij} - \delta_{ij} \right\|_{\bar{W}^{2,\left(\frac{n}{2},1\right)}(f^{-1}(\B),\geu)} \\[3mm]
				& \leq  C(n)\left( \sum_{i,j} \|g_{ij}-\delta_{ij}\|_{W^{2,\frac{n}{2}}(\s^{n-1})} + \sum_{i,j}  \|\dr_r g_{ij} \|_{W^{1,\frac{n}{2}}(\s^{n-1})} +  \left\| \Riem^g \right\|_{L^{\left(\frac{n}{2},1\right)}(\B_1,g)} \right).
			\end{align*}
		\end{proof}
		
		\section{Application to immersions of $\R^n$}\label{sec:Imm}
		
		In this section, we consider immersions $\vPhi\colon \Sigma\to \R^d$, where $n\geq 4$ is even and $\Sigma$ will be an $n$-dimensional complete manifold possibly with boundary. In \Cref{sec:Sobolev}, we will study the Sobolev constants on the manifold $(\Sigma,g_{\vPhi})$, where $g_{\vPhi}$ is the induced metric of $\vPhi$, and $\Sigma$ has non-empty boundary. In \Cref{sec:Riem_imm}, we will study the case of immersions such that $\Er_n(\vPhi)<+\infty$. We prove that in this case, it holds $\Riem^{g_{\vPhi}}\in L^{\left(\frac{n}{2},1\right)}(\Sigma,g_{\vPhi})$. This implies \Cref{th:Imm} for $\Sigma=\R^n$. In \Cref{sec:Local_Imm}, we consider the case $\Sigma=\B^n$ with $\vPhi(\dr\B^n)$ parametrizing a graph and study the extension of such immersions. In order to study the regularity of immersions having the least possible regularity to define $\Er_n(\vPhi)$, we consider the notion of weak immersions, first developed by the second author, see for instance \cite{riviere2016,Lan2025}. We denote $\Imm(\Sigma;\R^d)$ the space of $C^{\infty}$ immersions of $\Sigma$ into $\R^d$. We now define the space of weak immersions.
	
	\begin{defi}
		Let $(\Sigma,h)$ be a closed orientable $n$-dimensional Riemannian manifold and $d>n$ be an integer. Given $k\in\N$ and $p\in[1,+\infty]$, we define the notion of weak immersion $\I_{k,p}(\Sigma;\R^d)$ as follows:
		\begin{align*}
			\I_{k,p}(\Sigma;\R^d)\coloneqq \left\{
			\vPhi\in W^{k+2,p}(\Sigma;\R^d) : \exists c_{\vPhi}>0,\ c^{-1}_{\vPhi} h \leq g_{\vPhi} \leq c_{\vPhi}\, h
			\right\}.
		\end{align*}    
        We also have the Sobolev--Lorentz version, for $q\in[1,+\infty]$
        \begin{align*}
			\I_{k,(p,q)}(\Sigma;\R^d)\coloneqq \left\{
			\vPhi\in W^{k+2,(p,q)}(\Sigma;\R^d) :  \exists c_{\vPhi}>0,\ c_{\vPhi}^{-1} h \leq g_{\vPhi} \leq c_{\vPhi}\, h
			\right\}.
		\end{align*}    
	\end{defi}
    As explained in the introduction, a weak immersion $\vPhi\in \I_{k,(p,q)}(\Sigma;\R^d)$ has a well-defined notion of second fundamental form $\vII_{\vPhi}$ whose coefficient in some coordinates where $\vPhi\in W^{k+2,(p,q)}(\B^n)$ lie in the Lorentz--Sobolev space $W^{k,(p,q)}(\Sigma)$.
    
		\subsection{Sobolev injections for immersed submanifolds}\label{sec:Sobolev}
		
		In this section, we prove that the constant in the Sobolev injections $W^{1,(p,q)}(\Sigma,g_{\vPhi})\hookrightarrow L^{\left( \frac{np}{n-p} , q\right) }(\Sigma,g_{\vPhi})$ can be taken independent of the immersion $\vPhi$ if the $L^{(n,\infty)}$-norm of $\vH_{\vPhi}$ is small enough. To do so, we rely on the Sobolev inequality for submanifolds of $\R^n$ proved by Brendle \cite{brendle2021}. The goal of this section is to prove the following result.
		\begin{theorem}\label{th:Sobolev}
            Let $d>n\geq 3$ be integers. 
			There exists $\eps_*=\eps(n,d)>0$ such that the following holds. Let $\vPhi\colon \Sigma\to \R^d$ be a $C^{\infty}$ immersion from a compact smooth orientable manifold $\Sigma^n$ with $\dr\Sigma\neq \emptyset$ and $g_{\vPhi} \coloneqq \vPhi^*\geu $. Assume that $\|\vH_{\vPhi}\|_{L^{(n,\infty)}(\Sigma,g_{\vPhi})}\leq \eps_*$.\\
			Let $p\in[1,n)$ and $q\in[1,+\infty]$. There exists a constant $c>0$ depending only on $n,d,p,q$ such that 
			\begin{align*}
				\forall \vp\in C^{\infty}_c(\Sigma),\qquad \|\vp\|_{L^{\left(\frac{np}{n-p},q\right)}\left(\Sigma,g_{\vPhi} \right)} \leq c\, \|d\vp\|_{L^{(p,q)}\left( \Sigma,g_{\vPhi} \right)}.
			\end{align*}
		\end{theorem}
		
		In order to prove \Cref{th:Sobolev}, we start by proving the case $p=1$ and $q=\frac{n}{n-1}$ as a consequence of \eqref{eq:isop}. By a good choice of test function, we will recover the standard Sobolev injections $W^{1,p}\hookrightarrow L^{p^*}$. \Cref{th:Sobolev} is then a consequence of an interpolation inequality.

		\begin{proof}
             We will denote $\B^m$ the unit ball of $\R^m$ and $\Hr^m$ the Lebesgue measure. We denote 
		\begin{align*}
			L\coloneqq n\left(\frac{d}{d-n}\, \frac{\Hr^d(\B^d)}{\Hr^{d-n}(\B^{d-n})}\right)^{\frac{1}{n}}.
		\end{align*}
		We start from the Sobolev inequality obtained in \cite[Theorem 1]{brendle2021}\footnote{The result is stated for submanifolds, but the proof is actually the same for immersed manifolds of class $C^{\infty}$.}: for any $\vp\in C^{\infty}(\Sigma;[0,+\infty))$, 
		\begin{align}\label{eq:isop}
			L\, \left(\int_{\Sigma} \vp^{\frac{n}{n-1}}\, d\vol_{g_{\vPhi}} \right)^{\frac{n-1}{n}} \leq \int_{\dr\Sigma} \vp\, d\vol_{g_{\vPhi}} + \int_{\Sigma} \sqrt{ |d\vp|^2_{g_{\vPhi}} + n^2\, \vp^2\, |H_{\vPhi}|^2 }\ d\vol_{g_{\vPhi}}.
		\end{align}
        
			Let $U\subset \Sigma$ be a smooth open set. Then $\vPhi\colon U\to \R^d$ is a smooth immersion and we have, for any $\vp\in C^{\infty}(U)$ with $\vp\geq 0$, that

			\begin{align*}
				L\, \left(\int_{U} \vp^{\frac{n}{n-1}}\, d\vol_{g_{\vPhi}} \right)^{\frac{n-1}{n}} \leq \int_{\dr U} \vp\, d\vol_{g_{\vPhi}} + \int_{U} \sqrt{ |d\vp|^2_{g_{\vPhi}} + n^2\, \vp^2\, |H_{\vPhi}|^2 }\ d\vol_{g_{\vPhi}}.
			\end{align*}

			Taking $\vp=1$, we obtain the following isoperimetric inequality
			\begin{align}\label{eq:Isop}
				L\, \vol_{g_{\vPhi}}(U)^{\frac{n-1}{n}} \leq \vol_{g_{\vPhi}}(\dr U) + \int_U |H_{\vPhi}|\, d\vol_{g_{\vPhi}}.
			\end{align}
			Let $\vp\in C^{\infty}_c(\Sigma)$. We have 
			\begin{align*}
				\|\vp\|_{L^{\left( \frac{n}{n-1} ,1\right) }\left(\Sigma,g_{\vPhi} \right)} = \frac{n-1}{n}\int_0^{+\infty} \vol_{g_{\vPhi}}\left(\left\{ \vp > t\right\}\right)^{\frac{n-1}{n}}\ dt.
			\end{align*}
			Since $\vp$ is smooth, the set of regular values has full measure thanks to Sard's theorem. Since $\vp$ has compact support, for such values $t$, it holds $\{\vp >t\}\Subset \Sigma$, so that $\dr\{\vp >t\} =  \{\vp=t\}$. We apply the isoperimetric inequality \eqref{eq:Isop}:
			\begin{align*}
				\|\vp\|_{L^{\left(\frac{n}{n-1},1\right)}\left(\Sigma,g_{\vPhi} \right)} \leq \frac{n}{L\, (n-1)} \int_0^{+\infty} \left( \vol_{g_{\vPhi}}\big(\{ \vp = t \}\big) + \int_{\{\vp>t\}} |\vH_{\vPhi}|\,  d\vol_{g_{\vPhi}}\right)\, dt.
			\end{align*}
			Thanks to the coarea formula, we obtain
			\begin{align*}
				\int_0^{+\infty} \left( \vol_{g_{\vPhi}}\big(\{ \vp = t \}\big) + \int_{\{\vp>t\}} |\vH_{\vPhi}|\,  d\vol_{g_{\vPhi}}\right)\, dt = \int_{\Sigma} |d\vp|_{g_{\vPhi}}\, d\vol_{g_{\vPhi}} + \int_{\Sigma} |\vp|\, |\vH_{\vPhi}|\, d\vol_{g_{\vPhi}}.
			\end{align*}
			By duality $L^{\left(\frac{n}{n-1},1\right)}-L^{(n,\infty)}$, we obtain
			\begin{align*}
				\|\vp\|_{L^{\left(\frac{n}{n-1},1\right)}(\Sigma,g_{\vPhi})} \leq \frac{n}{L\, (n-1)}\left(\int_{\Sigma} |d\vp|_{g_{\Phi}}\, d\vol_{g_{\vPhi}} + \|\vp\|_{L^{\left(\frac{n}{n-1},1\right)}(\Sigma,g_{\vPhi})} \|\vH_{\vPhi}\|_{L^{(n,\infty)}(\Sigma,g_{\vPhi})} \right).
			\end{align*}
			If $\|\vH_{\vPhi}\|_{L^{(n,\infty)}(\Sigma,g_{\vPhi})} < L\frac{n-1}{2\, n}$, then we obtain 
			\begin{align*}
				\forall \vp\in C^{\infty}_c(\Sigma),\qquad \|\vp\|_{L^{\left(\frac{n}{n-1},1\right)}(\Sigma,g_{\vPhi})} \leq \frac{2n}{L\, (n-1)} \|d\vp\|_{L^1(\Sigma,g_{\vPhi})}.
			\end{align*}
			let $p>1$ and consider now the choice $\vp = \psi^s$ for $s=\frac{p^*}{1^*} = \frac{np}{n-p}\frac{n-1}{n}= p \frac{n-1}{n-p}>p$. We obtain
			\begin{align*}
				\|\psi \|_{L^{\left(\frac{ ns }{n-1}, s\right)}\left(\Sigma,g_{\vPhi} \right)}^s \leq \frac{2n\, L\, s}{n-1} \int_{\Sigma} \psi^{s-1} |d\psi|_{g_{\vPhi}}\, d\vol_{g_{\vPhi}}.
			\end{align*}
			We compute the exponents:
			\begin{align*}
				\frac{ns}{n-1} = \frac{np}{n-p}, & & s-1 = \frac{n(p-1)}{n-p}.
			\end{align*}
			Hence, we obtain
			\begin{align*}
				\|\psi \|_{L^{\left(\frac{ np }{n-p}, p\frac{n-1}{n-p}\right)}\left(\Sigma,g_{\vPhi} \right)}^s & \leq \frac{2n\, L\, s}{n-1}\, \|d\psi\|_{L^p\left(\Sigma,g_{\vPhi} \right)}\, \|\psi\|_{L^{p\frac{s-1}{p-1}}\left(\Sigma,g_{\vPhi} \right)}^{s-1}\\[2mm]
				& \leq \frac{2n\, L\, s}{n-1}\, \|d\psi\|_{L^p\left(\Sigma,g_{\vPhi} \right)}\, \|\psi\|_{L^{\frac{np}{n-p}}\left(\Sigma,g_{\vPhi} \right)}^{s-1}.
			\end{align*}
			We obtain the Sobolev injection $W^{1,p}\hookrightarrow L^{\left(\frac{np}{n-p},\frac{p(n-1)}{n-p}\right)}$, but the Lorentz exponent is too large $p\frac{n-1}{n-p}>p$. Hence, we will only record the standard Sobolev inequality
			\begin{align*}
				\forall \psi\in C_c^{\infty}(\Sigma),\qquad \|\psi \|_{L^{\frac{ np }{n-p}}\left(\Sigma,g_{\vPhi} \right)} \leq C(n,d,p)\, \|d\psi\|_{L^p\left(\Sigma,g_{\vPhi} \right)}.
			\end{align*}
			We now recover all the Sobolev injections by interpolation. Indeed, the identity map $I\colon W^{1,p}(\Sigma,g_{\vPhi})\to L^{\frac{np}{n-p}}(\Sigma,g_{\vPhi})$ is a bounded linear operator. By \cite[Theorem 1.12]{bennett1988}, we obtain that for any $1\leq p_1<p_2<n$, $\theta\in(0,1)$ and $q\in[1,+\infty)$, the map $I\colon (W^{1,p_1},W^{1,p_2})_{\theta,q}\to (L^{p_1^*},L^{p_2^*})_{\theta,q}$ is a bounded operator. Let $p\in(p_1,p_2)$ and $\theta\in(0,1)$ such that
			\begin{align*}
				\frac{1}{p} = \frac{1-\theta}{p_1} + \frac{\theta}{p_2}.
			\end{align*}	
			Then we have 
			\begin{align*}
				\frac{1}{p^*} = \frac{1-\theta}{p_1^*} + \frac{\theta}{p_2^*}.
			\end{align*}
			Thus, we obtain that $I\colon W^{1,(p,q)}(\Sigma,g_{\vPhi})\to L^{(p^*,q)}(\Sigma,g_{\vPhi})$ is a bounded operator, with constant depending only on $p_1,p_2,\theta,q,n,d$.
		\end{proof}

		\subsection{Regularity of the Riemann tensor for immersions}\label{sec:Riem_imm}
		
		In this section, we consider an immersion $\vPhi\colon \Sigma^n\to \R^d$ with $n\geq 4$ even. We establish the regularity of the induced metric $g_{\vPhi}$ and its Riemann tensor under the assumption $\Er_n(\vPhi)<+\infty$. To define $\Er_n(\vPhi)$, we need \textit{a priori} that $\vII_{\vPhi}\in W^{\frac{n}{2}-1,2}(\Sigma,g_{\vPhi})$, that is to say, the generalized Gauss map $\vn_{\vPhi}$ lies in $W^{\frac{n}{2},2}(\Sigma,g_{\vPhi})$. This requires at least that $\vPhi$ lies in $W^{\frac{n}{2}+1,2}(\Sigma)$.
		
		\begin{lemma}
			Let $n\geq 4$ be an integer and $\vPhi\in \I_{\frac{n}{2}-1,2}(\B^n;\R^d)$ be a weak immersion. We then have
			\begin{align*}
				\big( g_{\vPhi} \big)_{ij} \in W^{\frac{n}{2},2}(\B^n)\hookrightarrow W^{2,\left(\frac{n}{2},2\right)}(\B^n) \hookrightarrow W^{1,(n,2)}(\B^n).
			\end{align*}
		\end{lemma}
		\begin{proof}
			We have $d\vPhi\in W^{\frac{n}{2},2}(\B^n)\cap L^{\infty}(\B^n)$. Thus, we have 
			\begin{align*}
				\left(g_{\vPhi} \right)_{\alpha\beta} = \dr_{\alpha} \vPhi \cdot \dr_{\beta} \vPhi \in W^{\frac{n}{2},2}(\B^n).
			\end{align*}
			If $n=4$, we have $W^{\frac{n}{2},2}=W^{2,\left( \frac{n}{2},2\right)}$. If $n>4$, we use the Sobolev embedding $W^{\frac{n}{2},2}(\B^n)\hookrightarrow W^{2,(p,2)}(\B^n)$, where $p>1$ is given by the following formula:
			\begin{align*}
				\frac{1}{2} - \frac{n}{2n} = \frac{1}{p} - \frac{2}{n}.
			\end{align*}
			That is to say, $p=\frac{n}{2}$.
		\end{proof}
		
		Therefore, the metric $g_{\vPhi}$ of a given weak immersion $\vPhi\in \I_{\frac{n}{2}-1,2}(\Sigma^n;\R^d)$ might not be continuous for an arbitrary choice of coordinates. However, by definition of $\I_{\frac{n}{2}+1,2}(\Sigma^n;\R^d)$, the matrix $\big( g_{\vPhi} \big)_{ij}$ is invertible. Hence, we can define the Christoffel symbols and the curvature tensors by the usual formulas. One can also define the Riemann tensor of $g_{\vPhi}$ thanks to the Gauss--Codazzi equation:
		\begin{align}\label{eq:Gauss_Codazzi}
			\Riem^{g_{\vPhi}}_{ijkl} = \vII_{ik}\cdot \vII_{jl} - \vII_{il}\cdot \vII_{jk} .
		\end{align}
		
		\begin{lemma}\label{lm:Int_II}
			For weak immersion $\vPhi\in \I_{\frac{n}{2}-1,2}(\Sigma^n;\R^d)$, we have $\Er_n(\vPhi)<+\infty$ and $\vII_{\vPhi}\in L^{(n,2)}\left(\Sigma,g_{\vPhi}\right)$.
		\end{lemma}
		
		\begin{proof}
			By Sobolev embeddings, we have that $W^{\frac{n}{2}-1,2}(\Sigma^n)\hookrightarrow W^{k,p}(\Sigma^n)$ for every $\frac{n}{2}-1>k\geq 0$ and $p>2$ such that 
			\begin{align*}
				\frac{1}{n}=\frac{1}{2} - \frac{\frac{n}{2}-1}{n} = \frac{1}{p} - \frac{k}{n}.
			\end{align*}
			We consider the sequence $k_i = \left(\frac{n}{2}-1\right)-i$ for $0\leq i\leq \frac{n}{2}-1$. The associated exponent $p_i$ is given by
			\begin{align*}
				\frac{1}{p_i} = \frac{1}{n}+\frac{k_i}{n} = \frac{\frac{n}{2}-i}{n}.
			\end{align*}
			Given $\vPhi\in \I_{\frac{n}{2}+1,2}(\Sigma;\R^d)$, we obtain 
			\begin{align*}
				\vII \in W^{\frac{n}{2}-1,2} \hookrightarrow W^{\frac{n}{2}-2,\left( \frac{n}{\frac{n}{2}-1},2\right)} \hookrightarrow \cdots \hookrightarrow W^{j,\left( \frac{n}{j+1},2\right)} \hookrightarrow\cdots \hookrightarrow L^{\left(n,2 \right)}.
			\end{align*}
		\end{proof}
		
		Let $\vPhi\in \I_{\frac{n}{2}-1,2}(\Sigma^n;\R^d)$. The Riemann tensor of $g_{\vPhi}$ is given in coordinates by the formula \eqref{eq:Gauss_Codazzi}
		\begin{align*}
			\Riem^{g_{\vPhi}}_{ijkl} = \vII_{ik}\cdot \vII_{jl} - \vII_{il}\cdot \vII_{jk} \in W^{\frac{n}{2}-1,\left(\frac{n}{\frac{n}{2}+1},1\right)}\left(\Sigma,g_{\vPhi} \right).
		\end{align*}
		Thus, we obtain $\Riem^{g_{\Phi}}\in L^{\left( \frac{n}{2},1 \right)}(\Sigma^n)$ with the estimate
		\begin{align}\label{eq:Riem_imm_Lorentz}
			\left\| \Riem^{g_{\vPhi}} \right\|_{L^{\left(\frac{n}{2},1\right)}(\Sigma,g_{\vPhi})} \leq 2\, \left\| \vII_{\vPhi} \right\|_{L^{(n,2)}(\Sigma,g_{\vPhi})}^2.
		\end{align}
		The above estimate actually makes sense for $\vPhi\in \I_{0,(n,2)}(\Sigma^n;\R^d)$. If $\|H_{\vPhi}\|_{L^{(n,\infty)}(\Sigma,g_{\vPhi})} \leq \eps_*$, where $\eps_*$ is defined in \Cref{th:Sobolev}, we obtain 
		\begin{align}\label{eq:Riem_imm}
			\left\| \Riem^{g_{\vPhi}} \right\|_{L^{\left(\frac{n}{2},1\right)}(\Sigma,g_{\vPhi})} \leq C(n)\, \sum_{i=0}^{\frac{n}{2}-1} \left( \int_{\Sigma} \big| \g^i \vII_{\vPhi} \big|_{g_{\vPhi}}^{\frac{n}{i+1}}\ d\vol_{g_{\vPhi}}\right)^{2\, \frac{i+1}{n}}.
		\end{align} 
		
		\subsection{Immersions of the unit ball}\label{sec:Local_Imm}
		
		In this section, we prove a local version of the existence of harmonic coordinates for immersions in ${\mathcal I}_{0,(n,2)}(\B^n,{\mathbb R}^d)$ assuming the boundary of 
        $\partial\B$ is sent to the boundary of a ball and realizes an ``almost flat graph'' in this sphere. To do so we are going to use the extension result from the proof of \Cref{th:local}.  We consider the case of immersions $\vPhi\colon \B^n\to \R^d$ such that its restriction $\vPhi\colon \s^{n-1}\to \s^{d-1}$ parametrizes a small perturbation of a slice $\Pc\cap \s^{d-1}$ for some affine $n$-dimensional plane $\Pc\subset \R^d$.

		\begin{theorem}\label{th:local_imm_v2}
			Let $d\geq 1$ and $n\geq 3$ be integers, $K>1$ and $r_0\in(0,1)$. There exists $\eps_4>0$ and $C_4>0$ depending only on $n$, $d$, $K$ and $r_0$ such that the following holds for any $\eps\in(0,\eps_4)$. 
			Let $\Pc\coloneq \vq + \R^n\times \{0\}^{d-n}\subset \R^d$ for some $\vq\in \{0\}^n\times \B^{d-n}$ such that $|\vq| \in [0,r_0)$. Let $\vPhi\in \I_{0,(n,2)}(\B^n;\R^d)$ be an immersion such that 
			\begin{enumerate}
				\item It holds $\|\vII_{\vPhi}\|_{L^{(n,2)}(\B^n,g_{\vPhi})}<\eps$.
				\item The map $\vPhi\colon \s^{n-1}\to \s^{d-1}$ parametrizes a graph $\left\{ \vtheta + \vphi(\vtheta) : \vtheta\in \Pc\cap \s^{d-1} \right\}$, where $\vphi\colon \s^{n-1}\to \R^d$ verifies 
				\begin{align}\label{hyp:graph}
					\eps^{-1}\left( \|\vphi\|_{L^{\infty}(\s^{n-1})} + \|\g \vphi\|_{L^{\infty}(\s^{n-1})}\right) + \|\g^2 \vphi\|_{L^n(\s^{n-1})} \leq K.
				\end{align}

				\item The vector field $\vtau \coloneq \proj_{T\vPhi(\B)}(\vPhi)\big|_{\s^{n-1}}$ verifies
				\begin{align}\label{hyp:tangent}
					\left\|\vtau - \Id_{\Pc\cap \s^{d-1}} \right\|_{L^{\infty}(\s^{n-1})} + \left\|\g \left( \vtau- \Id_{\Pc\cap \s^{d-1}} \right) \right\|_{L^n(\s^{n-1})} \leq K\, \eps.
				\end{align}
			\end{enumerate}
			Then, there exists an immersion $\vPsi\colon \R^n\to \R^d$ such that $\vPsi(x) = \vq + \begin{pmatrix}
				x \\ 0
			\end{pmatrix}$ for $x\in \R^n\setminus \B_2$ with $\vPsi=\vPhi$ in $\B$ and 
			\begin{align*}
				\| \vII_{\vPsi}\|_{L^{(n,2)}(\R^n,g_{\vPsi})} \leq C_4\, \eps^{\frac{1}{n}}.
			\end{align*}
			Moreover, there exists a map $z\in \BiLip(\R^n)$ providing harmonic coordinates for $g_{\vPhi}$ on $\B^n$ such that for all $1\leq i,j\leq n$, it holds
			\begin{align}\label{eq:control_coordinates}
				\left\| \big( g_{\vPhi\circ z} \big)_{ij} - \delta_{ij} \right\|_{W^{2,\left(\frac{n}{2},1\right)}(z^{-1}(\B),\geu)} \leq C_4\, \eps^{\frac{2}{n}}.
			\end{align}
		\end{theorem}
		
		\begin{remark}
			\begin{enumerate}
				\item The affine plane $\Pc$ intersects $\s^{d-1}$ if and only if $|\vq|<1$. The bound $|\vq|<r_0$ amounts to ask that $\Pc\cap \s^{d-1}$ is an $(n-1)$-sphere with radius bounded from below.
				
				\item The assumptions \eqref{hyp:graph} and \eqref{hyp:tangent} imply some regularity on the domain $z^{-1}(\B)$. Indeed, they imply that $\|(g_{\vPhi})_{ij} - \delta_{ij}\|_{L^{\infty}(\s^{n-1})}\leq C(n,K)\, \eps$. Together with \eqref{eq:control_coordinates}, we obtain by proceeding in the same manner as in \eqref{eq:est_harm_coord} that 
				\begin{align*}
					\| \dr_i z\cdot \dr_j z - \delta_{ij} \|_{L^{\infty}(\s^{n-1})} \leq C(n,K)\, \eps^{\frac{2}{n}}.
				\end{align*}
				In other words, the matrix $(\dr_i z^j)_{1\leq i,j\leq n}$ is uniformly close to $O(n)$, hence its distortion is controlled. Since the same property holds for $z^{-1}$, we obtain that the boundary $\dr\left( z^{-1}(\B) \right) $ has controlled Lipschitz constant.
			\end{enumerate}
		\end{remark}
		
		For later purposes in \cite{MarRiv20252} We shall prove in fact a slightly stronger version of the first part of \Cref{th:local_imm_v2} --- that we state as a lemma --- assuming only that $\vPhi$ is defined on an arbitrary $n-$dimensional manifold $\Sigma^n$ such that $\dr\Sigma^n = \s^{n-1}$. Then we apply it to the case $\Sigma=\B^n$ and prove \Cref{th:local_imm_v2}.
		
		\begin{lemma}\label{lm:extension}
			Let $d\geq 1$ and $n\geq 3$ be integers, $K>1$ and $r_0\in(0,1)$. There exists $\eps_4>0$ and $C_4>0$ depending only on $n$, $d$, $K$ and $r_0$ such that the following holds for any $\eps\in(0,\eps_4)$. 
			Let $\Pc\coloneq \vq + \R^n\times \{0\}^{d-n}\subset \R^d$ for some $\vq\in \{0\}^n\times \B^{d-n}$ such that $|\vq| \in [0,r_0)$. Let $\vPhi\in \I_{0,(n,2)}(\Sigma;\R^d)$ be an immersion such that 
			\begin{enumerate}
				\item $\Sigma$ is an orientable compact $n$-dimensional manifold with boundary $\dr\Sigma = \s^{n-1}$.
				
				\item It holds $\|\vII_{\vPhi}\|_{L^{(n,2)}(\Sigma,g_{\vPhi})}<\eps$.
				
				\item The map $\vPhi\colon \s^{n-1}\to \s^{d-1}$ parametrizes a graph $\left\{ \vtheta + \vphi(\vtheta) : \vtheta\in \Pc\cap \s^{d-1} \right\}$, where $\vphi\colon \s^{n-1}\to \R^d$ verifies 
				\begin{align}\label{hyp:graph2}
					\eps^{-1}\left( \|\vphi\|_{L^{\infty}(\s^{n-1})} + \|\g \vphi\|_{L^{\infty}(\s^{n-1})}\right) + \|\g^2 \vphi\|_{L^n(\s^{n-1})} \leq K.
				\end{align}			
				
				\item The vector field $\vtau \coloneq \proj_{T\vPhi(\B)}(\vPhi)\big|_{\s^{n-1}}$ verifies
				\begin{align}\label{hyp:tangent2}
					\left\|\vtau - \Id_{\Pc\cap \s^{d-1}} \right\|_{L^{\infty}(\s^{n-1})} + \left\|\vtau- \Id_{\Pc\cap \s^{d-1}}  \right\|_{W^{1,n}(\s^{n-1})} \leq K\, \eps.
				\end{align}
			\end{enumerate}
			Let $\tilde{\Sigma}$ be the manifolds obtained by gluing $\Sigma$ with $\R^n\setminus \B^n$ along their common boundary $\s^{n-1}$. Then, there exists an immersion $\vPsi\colon \tilde{\Sigma}\to \R^d$ such that $\vPsi(\R^n\setminus \B_2) = (\R^n\setminus \B_2)\times \{0\}^{d-n}$, $\vPsi=\vPhi$ in $\B_1$ and 
			\begin{align}\label{eq:est_extension}
				\| \vII_{\vPsi}\|_{L^{(n,2)}(\tilde{\Sigma},g_{\vPsi})} \leq C_4\, \eps^{\frac{1}{n}}.
			\end{align}
		\end{lemma}
		
		\begin{proof}
			Up to dilation and translation, we can change $\Pc\cap \s^{d-1}$ in $\s^{n-1}\times \{0\}^{d-n}$, since both the energy $\Er_n$ and the $L^{(n,2)}$-norm of the second fundamental form are scale-invariant. \\
			
			\emph{Step 1: Definition of the extension.}\\
			We consider the trace operator:
			\begin{align*}
				\tr \colon \left| \begin{array}{c c c}
					W^{2+\frac{1}{n},n}(\B_2\setminus \B_1) & \to & W^{2,n}(\s^{n-1}_2\cup \s^{n-1}_1)\times W^{1,n}(\s^{n-1}_2\cup \s^{n-1}_1) \\[2mm]
					u & \mapsto & (u,\dr_r u)
				\end{array}
				\right. 
			\end{align*}
			It has a bounded right-inverse (an extension operator, see for instance \cite{lamberti2020}, \cite[Theorem 14.1.8]{agranovich2015}, \cite[Theorem 1.11]{amrouche} or \cite{jonsson1984})
			\begin{align*}
				E\colon W^{2,n}(\s^{n-1}_2\cup \s^{n-1}_1) \times W^{1,n}(\s^{n-1}_2\cup \s^{n-1}_1)  \to W^{2+\frac{1}{n},n}(\B_2\setminus \B_1) .
			\end{align*} 
			It also defines a bounded operator
			\begin{align*}
				E\colon W^{2-\frac{1}{n},n}(\s^{n-1}_2\cup \s^{n-1}_1) \times W^{1-\frac{1}{n},n}(\s^{n-1}_2\cup \s^{n-1}_1)  \to W^{2,n}(\B_2\setminus \B_1) .
			\end{align*} 
			Hence there exists a map $\vpsi'\in W^{2+\frac{1}{n},n}(\B_2\setminus \B_1;\R^d)$ such that
			\begin{align}\label{eq:def_vpsip}
				\begin{cases} 
					\vpsi' = \vphi &  \text{ on }\s^{n-1}_1,\\[3mm]
					\displaystyle \dr_r \vpsi' = \vtau - \vtheta &  \text{ on }\s^{n-1}_1, \\[3mm]
					\vpsi' = 0 & \text{ on }\s^{n-1}_2, \\[3mm]
					\dr_r \vpsi' = 0 & \text{ on }\s^{n-1}_2.
				\end{cases} 
			\end{align}
			Moreover, we have the following estimates
			\begin{equation}\label{eq:est_vpsip}
				\left\{
				\begin{aligned}
					& \|\vpsi'\|_{W^{2,n}(\B_2\setminus \B_1)} \leq C\, \|\vphi\|_{W^{2-\frac{1}{n},n}(\s^{n-1}_1)} + C\, \left\| \vtau - \vtheta \right\|_{W^{1-\frac{1}{n},n}(\s^{n-1}_1)}, \\[3mm]
					& \|\vpsi'\|_{W^{2+\frac{1}{n},n}(\B_2\setminus \B_1)}  \leq C \, \|\vphi\|_{W^{2,n}(\s^{n-1}_1)} + C\, \left\| \vtau - \vtheta \right\|_{W^{1,n}(\s^{n-1}_1)} .
				\end{aligned} 
				\right.
			\end{equation}
			By Gagliardo--Nirenberg interpolation inequality, we obtain:
			\begin{align*}
				\|\vphi\|_{W^{2-\frac{1}{n},n}(\s^{n-1})}  \leq C(n)\, \|\vphi\|_{W^{1,n}(\s^{n-1})}^{\frac{1}{n}}\, \|\vphi\|_{W^{2,n}(\s^{n-1})}^{\frac{n-1}{n}} + C(n)\, \|\vphi\|_{L^{\infty}(\s^{n-1})}  \leq C(n,K)\, \eps^{\frac{1}{n}}.
			\end{align*}
			In a similar manner, we have 
			\begin{align*}
				\left\|\vtau - \vtheta \right\|_{W^{1-\frac{1}{n},n}(\s^{n-1}_1)} \leq C(n,K)\, \eps.
			\end{align*}
			Using the Sobolev embedding $W^{2+\frac{1}{n},n}(\B^n_2\setminus \B^n_1)\hookrightarrow C^{1,\frac{1}{2n}}(\B^n_2\setminus \B^n_1)$ and $W^{2,n}(\B^n_2\setminus \B^n_1)\hookrightarrow L^{\infty}(\B^n_2\setminus \B^n_1)$, we obtain 
			\begin{align}\label{eq:est_w}
				\begin{cases} 
					\|\vpsi'\|_{L^{\infty}(\B^n_2\setminus \B^n_1)} + \|\g \vpsi'\|_{L^n(\B^n_2\setminus \B^n_1)} + \|\g^2 \vpsi'\|_{L^n(\B^n_2\setminus \B^n_1)} \leq C(n,K)\, \eps^{\frac{1}{n}},\\[3mm]
					\|\vpsi'\|_{C^{1,\frac{1}{2n}}(\B^n_2\setminus \B^n_1)} + \|\vpsi'\|_{W^{2+\frac{1}{n},n}(\B^n_2\setminus \B^n_1)}  \leq C(n,K).
				\end{cases} 
			\end{align}
			We consider the following map given for $x\in \R^n\setminus \B_1$:
			\begin{align}\label{eq:def_vpsi}
				\vpsi(x) \coloneqq \begin{pmatrix} 
					x \\ 0 
				\end{pmatrix}
				+ \vpsi'(x).
			\end{align}
			By \eqref{eq:est_w}-\eqref{eq:def_vpsi}, there exists $\eps_1>0$ and $r_0>1$ depending only on $n$, $d$ and $V$ such that if $\eps<\eps_1$, then $\vpsi$ can be described on $\B_{r_0}\setminus \B_1$ as a graph of a map $\vw\colon \B_{r_0}\setminus \B_1\to \R^{d-n}$ such that (thanks to \eqref{eq:est_w}):
			\begin{align}\label{eq:est_w2}
				\begin{cases} 
					\|\vw\|_{L^{\infty}(\B_{r_0}\setminus \B_1)} + \|\g \vw\|_{L^n(\B_{r_0}\setminus \B_1)} + \|\g^2 \vw\|_{L^n(\B_{r_0}\setminus \B_1)} \leq C(n,K)\, \eps^{\frac{1}{n}},\\[3mm]
					\|\vw\|_{C^{1,\frac{1}{2n}}(\B_{r_0}\setminus \B_1)} + \|\vw\|_{W^{2+\frac{1}{n},n}(\B_{r_0}\setminus \B_1)}\leq C(n,K).
				\end{cases} 
			\end{align}
			Using the Sobolev injection $W^{\frac{1}{n},n}(\B_{r_0}\setminus \B_1)\hookrightarrow L^{\frac{n^2}{n-1}}(\B_{r_0}\setminus \B_1)$, we obtain
			\begin{align*}
				\|\g^2 \vw\|_{L^{\frac{n^2}{n-1}} (\B_{r_0}\setminus \B_1) } \leq C(n)\, \|\g^2 \vw\|_{W^{\frac{1}{n},n}(\B_{r_0}\setminus \B_1)}\leq C(n,K).
			\end{align*}
			We also have $\frac{n^2}{n-1}>\frac{n^2}{n-\frac{1}{2}}>n$ with the equality
			\begin{align*}
				\frac{n-\frac{1}{2}}{n^2} = \frac{1}{2}\cdot \frac{1}{n} + \frac{1}{2}\cdot \frac{n-1}{n^2}.
			\end{align*}
			We obtain the following estimate
			\begin{align}
				\|\g^2\vw\|_{L^{(n,1)} (\B_{r_0}\setminus \B_1) } & \leq C(n) \|\g^2 \vw\|_{L^{\frac{n^2}{n-\frac{1}{2}}} (\B_{r_0}\setminus \B_1) }  \nonumber  \\[3mm]
				& \leq C(n) \|\g^2 \vw\|_{L^{\frac{n^2}{n-1}} (\B_{r_0}\setminus \B_1) }^{\frac{1}{2}}\, \|\g^2 \vw\|_{L^n (\B_{r_0}\setminus \B_1) }^{\frac{1}{2}}  \nonumber  \\[1mm]
				& \leq C(n,K)\, \eps^{\frac{1}{2}}. \label{eq:Lorentz}
			\end{align}
			We consider a cut-off function $\chi\in C^{\infty}([1,+\infty);[0,1])$ such that 
			\begin{align}\label{eq:def_cutoff}
				\begin{cases}
					\chi = 1 & \text{in } \left[1,\frac{1+r_0}{2} \right],\\[2mm]
					\chi = 0 & \text{in } [2,+\infty),\\[2mm]
					|\g \chi| + |\g^2 \chi| \leq C & \text{in } \left[ \frac{1+r_0}{2} ,2 \right].
				\end{cases}
			\end{align}
			We paste $\vpsi$ with a flat $n$-plane. We define the immersion
			\begin{align}\label{eq:def_vPsi}
				\forall x\in \R^n\setminus \B_1,\qquad \vPsi(x) \coloneqq  \big(1-\chi(|x|)\big)\, \begin{pmatrix}
					x \\ 0
				\end{pmatrix} + \chi(|x|)\, \vpsi(x).
			\end{align}
			
			\emph{Step 2: Second fundamental form of $\vpsi$ and $\vPsi$.}\\
			In order to show that $\vPsi$ is a weak immersion, we need to check that its second fundamental is well-defined in $L^n$ across the junction. By \eqref{eq:Lorentz} and \cite[Lemma 2.2]{breuning2015}, we have that $\vPsi$ has a second fundamental form in $L^{(n,1)}(\R^n\setminus \B_1)$. Thanks to \eqref{eq:est_w}-\eqref{eq:def_vpsi}, we obtain that in radial coordinates $(r,\theta)\in[1,+\infty)\times\s^{n-1}$, the maps $\vpsi(r,\cdot)$ converge strongly to $\vPhi$ as $r\to 1$ in the $C^{1,\frac{1}{2n}}(\s^{n-1})$-topology. We obtain that 
			\begin{align}\label{eq:C1_extension}
				\vPsi(r,\cdot) \xrightarrow[r\to 1]{C^{1,\frac{1}{2n}}} \vPhi.
			\end{align}
			We now show that the Gauss maps also coincides, this will imply that the second fundamental form of $\vPsi$ is indeed an element of $L^n$ across the junction.
			\begin{claim}\label{cl:conv_normal}
				It holds
				\begin{align*}
					\vn_{\vPsi}(r,\cdot) \xrightarrow[r\to 1]{C^{0,\frac{1}{2n}}} \vn_{\vPhi}\Big|_{\s^{n-1}}.
				\end{align*}
			\end{claim}
			\begin{proof} 
				By \eqref{eq:def_vpsi} and \eqref{eq:def_vpsip}, we have the following convergence:
				\begin{align*}
					\vn_{\vpsi}(r,\cdot) = *_{\R^d} \frac{\dr_r\vpsi \wedge \dr_{\theta_1}\vpsi \wedge \cdots \wedge \dr_{\theta_{n-1}} \vpsi }{ \left| \dr_r \vpsi \wedge \dr_{\theta_1} \vpsi \wedge \cdots \wedge \dr_{\theta_{n-1}} \vpsi \right|} \xrightarrow[r\to 1]{C^{0,\frac{1}{2n}}} *_{\R^d} \frac{\vtau \wedge \dr_{\theta_1} \vPhi\wedge \cdots \wedge \dr_{\theta_{n-1}} \vPhi }{ \left| \vtau \wedge \dr_{\theta_1} \vPhi \wedge \cdots \wedge \dr_{\theta_{n-1}} \vPhi \right| }.
				\end{align*}
				The tangent space of $\vPhi_{|\s^{n-1}}$ is spanned by $\dr_{\theta_1}\vPhi,\ldots,\dr_{\theta_{n-1}} \vPhi$. To obtain the complete tangent space of $\vPhi(\Sigma)$, one needs to add the direction 
				\begin{align*} 
					\vtau \coloneq \vPhi_*\left(\g^{g_{\vPhi}} |\vPhi|\right) = \proj_{T\vPhi(\Sigma)}\left(\frac{\vPhi}{|\vPhi|}\right) = \proj_{T\vPhi(\Sigma)}\left( \vPhi \right) .
				\end{align*} 
				Indeed, we have $\vtau\in T\vPhi(\Sigma)$ and for any $i\in\{1,\ldots,n\}$,
				\begin{align*}
					\vtau\cdot \dr_{\theta_i} \vPhi = \vPhi\cdot \dr_{\theta_i}\vPhi = \dr_{\theta_i}\left(\frac{|\vPhi|^2}{2}\right) =0.
				\end{align*}
				Thus, the vector field $\vtau$ is a tangent to $\vPhi(\Sigma)$ not being one of the $\dr_{\theta_i}\vPhi$. Hence, the limit in the above expression is exactly $\vn_{\vPhi}$. Since $\vn_{\vpsi} = \vn_{\vPsi}$ for $r<\frac{1+r_0}{2}$, we obtain
				\begin{align*}
					\vn_{\vPsi}(r,\cdot) \xrightarrow[r\to 1]{C^{0,\frac{1}{2n}}} \vn_{\vPhi}\Big|_{\s^{n-1}}.
				\end{align*}
			\end{proof} 
			
			By \eqref{eq:C1_extension} and Claim \ref{cl:conv_normal}, we can can extend $\vPhi$ by $\vPsi$ on $\R^n\setminus \B_1$ as a weak immersion with second fundamental form in $L^n$. Moreover, $\vPsi$ is the graph of the function $\vu$ (by definition) given by
			\begin{align}\label{eq:def_ust}
				\vu = \chi\, \vw.
			\end{align}
			We have
			\begin{align*}
				\big|\vII_{\vPsi} \big|_{g_{\vPsi}} & \leq C\, |\vw|\, |\g^2\, \chi|_{g_{eucl}} + C\, \chi\, |\g^2 \vw|_{g_{eucl}} + C\, |\g \chi|_{g_{eucl}}\, |\g \vw|_{g_{eucl}} \\[2mm]
				& \leq C\left( \mathbf{1}_{\R^n\setminus \B_1} |\vw|  + \mathbf{1}_{\R^n\setminus \B_1} |\g \vw|_{g_{eucl}} + \chi\, |\g^2 \vw|_{g_{eucl}} \right).
			\end{align*}
			By \eqref{eq:est_w2} and \eqref{eq:Lorentz}, we obtain
			\begin{align*}
				\|\vII_{\vPsi}\|_{L^{(n,2)}(\R^n\setminus \B_1)} \leq C\, \eps^{\frac{1}{n}}.
			\end{align*}
			In $\Sigma$, we have $\vPsi=\vPhi$ and $\vII_{\vPhi} \in L^{(n,2)}(\Sigma)$.
		\end{proof}

        \begin{remark}
            Instead of extending the immersion $\vPhi\colon \B^n\to \R^d$ to $\R^n\setminus \B^n$, one could also extend an immersion $\vPhi\colon \R^n\setminus \B^n\to\R^d$ to $\B^n$ with minor modifications. We consider $\vPhi\in \I_{0,(n,2)}(\Sigma;\R^d)$ satisfying the assumptions 2 to 4 in Lemma \ref{lm:extension}, but $\Sigma$ is now a non-compact manifold with $\dr\Sigma=\s^{n-1}$. We then modify the definition of $\vpsi'$ to be a map $\vpsi"\colon \B^n_1\setminus \B^n_{1/2}\to \R^d$ such that
            \begin{align*}
				\begin{cases} 
					\vpsi" = \vphi &  \text{ on }\s^{n-1}_1,\\[3mm]
					\displaystyle \dr_r \vpsi" = \vtau - \vtheta &  \text{ on }\s^{n-1}_1, \\[3mm]
					\vpsi" = 0 & \text{ on }\s^{n-1}_{1/2}, \\[3mm]
					\dr_r \vpsi" = 0 & \text{ on }\s^{n-1}_{1/2}.
				\end{cases} 
			\end{align*}
            Defining $\vpsi"$ as in the above proof (using a right-inverse of the trace operator), we recover estimates similar to \eqref{eq:est_w} in $\B^n_1\setminus \B^n_{1/2}$ instead of $\B^n_2\setminus \B^n_1$. Plugging these minor changes in the rest of the proof, we obtain a weak immersion $\vPsi\in \I_{0,(n,2)}(\hat{\Sigma};\R^d)$ still verifying \eqref{eq:est_extension} with $\hat{\Sigma}$ being the manifold obtained by gluing $\Sigma$ and $\B^n$ along their common boundary $\s^{n-1}$.
        \end{remark}
		
		We can now conclude the proof of \Cref{th:local_imm_v2}.
		
		\begin{proof}[Proof of \Cref{th:local_imm_v2}]
			We obtain the extension $\vPsi\colon \R^n\to \R^d$ by direct application of Lemma \ref{lm:extension}. Combining now \Cref{th:Sobolev} and \eqref{eq:Riem_imm_Lorentz}, we obtain 
			\begin{align*}
				\left\| \Riem^{g_{\vPhi}} \right\|_{L^{\left(\frac{n}{2},1\right)}(\Sigma,g_{\vPhi})} \leq C(n,d,K)\, \eps^{\frac{2}{n}}.
			\end{align*}
			For $\eps<\eps_4=\eps_4(n,d,K)$ small enough, we can apply \Cref{th:Weak_metrics_harm}. There exists $z\in \Diff(\R^n)$ such that for all $1\leq i,j\leq n$, it holds
			\begin{align*}
				\left\| \big( g_{\vPsi\circ z} \big)_{ij} - \delta_{ij} \right\|_{\bar{W}^{2,\left(\frac{n}{2},1\right)}(\R^n,\geu)} \leq C(n,d,K)\, \eps^{\frac{2}{n}}.
			\end{align*}
			We obtain \Cref{th:local_imm_v2} by restricting the domain to $z^{-1}(\B)$, where we have $\vPsi\circ z = \vPhi\circ z$:
			\begin{align*}
				\left\| \big( g_{\vPhi\circ z} \big)_{ij} - \delta_{ij} \right\|_{\bar{W}^{2,\left(\frac{n}{2},1\right)}(z^{-1}(\B),\geu)} \leq C(n,d,K)\, \eps^{\frac{2}{n}}.
			\end{align*}
		\end{proof}
		
	\bibliography{sobolovBib.bib}
	\bibliographystyle{plain}
\end{document}